\documentclass[11pt]{article} 

\usepackage{amsmath,amssymb,amsthm}
\usepackage[utf8]{inputenc}
\usepackage{epic}
\usepackage{eepic}
\usepackage{rotating}
\usepackage[affil-it]{authblk}

\usepackage{geometry}
\numberwithin{equation}{section}

\usepackage{graphicx}
\usepackage{graphics}
\usepackage[font=sf, labelfont={sf,bf}, margin=1.2cm]{caption}

\makeatletter
\newcommand{\oset}[2]{%
  {\mathop{#2}\limits^{\vbox to -.5\ex@{\kern-\tw@\ex@
   \hbox{\scriptsize #1}\vss}}}}
\makeatother

\newtheorem{theorem}{Theorem}[section]

\newtheorem{definition}[theorem]{Definition}
\newtheorem{proposition}[theorem]{Proposition}
\newtheorem{corollary}[theorem]{Corollary}
\newtheorem{lemma}[theorem]{Lemma}

\usepackage{booktabs} 
\usepackage{array} 
\usepackage{paralist} 
\usepackage{verbatim} 
\usepackage{subfig} 

\usepackage{fancyhdr} 
\usepackage{sectsty}
\allsectionsfont{\sffamily\mdseries\upshape} 

\usepackage[nottoc,notlof,notlot]{tocbibind}
\usepackage[titles,subfigure]{tocloft}

\title{Kakeya-type sets over Cantor sets of directions in $\mathbb{R}^{d+1}$}
\author{\textsc{Edward Kroc and Malabika Pramanik}}
\date{} 

\begin{document}
\maketitle
{\allowdisplaybreaks 

\begin{abstract}
Given a Cantor-type subset $\Omega$ of a smooth curve in $\mathbb R^{d+1}$, we construct examples of sets that contain unit line segments with directions from $\Omega$ and exhibit analytical features similar to those of classical Kakeya sets of arbitrarily small $(d+1)$-dimensional Lebesgue measure. The construction is based on probabilistic methods relying on the tree structure of $\Omega$, and extends to higher dimensions an analogous planar result of Bateman and Katz \cite{BatemanKatz}. In particular, the existence of such sets implies that the directional maximal operator associated with the direction set $\Omega$ is unbounded on $L^p(\mathbb{R}^{d+1})$ for all $1\leq p<\infty$.  
\end{abstract}
\renewcommand{\thefootnote}{\fnsymbol{footnote}} 
\renewcommand{\thefootnote}{\arabic{footnote}}
\tableofcontents

\section{Introduction}

\subsection{Background}\label{background}

A Kakeya set (also called a Besicovitch set) in $\mathbb{R}^{d+1}$ is a set that contains a unit line segment in every direction.  The study of such sets spans approximately a hundred years.  The first major analytical result in this area, due to Besicovitch~\cite{Besicovitch}, shows that there exist Kakeya sets with Lebesgue measure zero.  Over the past forty-plus years, dating back at least to the work of Fefferman~\cite{Fefferman}, the study of Kakeya sets has been a simultaneously fruitful and vexing endeavor. On one hand its applications have been found in many deep and diverse corners of analysis, PDEs, additive combinatorics and number theory.  On the other hand, certain fundamental questions concerning the size and dimensionality of such sets have eluded complete resolution.  

In order to obtain quantitative estimates for analytical purposes, it is often convenient to work with the $\delta$-neighborhood of a Kakeya set, rather than the set itself. Here $\delta$ is an arbitrarily small positive constant. The $\delta$-neighborhood of a Kakeya set is therefore an object that consists of many thin $\delta$-tubes. A $\delta$-tube is by definition a cylinder of unit axial length and spherical cross-section of radius $\delta$. The defining property of a zero measure Kakeya set dictates that the volume of its $\delta$-neighborhood goes to zero as $\delta \rightarrow 0$, while the sum total of the sizes of these tubes is roughly a positive absolute constant. Indeed, a common construction of thin Kakeya sets in the plane (see for example \cite[Chapter 10]{SteinHA}) relies on the following fact: given any $\epsilon > 0$, there exists an integer $N \geq 1$ and a collection of $2^{-N}$-tubes, i.e., a family of $1 \times 2^{-N}$ rectangles, $\{ P_t : 1 \leq t \leq 2^{N}\}$ in $\mathbb R^{2}$ such that 
\begin{equation}  \label{2d Kakeya} \bigl| \bigcup_t P_t \bigr| < \epsilon, \qquad \text{ and } \qquad \sum_{t} |\widetilde{P}_t| = 1. \end{equation}   
Here $|\cdot|$ denotes Lebesgue measure (in this case two-dimensional), and $\widetilde{P}_t$ denotes the ``reach" of the tube $P_t$, namely the tube obtained by translating $P_t$ by two units in the positive direction along its axis. While it is not known that every Kakeya set in two or higher dimensions shares a similar feature, the ones that do have found repeated applications in analysis. Fundamental results have relied on the existence of such sets, for example the lack of differentiation for integral averages over parallelepipeds of arbitrary orientation, and the counterexample of the ball multiplier \cite[Chapter 10]{SteinHA}. The property described above continues to be the motivation for the Kakeya-type sets that we will study in the present paper. 
\begin{definition}\label{Kakeya-type set}
For $d \geq 1$, we define a set of directions $\Omega$ to be a compact subset of $\mathbb R^{d+1}$.  We say that a tube in $\mathbb R^{d+1}$ has orientation $\omega \in \Omega$ or a tube is oriented in direction $\omega$ if its axis is parallel to $\omega$.  We say that $\Omega$ {\em{admits Kakeya-type sets}} if one can find a constant $C_0 \geq 1$ such that for any $N\geq 1$, there exists $\delta_N > 0$, $\delta_N \rightarrow 0$ as $N \rightarrow \infty$ and a collection of $\delta_N$-tubes $\{P_t^{(N)}\} \subseteq \mathbb R^{d+1}$ with orientations in $\Omega$ with the following property: 
\begin{equation}\label{Kakeya-type condition}
	\text{if} \quad E_N := \bigcup_t P_t^{(N)},\quad E_N^*(C_0) := \bigcup_t C_0P_t^{(N)},\quad\text{then}\quad \lim_{N\rightarrow\infty}\frac {|E^*_N(C_0)|}{|E_N|} = \infty.
\end{equation}
Here $|\cdot|$ denotes $(d+1)$-dimensional Lebesgue measure, and $C_0P_t^{(N)}$ denotes the tube with the same centre, orientation and cross-sectional radius as $P_t^{(N)}$, but $C_0$ times its length. We will refer to $\{ E_N : N \geq 1\}$ as sets of Kakeya-type. 
\end{definition}
\noindent Specifically in this paper, we will be concerned with certain subsets of a curve, either on the sphere $\mathbb S^d$, or equivalently on a hyperplane at unit distance from the origin, that admit Kakeya-type sets.

Kakeya and Kakeya-type sets of zero measure have intrinsic structural properties that continually prove useful in an analytical setting.  The most important of these properties is arguably the so-called \textit{stickiness} property, originally observed by Wolff~\cite{Wolff}.  Roughly speaking, if a Kakeya-type set is a collection of many overlapping line segments, then stickiness dictates that the map which sends a direction to the line segment in the set with that direction is almost Lipschitz, with respect to suitably defined metrics. Another way of expressing this is that if the origins of two overlapping $\delta$-tubes are positioned close together, then the angle between these thickened line segments must be small, resulting in the intersection taking place far away from the respective bases.  This idea, which has been formalized in several different ways in the literature \cite{Wolff}, \cite{KatzTao}, \cite{KolasaWolff}, \cite{KatzLabaTao}, will play a central role in our results, as we will discuss in Section~\ref{stickiness section}.

Geometric and analytic properties of Kakeya and Kakeya-type sets are often studied using a suitably chosen maximal operator. Conversely, certain blow-up behavior for such operators typically follow from the existence of such sets.  We introduce two such well-studied operators for which the existence of Kakeya-type sets implies unboundedness.  

Given a set of directions $\Omega$, consider the directional maximal operator $D_{\Omega}$ defined by
\begin{equation}\label{max op}
	D_{\Omega}f(x) := \sup_{\omega\in\Omega}\sup_{h>0}\frac 1{2h}\int_{-h}^{h} |f(x+\omega t)|dt,
\end{equation}
where $f:\mathbb{R}^{d+1}\rightarrow\mathbb{C}$ is a function that is locally integrable along lines.  Also, for any locally integrable function $f$ on $\mathbb R^{d+1}$, consider the Kakeya-Nikodym maximal operator $M_{\Omega}$ defined by
\begin{equation}\label{Kakeya max op}
	M_{\Omega}f(x) := \sup_{\omega\in\Omega}\sup_{\substack{P\ni x\\ P\parallel\omega}} \frac 1{|P|} \int_P |f(y)|dy,
\end{equation}
where the inner supremum is taken over all cylindrical tubes $P$ containing the point $x$, oriented in the direction $\omega$.  The tubes are taken to be of arbitrary length $l$ and have circular cross-section of arbitrary radius $r$, with $r\leq l$.  If $\Omega$ is a set with nonempty interior, then due to the existence of Kakeya sets with $(d+1)$-dimensional Lebesgue measure zero~\cite{Besicovitch}, $D_{\Omega}$ and $M_{\Omega}$ are both unbounded as operators on $L^p(\mathbb{R}^{d+1})$ for all $1 \leq p<\infty$.  More generally, if $\Omega$ admits Kakeya-type sets, then these operators are unbounded on $L^p(\mathbb{R}^{d+1})$ for all $1 \leq p<\infty$ (see Section~\ref{results} below).  

The complementary case when $\Omega$ has empty interior has been studied extensively in the literature.  It is easy to see that the operators in \eqref{max op} and \eqref{Kakeya max op} exhibit a kind of monotonicity: if $\Omega\subset\Omega'$, then $D_{\Omega}f(x)\leq D_{\Omega'}f(x)$ and $M_{\Omega}f(x)\leq M_{\Omega'}f(x)$, for any suitable function $f$.  Since these operators are unbounded when $\Omega'= \text{ the unit sphere } \mathbb S^d$, treatment of the positive direction -- identifying ``small" sets of directions $\Omega$ for which these operators are bounded on some $L^p$ -- has garnered much attention \cite{NagelSteinWainger, Carbery, SjogrenSjolin, AlfonsecaSoriaVargas, Alfonseca, ParcetRogers}.  These types of results rely on classical techniques in $L^p$-theory, such as square function estimates, Littlewood-Paley theory and almost-orthogonality principles.

For a general dimension $d\geq 1$, Nagel, Stein and Wainger \cite{NagelSteinWainger} showed that $D_{\Omega}$ is bounded on all $L^p(\mathbb R^{d+1})$, $1<p\leq\infty$, when $\Omega = \{(v_i^{a_1},\ldots,v_i^{a_{d+1}}) : i \geq 1\}$. Here $0<a_1<\cdots<a_{d+1}$ are fixed constants, and $\{ v_i : i \geq 1\}$ is a sequence obeying $0<v_{i+1}\leq\lambda v_i$ for some lacunary constant $0<\lambda<1$.  Carbery \cite{Carbery} showed that $D_{\Omega}$ is bounded on all $L^p(\mathbb R^{d+1})$, $1<p\leq\infty$, in the special case when $\Omega$ is the set given by the $(d+1)$-fold Cartesian product of a geometric sequence, namely  $\Omega = \{(r^{k_1},\ldots,r^{k_{d+1}}) : k_1,\ldots,k_{d+1}\in\mathbb{Z}^+\}$ for some $0<r<1$.  Very recently, Parcet and Rogers \cite{ParcetRogers} generalized an almost-orthogonality result of Alfonseca \cite{Alfonseca} to extend the boundedness of $D_{\Omega}$ on all $L^p(\mathbb R^{d+1})$, $1<p\leq\infty$, for sets $\Omega$ that are lacunary of finite order, defined in a suitable sense.  Building upon previous work of Alfonseca, Soria, and Vargas~\cite{AlfonsecaSoriaVargas}, Sj\"ogren and Sj\"olin \cite{SjogrenSjolin} and Nagel, Stein and Wainger~\cite{NagelSteinWainger}, the recent result of Parcet and Rogers \cite{ParcetRogers} recovers those of its predecessors. 

Aside from this set of positive results with increasingly weak hypotheses, there has also been much development in the negative direction, pioneered by Bateman, Katz and Vargas \cite{Vargas, Katz, BatemanKatz, Bateman}.  
Of special significance to this article is the work of Bateman and Katz~\cite{BatemanKatz},  where the authors establish that $D_{\Omega}$ is unbounded in $L^p(\mathbb R^2)$ for all $1 \leq p < \infty$ if $\Omega = \{(\cos \theta, \sin \theta): \theta \in \mathcal C_{1/3} \}$, where $\mathcal C_{1/3}$ is the Cantor middle-third set.  A crowning success of the methodology of~\cite{BatemanKatz} combined with the aforementioned work in the positive direction (in particular \cite{Alfonseca}) is a result by Bateman~\cite{Bateman} that gives a complete characterization of the $L^p$-boundedness of $D_{\Omega}$ and $M_{\Omega}$ in the plane, while also describing all direction sets $\Omega$ that admit planar sets of Kakeya-type. The distinctive feature  of this latter body of work \cite{BatemanKatz, Bateman} dealing with the negative point of view is the construction of counterexamples using a random mechanism that exploits the property of stickiness.  We too adopt this approach to construct Kakeya-type sets in $\mathbb R^{d+1}$, $d \geq 2$ consisting of tubes whose orientations lie along certain subsets of a curve on the hyperplane $\{1\}\times\mathbb{R}^d$.

%%%%%%%%%%%%%%%%%%%%%%%%%%%%%%%%%%%%%%%%%%%%%%%%%%%%%%%%%%%%%%%%%%%%%%%

\subsection{Results}\label{results}

As mentioned above, Bateman and Katz~\cite{BatemanKatz} establish the unboundedness of $D_{\Omega}$ and $M_{\Omega}$ on $L^p(\mathbb{R}^2)$, for all $p \in [1,\infty)$, when $\Omega = \{(\cos \theta, \sin \theta) : \theta \in \mathcal C_{1/3} \}$ by constructing suitable Kakeya-type sets in the plane.  In this paper, we extend their result to the general $(d+1)$-dimensional setting.  To this end, we first describe what we mean by a Cantor set of directions in $(d+1)$ dimensions.

Fix some integer $M\geq 3$.  Construct an arbitrary Cantor-type subset of $[0,1)$ as follows.  
\begin{enumerate}
	\item[$\bullet$] Partition $[0,1]$ into $M$ subintervals of the form $[a,b]$, all of equal length $M^{-1}$.  Among these $M$ subintervals, choose any two that are not adjacent (i.e., do not share a common endpoint); define $\mathcal{C}_M^{[1]}$ to be the union of these chosen subintervals, called first stage basic intervals.  
	\item[$\bullet$] Partition each first stage basic interval into $M$ further (second stage) subintervals of the form $[a,b]$, all of equal length $M^{-2}$.  Choose two non-adjacent second stage subintervals from each first stage basic one, and define $\mathcal{C}_M^{[2]}$ to be the union of the four chosen second stage (basic) intervals.
	\item[$\bullet$] Repeat this procedure \textit{ad infinitum}, obtaining a nested, non-increasing sequence of sets. Denote the limiting set by $\mathcal{C}_M$: $$\mathcal{C}_M = \bigcap_{k=1}^{\infty} \mathcal{C}_M^{[k]}.$$ We call $\mathcal{C}_M$ a \textit{generalized Cantor-type set} (with base $M$).
\end{enumerate}
While conventional uniform Cantor sets, such as the Cantor middle-third set, are special cases of generalized Cantor-type sets, the latter may not in general look like the former. In particular, sets of the form $\mathcal C_M$ need not be self-similar, although the actual sequential selection criterion leading up to their definition will be largely irrelevant for the content of this article. It is well-known (see \cite [Chapter 4]{Falconer}) that such sets have Hausdorff dimension at most $\log 2/\log M$. By choosing $M$ large enough, we can thus construct generalized Cantor-type sets of arbitrarily small dimension.

In this paper, we prove the following.
\begin{theorem}\label{main theorem}
	Let $\mathcal{C}_M\subset[0,1]$ be a generalized Cantor-type set described above.  Let $\gamma : [0,1]\rightarrow \{1\} \times [-1,1]^d$ be an injective map that satisfies a bi-Lipschitz condition
\begin{equation}\label{Lipschitz}
	\forall\ x,y,\ c|x-y| \leq |\gamma(x)-\gamma(y)| \leq C|x-y|,
\end{equation}
for some absolute constants $0 < c < 1 < C < \infty$.  Set $\Omega = \{\gamma(t) : t\in\mathcal{C}_M\}$.  Then
\begin{enumerate}[(a)]
\item the set $\Omega$ admits Kakeya-type sets,
\item the operators $D_{\Omega}$ and $M_{\Omega}$ are unbounded on $L^p(\mathbb R^{d+1})$ for all $1 \leq p < \infty$. 
\end{enumerate}
\end{theorem}
The condition in Theorem~\ref{main theorem} that $\gamma$ satisfies a bi-Lipschitz condition can be weakened, but it will help in establishing some relevant geometry.  Throughout this exposition, it is instructive to envision $\gamma$ as a smooth curve on the plane $x_1=1$, and we recommend the reader does this to aid in visualization.  Our underlying direction set of interest $\Omega = \gamma(\mathcal{C}_M)$ is essentially a Cantor-type subset of this curve. 

The main focus of this article, for reasons explained below, is on (a), not on (b). Indeed, the implication (a) $\implies$ (b) is well-known in the literature; if $f = \mathbf{1}_{E_N}$, where $E_N$ is as in \eqref{Kakeya-type condition}, then there exists a constant $c_0 = c_0(d, C_0) > 0 $ such that \begin{equation} \label{E_N and D_Omega} \min \bigl[ D_{\Omega}f(x), M_{\Omega}f(x) \bigr] \geq c_0 \quad \text{ for } x \in E_N^{\ast}(C_0).  \end{equation}  
This shows that \[\min \bigl[ ||D_{\Omega}||_{p \rightarrow p}, ||M_{\Omega}||_{p \rightarrow p} \bigr] \geq c_0 \left(\frac{|E_N^{\ast}(C_0)|}{|E_N|} \right)^{\frac{1}{p}}, \quad \text{ which } \rightarrow \infty \text{ if } 1 \leq p < \infty. \] 
On the other hand, the condition (a) of Theorem~\ref{main theorem} is not \textit{a priori} strictly necessary in order to establish part (b) of the theorem. Suppose that $\{G_N : N \geq 1\}$ and $\{ \widetilde{G}_N : N \geq 1\}$ are two collections of sets with $|\widetilde{G}_N|/|G_N| \rightarrow \infty$, enjoying the additional property that for any point $x \in \widetilde{G}_N$, there exists a finite line segment originating at $x$ and pointing in a direction of $\Omega$, which spends at least a fixed positive proportion of its length in $G_N$. By an easy adaptation of the argument in \eqref{E_N and D_Omega}, the sequence of test functions $f_N = 1_{G_N}$ would then prove the claimed unboundedness of $D_{\Omega}$. Kakeya-type sets, if they exist, furnish one such family of test functions with $G_N = E_N$ and $\widetilde{G}_N = E_N^{\ast}$. 

In~\cite{ParcetRogers}, Parcet and Rogers construct, for certain examples of direction sets, families of sets $G_N$ that supply a different class of test functions sufficient to prove unboundedness of the associated directional maximal operators. Similar constructions could in principle be applied in our situation as well to establish the unboundedness of directional maximal operators associated with our sets of interest.  However, a set as constructed in \cite{ParcetRogers} is typically a Cartesian product of a planar Kakeya-type set with a cube, and as such not of Kakeya-type according to Definition \ref{Kakeya-type set}.  In particular, it consists of rectangular parallelepipeds with possibly different sidelengths, with these sides not necessarily pointing in a direction from the underlying direction set $\Omega$, although there are line segments with orientation from $\Omega$ contained within them. Further, in contrast with Definition \ref{Kakeya-type set}, $\widetilde{G}_N$ need not be obtained by translating $G_N$ along its longest side.    

The reason for considering Kakeya-type sets in this paper is twofold. First, they appear as natural generalizations of a classical feature of planar Kakeya set constructions, as explained in \eqref{2d Kakeya}. Studying higher dimensional extensions of this phenomenon is of interest in its own right, and this article provides a concrete illustration of a sparse set of directions that gives rise to a similar phenomenon. Perhaps more importantly, we use the special direction sets  in this paper as a device for introducing certain machinery whose scope reaches beyond these examples. As discussed in Section \ref{background}, the problem of determining a characterization of direction sets $\Omega$ that give rise to $L^p$-bounded maximal operators $D_{\Omega}$ and $M_{\Omega}$ has garnered much attention. In \cite{ParcetRogers}, Parcet and Rogers obtain a positive result for such operators to be bounded on all Lebesgue spaces in general dimensions under certain hypotheses involving lacunarity, and conjecture that their condition is essentially sharp. The counterexamples in \cite{ParcetRogers} mentioned above were furnished as supporting evidence for this claim. We address this conjecture in \cite{KrocPramanik}. The property of admittance of Kakeya-type sets in the sense of Definition \ref{Kakeya-type set} turns out to be a critical feature of this study, and indeed equivalent to the unboundedness of directional maximal operators. In addition to the framework introduced in \cite{BatemanKatz}, the methods developed in the present article, specifically the investigation of root configurations and slope probabilities in Sections \ref{probability estimation section} and \ref{lowerboundsection} are central to the analysis in \cite{KrocPramanik}. While the consideration of general direction sets in \cite{KrocPramanik} necessarily involves substantial technical adjustments, many of the main ideas of that analysis can be conveyed in the simpler setting of the Cantor example that we treat here.  As such, we recommend that the reader approach the current paper as a natural first step in the process of understanding properties of direction sets that give rise to unbounded directional and Kakeya-Nikodym maximal operators on $L^p(\mathbb{R}^{d+1})$.

\subsection*{Acknowledgements}

The second author would like to thank Gordon Slade of the Department of Mathematics at the University of British Columbia for a helpful discussion on percolation theory. The research was partially supported by an NSERC Discovery Grant.  

%%%%%%%%%%%%%%%%%%%%%%%%%%%%%%%%%%%%%%%%%%%%%%%%%%%%%%%%%%%%%%%%%%%%%%%

\section{Overview of the proof of Theorem~\ref{main theorem}}

%%%%%%%%%%%%%%%%%%%%%%%%%%%%%%%%%%%%%%%%%%%%%%%%%%%%%%%%%%%%%%%%%%%%%%%

\subsection{Steps of the proof and layout} 

The basic structure of the proof is modeled on \cite{BatemanKatz}, with some important distinctions that we point out below. Our goal is to construct a family of tubes rooted on the hyperplane $\{0\} \times [0,1)^d$, the union of which will eventually give rise to the Kakeya-type set. The slopes of the constituent tubes will be assigned from $\Omega$ via a random mechanism involving stickiness akin to the one developed by Bateman and Katz \cite{BatemanKatz}. The description of this random mechanism is in Section~\ref{stickiness section}, with the required geometric and probabilistic background collected en route in Sections~\ref{geometry section}, \ref{tree section} and \ref{percolation section}. The essential elements of the construction, barring the details of the slope assignment, have been laid out in Section \ref{preliminary construction subsection} below. The main estimates leading to the proof of Theorem~\ref{main theorem} are (\ref{generic lower bound}) and (\ref{generic upper bound}) in Proposition~\ref{generic upper and lower bounds} in this section.  Of these the first, a precise version of which is available in Proposition \ref{Kakeya on average}, provides a lower bound of $a_N= \sqrt{\log N}/N$ on the size of the part of the tubes lying near the root hyperplane. The second inequality, also quantified in  Proposition \ref{Kakeya on average}, yields an upper bound of $b_N = 1/N$ for the portion away from it. The disparity in the relative sizes of these two parts is the desired conclusion of Theorem \ref{main theorem}

The language of trees was a key element in the random construction of \cite{Bateman, BatemanKatz}. We continue to adopt this language, introducing the relevant definitions in Section \ref{tree section} and providing some detail on the connection between the geometry of $\Omega$ and a tree encoding it.  Specifically, the notion of Bernoulli percolation on trees plays an important role in the proof of (\ref{generic upper bound}) with $b_N = 1/N$, as it did in the two-dimensional setting. The higher-dimensional structure of $\Omega$ does however result in minor changes to the argument, and the general percolation-theoretic facts necessary for handling (\ref{generic upper bound}) have been compiled in Section \ref{percolation section}. Other probabilistic estimates specific to the random mechanism of Section \ref{stickiness section} and central to the derivation of \eqref{generic lower bound} are separately treated in Section \ref{probability estimation section}. The proof is completed in Sections \ref{lowerboundsection} and \ref{upperboundsection}.

Of the two estimates \eqref{generic lower bound} and \eqref{generic upper bound} necessary for the Kakeya-type construction, the first is the most significant contribution of this paper. A deterministic analogue of (\ref{generic lower bound}) was used in \cite{Bateman, BatemanKatz}, where a similar lower bound for the size of the Kakeya-type set was obtained for {\em{every}} slope assignment $\sigma$ in a certain measure space. The counting argument that led to this bound fails to produce the necessary estimate in higher dimensions and is replaced here by a probabilistic statement that suffices for our purposes.  More precisely, the issue is the following. A large lower bound on a union of tubes follows if they do not have significant pairwise overlap among themselves; i.e. if the total size of pairwise intersections is small. In dimension two, a good upper bound on the size of this intersection was available uniformly in every sticky slope assignment. Although the argument that provided this bound is not transferable to general dimensions, it is still possible to obtain the desired bound with large probability. A probabilistic statement similar to but not as strong as \eqref{generic lower bound} can be derived relatively easily via an estimate on the first moment of the total size of random pairwise intersections. Unfortunately, this is still not sharp enough to yield the disparity in the sizes of the tubes and their translated counterparts necessary to claim the existence of a Kakeya-type set. To strengthen the bound, we need a second moment estimate on the pairwise intersections. Both moment estimates share some common features; for instance, 
\begin{enumerate}[-]
\item Euclidean distance relations between roots and slopes of two intersecting tubes,
\item interplay of the above with the relative positions of the roots and slopes within the respective trees that they live in, which affects the slope assignments.  
\end{enumerate}  
However, the technicalities are far greater for the second moment compared to the first. In particular, for the second moment we are naturally led to consider not just pairs, but triples and quadruples of tubes, and need to evaluate the probability of obtaining pairwise intersections among these. Not surprisingly, this probability depends on the structure of the root tuple within its ambient tree. It is the classification of these root configurations, computation of the relevant probabilities and their subsequent application to the estimation of expected intersections that we wish to highlight as the main contributions of this article.

%%%%%%%%%%%%%%%%%%%%%%%%%%%%%%%%%%%%%%%%%%%%%%%%%%%%%%%%%%%%%%%%%%%%%%%

\subsection{Construction of a Kakeya-type set} \label{preliminary construction subsection}

We now choose some integer $M\geq 3$ and a generalized Cantor-type set $\mathcal{C}_M\subseteq[0,1)$ as described in Section \ref{results}, and fix these items for the remainder of the article.  We also fix an injective map $\gamma : [0,1]\rightarrow \{1\} \times [-1,1]^d$ satisfying the bi-Lipschitz condition in \eqref{Lipschitz}.  These objects then define a fixed set of directions $\Omega = \{\gamma(t) : t\in\mathcal{C}_M\} \subseteq \{1\} \times [-1,1]^d$.  

Next, we define the thin tube-like objects that will comprise our Kakeya-type set.  Fix an arbitrarily large integer $N\geq 1$, typically much bigger than $M$.  Let $\{Q_t : t\in\mathbb{T}_N\}$, parametrized by the index set $\mathbb T_N$, be the collection of disjoint $d$-dimensional cubes of sidelength $M^{-N}$ generated by the lattice $M^{-N}\mathbb{Z}^d$ in the set $\{0\}\times[0,1)^d$.  More specifically, each $Q_t$ is of the form \begin{equation} Q_t = \{0 \} \times \prod_{l=1}^d\left[\frac {j_l}{M^N},\frac {j_{l}+1}{M^N}\right), \label{defn Q_t}\end{equation} for some $\mathbf{j} = ( j_1,\ldots, j_d)\in \{0, 1, \cdots, M^N-1\}^d$, so that $\#(\mathbb{T}_N) = M^{Nd}$.
For technical reasons, we also define $\widetilde{Q}_t$ to be the $\kappa_d$-dilation of $Q_t$ about its center point, where $\kappa_d$ is a small, positive, dimension-dependent constant. The reason for this technicality, as well as possible values of $\kappa_d$, will soon emerge in the sequel, but for concreteness choosing $\kappa_d = d^{-d}$ will suffice. 

Recall that the $N$th iterate $\mathcal{C}^{[N]}_M$ of the Cantor construction is the union of $2^N$ disjoint intervals each of length $M^{-N}$.  We choose a representative element of $\mathcal{C}_M$ from each of these intervals, calling the resulting finite collection $\mathcal{D}^{[N]}_M$.  Clearly $\text{dist}(x, \mathcal {D}^{[N]}_M) \leq M^{-N}$ for every $x \in \mathcal C_M$.  Set 
\begin{equation} \label{defn Omega_N} 
\Omega_N := \gamma(\mathcal{D}^{[N]}_M),
\end{equation} 
 so that $\text{dist}(\omega,\Omega_N) \leq CM^{-N}$ for any $\omega\in\Omega$, with $C$ as in (\ref{Lipschitz}).  

For any $t \in \mathbb T_N$ and any $\omega \in \Omega_N$, we define  
\begin{equation}  \mathcal P_{t, \omega} := \left\{ r + s \omega : r \in \widetilde{Q}_t, \; 0 \leq s \leq 10 C_0\right\},  \label{defn-Ptsigma} \end{equation}  
where $C_0$ is a large constant to be determined shortly (for instance, $C_0 = d^d c^{-1}$ will work, with $c$ as in \eqref{Lipschitz}). Thus the set $\mathcal P_{t, \omega}$ is a cylinder oriented along $\omega$. Its (vertical) cross-section in the plane $x_1=0$ is the cube $\widetilde{Q}_t$. We say that $\mathcal P_{t,\omega}$ is \textit{rooted} at $Q_t$.  While $\mathcal P_{t, \omega}$ is not strictly speaking a tube as defined in the introduction, the distinction is negligible, since $\mathcal P_{t, \omega}$ contains and is contained in constant multiples of $\delta$-tubes with $\delta = \kappa_d\cdot M^{-N}$. By a slight abuse of terminology but no loss of generality, we will henceforth refer to $\mathcal P_{t, \omega}$ as a tube.  

If a slope assignment $\sigma : \mathbb{T}_N \rightarrow \Omega_N$ has been specified, we set $P_{t,\sigma} := \mathcal P_{t, \sigma(t)}$. Thus $\{P_{t, \sigma} : t \in \mathbb T_N \}$ is a family of tubes rooted at the elements of an $M^{-N}$-fine grid in $\{0\} \times [0,1)^d$, with essentially uniform length in $t$ that is bounded above and below by fixed absolute constants. Two such tubes are illustrated in Figure~\ref{twotubes}.  For the remainder, we set 
\begin{equation}\label{Kakeya set original}
	K_N(\sigma) := \bigcup_{t\in\mathbb{T}_N} P_{t,\sigma}.
\end{equation}

\begin{figure}[h]
\setlength{\unitlength}{0.25mm}
\begin{picture}(100,0)(-300,0)
        \allinethickness{0.254mm}\path(-50,0)(-50,-200)(-130,-300)(-130,-100)(-50,0) 

        \allinethickness{0.254mm}\special{sh 0.99}\put(-104,-204){\ellipse{1}{2}} 
        \put(-122,-200){\large\shortstack{$t_1$}} 
        \put(45,-210){\Large\shortstack{$P_{t_1,\sigma}$}} 
        \allinethickness{0.254mm}\put(-104,-204){\ellipse{5}{10}} 
        \allinethickness{0.254mm}\path(-104,-209)(180,-166) 
        \allinethickness{0.254mm}\path(-104,-199)(180,-156) 
        \allinethickness{0.254mm}\put(180,-161){\ellipse{5}{10}} 

        \allinethickness{0.254mm}\special{sh 0.99}\put(-80,-100){\ellipse{1}{2}} 
        \put(100,-98){\Large\shortstack{$P_{t_2,\sigma}$}} 
        \put(-98,-110){\large\shortstack{$t_2$}} 
        \allinethickness{0.254mm}\put(-80,-100){\ellipse{5}{10}} 
        \allinethickness{0.254mm}\path(-80,-95)(170,-115) 
        \allinethickness{0.254mm}\path(-80,-105)(170,-125) 
        \allinethickness{0.254mm}\put(170,-120){\ellipse{5}{10}}

\end{picture}
\vspace{8cm}\caption{Two typical tubes $P_{t_1,\sigma}$ and $P_{t_2,\sigma}$ rooted respectively at $t_1$ and $t_2$ in the $\{x_1=0\}$--coordinate plane.}\label{twotubes}
\end{figure}
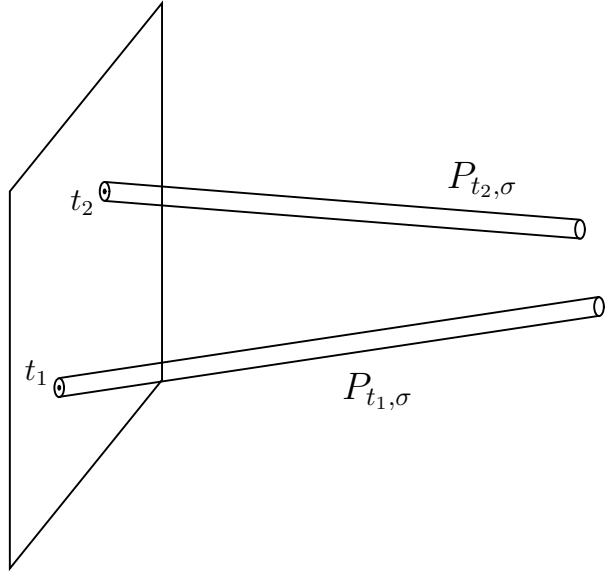
For a certain choice of slope assignment $\sigma$, this collection of tubes will be shown to generate a Kakeya-type set in the sense of Definition~\ref{Kakeya-type set}.  This particular slope assignment will not be explicitly described, but rather inferred from the contents of the following proposition.

\begin{proposition}\label{generic upper and lower bounds}
	For any $N\geq 1$, let $\Sigma_N$ be a finite collection of slope assignments from the lattice $\mathbb{T}_N$ to the direction set $\Omega_N$.  Every $\sigma  \in \Sigma_N$ generates a set $K_N(\sigma)$ as defined in (\ref{Kakeya set original}). Denote the power set of $\Sigma_N$ by $\mathfrak{P}(\Sigma_N)$.  

Suppose that $(\Sigma_N,\mathfrak{P}(\Sigma_N),\text{Pr})$ is a discrete probability space equipped with the probability measure $\text{Pr}$, for which the random sets $K_N(\sigma)$ obey the following estimates:
\begin{equation}\label{generic lower bound}
	\text{Pr}\left(\{\sigma : |K_N(\sigma)\cap [0,1]\times\mathbb{R}^d|\geq a_N\}\right) \geq \frac 34,
\end{equation}
and
\begin{equation}\label{generic upper bound}
	\mathbb{E}_{\sigma}|K_N(\sigma)\cap [C_0,C_0+1]\times\mathbb{R}^d|\leq b_N,
\end{equation}
where $C_0 \geq 1$ is a fixed constant, and $\{a_N\}$, $\{b_N\}$ are deterministic sequences satisfying 
$$\frac {a_N}{b_N}\rightarrow\infty,\quad\text{as}\quad N\rightarrow\infty.$$  Then $\Omega$ admits Kakeya-type sets.
\end{proposition}

\begin{proof}
	Fix any integer $N\geq 1$.  Applying Markov's Inequality to \eqref{generic upper bound}, we see that
\[ \text{Pr}\left(\{\sigma : |K_N(\sigma)\cap[C_0,C_0+1]\times\mathbb{R}^d|\geq 4b_N\}\right) \leq \frac{\mathbb{E}_{\sigma} |K_N(\sigma)\cap [C_0,C_0+1]\times\mathbb{R}^d|}{4b_N} \leq \frac14, \]  
so, 
\begin{equation}\text{Pr}\left(\{\sigma : |K_N(\sigma)\cap[C_0,C_0+1]\times\mathbb{R}^d|\leq 4b_N\}\right) \geq \frac 34. \label{generic upper bound 2} \end{equation} 
Combining this estimate with \eqref{generic lower bound}, we find that 
\begin{align*}
&\text{Pr}\left(\bigl\{\sigma : |K_N(\sigma)\cap[0,1]\times\mathbb{R}^d|\geq a_N \bigr\}\bigcap \bigl\{ \sigma : |K_N(\sigma)\cap[C_0,C_0+1]\times\mathbb{R}^d|\leq 4b_N \bigr\}\right) \\ 
&\geq \text{Pr}\left(\bigl\{|K_N(\sigma)\cap [0,1]\times\mathbb{R}^d|\geq a_N \bigr\}\right) + \text{Pr}\left(\bigl\{ |K_N(\sigma)\cap[C_0,C_0+1]\times\mathbb{R}^d|\leq 4b_N \bigr\}\right) -1 \\ 
&\geq \frac{3}{4} + \frac34 - 1 = \frac 12.  
\end{align*}
We may therefore choose a particular $\sigma\in \Sigma_N$ for which the size estimates on $K_N(\sigma)$ given by (\ref{generic lower bound}) and (\ref{generic upper bound 2}) hold simultaneously. Set
$$E_N := K_N(\sigma)\cap[C_0,C_0+1]\times\mathbb{R}^d, \quad \text{ so that } \quad E_N^{\ast} (2C_0+1)\supseteq  K_N(\sigma)\cap[0,1]\times\mathbb{R}^d.$$
Then $E_N$ is a union of $\delta$-tubes oriented along directions in $\Omega_N\subset\Omega$ for which \[  \frac{|E_N^{\ast}(2C_0+1)|}{|E_N|} \geq \frac{a_N}{4b_N} \rightarrow \infty,\quad\text{as}\quad N\rightarrow\infty,\] by hypothesis.  
This shows that $\Omega$ admits Kakeya-type sets, per condition \eqref{Kakeya-type condition}.
\end{proof}

Proposition~\ref{generic upper and lower bounds} proves part (a) of our Theorem~\ref{main theorem}. The implication (a) $\implies$ (b) has already been discussed in Section~\ref{results}.  The remainder of this paper is devoted to establishing a proper randomization over slope assignments $\Sigma_N$ that will then allow us to verify the hypotheses of Proposition~\ref{generic upper and lower bounds} for suitable sequences $\{a_N\}$ and $\{b_N\}$. We return to a more concrete formulation of the required estimates in Proposition \ref{Kakeya on average}.

%%%%%%%%%%%%%%%%%%%%%%%%%%%%%%%%%%%%%%%%%%%%%%%%%%%%%%%%%%%%%%%%%%%%%%%

\section{Geometric Facts} \label{geometry section}

In this section, we will take the opportunity to establish some geometric facts about two intersecting tubes in Euclidean space.  These facts will be used in several instances within the proof of Theorem~\ref{main theorem}. Nonetheless they are really general observations that are not limited to our specific arrangement or description of tubes.  
\begin{lemma} \label{intersection criterion lemma}
For $v_1, v_2 \in \Omega_N$ and $t_1, t_2 \in \mathbb T_N$, $t_1 \ne t_2$, let $\mathcal P_{t_1, v_1}$ and $\mathcal P_{t_2, v_2}$ be the tubes defined as in (\ref{defn-Ptsigma}). If there exists $p = (p_1, \cdots, p_{d+1}) \in  \mathcal P_{t_1, v_1} \cap \mathcal P_{t_2, v_2}$, then the inequality
 \begin{equation} \label{intersection criterion inequality} \bigl| \text{cen}(Q_{t_2}) - \text{cen}(Q_{t_1}) + p_1(v_2-v_1) \bigr| \leq 2 \kappa_d \sqrt{d} M^{-N}, \end{equation} 
holds, where $\text{cen}(Q)$ denotes the centre of the cube $Q$.  
\end{lemma}
\begin{proof}
The proof is described in the diagram below. If $p \in \mathcal P_{t_1, v_1} \cap \mathcal P_{t_2, v_2}$, then there exist $x_1 \in \widetilde{Q}_{t_1}$, $x_2 \in \widetilde{Q}_{t_2}$ such that $p = x_1 + p_1v_1 = x_2 + p_1 v_2$, i.e., $p_1(v_2-v_1) = x_1 - x_2$. The inequality (\ref{intersection criterion inequality}) follows since $|x_i - \text{cen}(Q_{t_i})| \leq \kappa_d \sqrt{d} M^{-N}$ for $i=1,2$.   
\end{proof} 

\begin{figure}[h!]
\setlength{\unitlength}{0.25mm}
\begin{picture}(100,0)(-300,0)
        \allinethickness{0.254mm}\path(-50,0)(-50,-200)(-130,-300)(-130,-100)(-50,0) 

        \allinethickness{0.254mm}\special{sh 0.99}\put(-104,-184){\ellipse{5}{5}} 
        \put(-128,-190){\large\shortstack{$x_1$}} 
        \put(-10,-210){\Large\shortstack{$\mathcal{P}_{t_1,v_1}$}} 
        \allinethickness{0.254mm}\put(-104,-184){\ellipse{15}{30}} 
        \allinethickness{0.254mm}\path(-104,-199)(180,-156) 
        \allinethickness{0.254mm}\path(-104,-169)(180,-126) 
        \allinethickness{0.254mm}\put(180,-141){\ellipse{15}{30}} 

        \allinethickness{0.254mm}\special{sh 0.99}\put(-80,-100){\ellipse{5}{5}} 
        \put(0,-105){\Large\shortstack{$\mathcal{P}_{t_2,v_2}$}} 
        \put(-80,-85){\large\shortstack{$x_2$}} 
        \allinethickness{0.254mm}\put(-80,-105){\ellipse{15}{30}} 
        \allinethickness{0.254mm}\path(-80,-90)(180,-175) 
        \allinethickness{0.254mm}\path(-80,-120)(180,-205) 
        \allinethickness{0.254mm}\put(180,-190){\ellipse{15}{30}} 

        \allinethickness{0.254mm}\special{sh 0.99}\put(90,-155){\ellipse{5}{5}} 
        \put(95,-162){\large\shortstack{$p$}} 

        \allinethickness{0.254mm}\dottedline{5}(-104,-184)(90,-155) 
        \allinethickness{0.254mm}\dottedline{5}(-80,-100)(90,-155) 
        \allinethickness{0.254mm}\dottedline{5}(-104,-184)(-80,-100) 

\end{picture}
\vspace{7.5cm}\caption{A simple triangle is defined by two rooted tubes, $\mathcal{P}_{t_1,v_1}$ and $\mathcal{P}_{t_2,v_2}$, and any point $p$ in their intersection.}\label{triangle}
\end{figure}
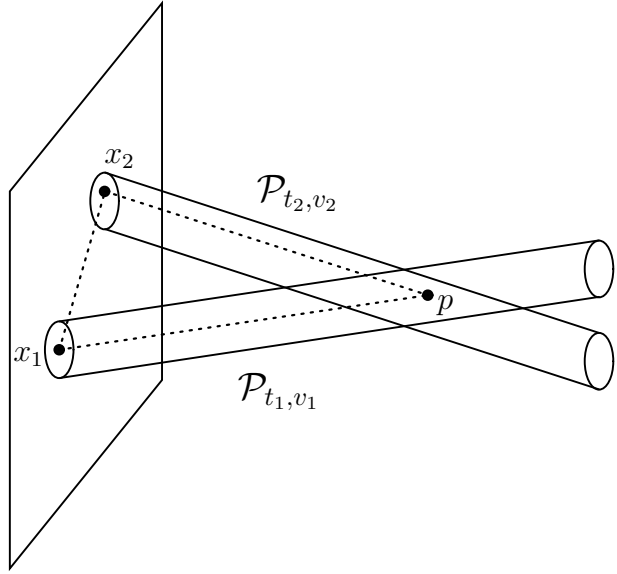

The inequality in \eqref{intersection criterion inequality} provides a valuable tool whenever an intersection takes place. For the reader who would like to look ahead, the Lemma~\ref{intersection criterion lemma} will be used along with Corollary~\ref{which is bigger corollary} to establish Lemma~\ref{dist to bdry lemma}.  The following Corollary~\ref{counting t_2 given t_1, v_1, v_2} will be needed for the proofs of Lemmas~\ref{defn of E(t_1, v_1) lemma} and~\ref{second moment lemma S_{42} estimate}.  
\begin{corollary} \label{which is bigger corollary}
Under the hypotheses of Lemma \ref{intersection criterion lemma} and for $\kappa_d > 0$ suitably small, 
\begin{equation} \label{which is bigger}
|p_1 (v_2- v_1)| \geq \kappa_d M^{-N}.  
\end{equation}  
\end{corollary} 
\begin{proof} 
Since $t_1 \ne t_2$, we must have $|\text{cen}(Q_{t_1}) - \text{cen}(Q_{t_2})| \geq M^{-N}$. Thus an intersection is possible only if
\[ p_1|v_2-v_1| \geq |\text{cen}(Q_{t_2}) - \text{cen}(Q_{t_1})| - 2\kappa_d \sqrt{d} M^{-N} \geq (1 - 2\kappa_d \sqrt{d}) M^{-N} \geq \kappa_d M^{-N}, \]
where the first inequality follows from \eqref{intersection criterion inequality} and  the last inequality holds for an appropriate selection of $\kappa_d$. 
\end{proof}
\begin{corollary} \label{counting t_2 given t_1, v_1, v_2}
If $t_1 \in \mathbb T_N$, $v_1, v_2 \in \Omega_N$ and a cube $Q \subseteq \mathbb R^{d+1}$ of sidelength $C_1M^{-N}$ with sides parallel to the coordinate axes are given, then there exists at most $C_2 = C_2(C_1)$ choices of $t_2 \in \mathbb T_N$ such that $\mathcal P_{t_1, v_1} \cap \mathcal P_{t_2, v_2} \cap Q \ne \emptyset$.  
\end{corollary} 
\begin{proof}
As $p = (p_1, \cdots, p_{d+1})$ ranges in $Q$, $p_1$ ranges over an interval $I$ of length $C_1 M^{-N}$. If $p \in \mathcal P_{t_1, v_1} \cap \mathcal P_{t_2, v_2} \cap Q$,  the inequality \eqref{intersection criterion inequality} and the fact diam$(\Omega) \leq$ diam$(\{1\}\times[-1,1]^d) = 2 \sqrt{d}$ implies
\begin{align*} \bigl| \text{cen}(Q_{t_2}) - \text{cen}(Q_{t_1}) + \text{cen}(I)(v_2-v_1) \bigr| &\leq |(p_1-\text{cen}(I))(v_2-v_1)| + 2 \kappa_d \sqrt{d} M^{-N} \\ &=2\sqrt{d} (C_1 + \kappa_d) M^{-N},   \end{align*}  
restricting $\text{cen}(Q_{t_2})$ to lie in a cube of sidelength $2\sqrt{d} (C_1 + \kappa_d) M^{-N}$ centred at $\text{cen}(Q_{t_1}) - \text{cen}(I) (v_2-v_1)$. Such a cube contains at most $C_2$ sub-cubes of the form \eqref{defn Q_t}, and the result follows.   
\end{proof} 

A recurring theme in the proof of Theorem \ref{main theorem} is the identification of a criterion that ensures that a specified point lies in the Kakeya-type set $K_N(\sigma)$ defined in \eqref{Kakeya set original}. With this in mind, we introduce for any $p = (p_1, p_2, \cdots, p_{d+1}) \in [0,10C_0] \times \mathbb R^d$  a set
\begin{equation} \label{defn Poss(p)}
\text{Poss}(p) := \bigl\{Q_t : t \in \mathbb T_N, \; \text{ there exists } v \in \Omega_N \text{ such that } p \in \mathcal P_{t, v} \bigr\}.
\end{equation}    
This set captures all the possible $M^{-N}$-cubes of the form \eqref{defn Q_t} in $\{0\} \times [0,1)^d$ such that a tube rooted at one of these cubes has the potential to contain $p$, provided it is given the correct orientation. Note that Poss$(p)$ is independent of any slope assignment $\sigma$. Depending on the location of $p$, Poss$(p)$ could be empty. This would be the case if $p$ lies outside a large enough compact subset of $[0, 10C_0] \times \mathbb R^d$, for example. Even if Poss$(p)$ is not empty, an arbitrary slope assignment $\sigma$ may not endow {\em{any}} $Q_t$ in Poss$(p)$ with the correct orientation.

In the next lemma, we list a few easy properties of Poss$(p)$ that will be helpful later, particularly during the proof of Lemma~\ref{Resistance}. Lemma~\ref{description of Poss(p) lemma} establishes the main intuition behind the Poss$(p)$ set, as we give a more geometric description of Poss$(p)$ in terms of an affine copy of the direction set $\Omega_N$.  This is illustrated in Figure~\ref{Fig:PossSet} for a particular choice of directions $\Omega_N$.
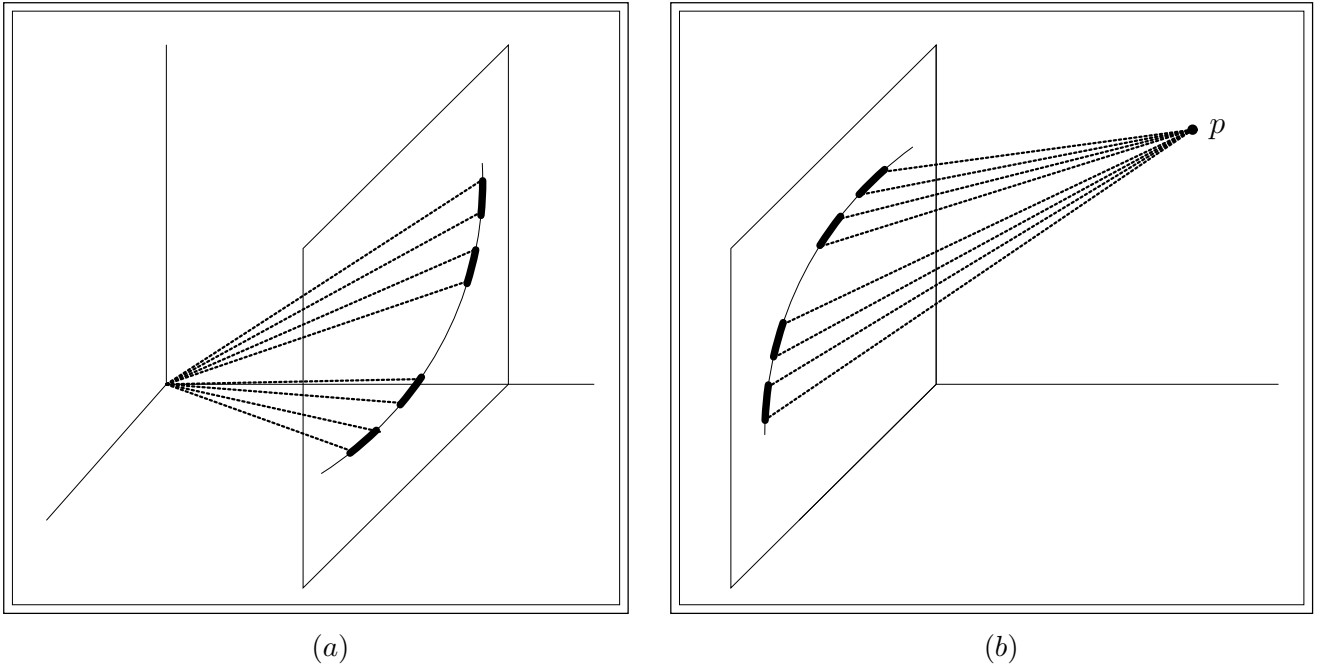
\begin{figure}[h!]
\setlength{\unitlength}{0.225mm}
\begin{picture}(-50,20)(-100,20)
        \allinethickness{0.1mm}\path(-50,0)(-50,-200)(-120,-280)
        \allinethickness{0.1mm}\path(-50,-200)(200,-200) 

        \allinethickness{0.1mm}\path(150,0)(150,-200)(30,-320)(30,-120)(150,0)

        \allinethickness{0.1mm}\put(-70,-80){\arc{410}{-.05}{1}} 
	
	        \allinethickness{1mm}\put(-70,-80){\arc{410}{.8}{.9}}
	        \allinethickness{1mm}\put(-70,-80){\arc{410}{.6}{.7}}
	        \allinethickness{1mm}\put(-70,-80){\arc{410}{.2}{.3}}
	        \allinethickness{1mm}\put(-70,-80){\arc{410}{0}{.1}}

        \allinethickness{0.3mm}\dottedline{3}(-50,-200)(60,-240)
        \allinethickness{0.3mm}\dottedline{3}(-50,-200)(75,-228)
        \allinethickness{0.3mm}\dottedline{3}(-50,-200)(87,-211)
        \allinethickness{0.3mm}\dottedline{3}(-50,-200)(98,-197)
        \allinethickness{0.3mm}\dottedline{3}(-50,-200)(127,-139)
        \allinethickness{0.3mm}\dottedline{3}(-50,-200)(132,-120)
        \allinethickness{0.3mm}\dottedline{3}(-50,-200)(135,-98)
        \allinethickness{0.3mm}\dottedline{3}(-50,-200)(134,-80)

        \allinethickness{0.1mm}\path(-140,20)(-140,-330)(215,-330)(215,20)(-140,20)
        \allinethickness{0.2mm}\path(-145,25)(-145,-335)(220,-335)(220,25)(-145,25)
        \put(35,-360){\shortstack{$(a)$}} 

%%%%%%%%%%%%%%%%%%

        \allinethickness{0.1mm}\path(400,0)(400,-200)(320,-280)
        \allinethickness{0.1mm}\path(400,-200)(600,-200)

        \allinethickness{0.1mm}\path(400,0)(400,-200)(280,-320)(280,-120)(400,0)

        \allinethickness{0.1mm}\put(510,-230){\arc{420}{-3.14}{-2.2}} 
	
	        \allinethickness{1mm}\put(510,-230){\arc{420}{-3.1}{-3}}
	        \allinethickness{1mm}\put(510,-230){\arc{420}{-2.92}{-2.82}}
	        \allinethickness{1mm}\put(510,-230){\arc{420}{-2.58}{-2.48}}
	        \allinethickness{1mm}\put(510,-230){\arc{420}{-2.4}{-2.3}}

        \allinethickness{0.254mm}\special{sh 0.99}\put(550,-50){\ellipse{5}{5}} 
        \put(560,-52){\large\shortstack{$p$}} 
        \allinethickness{0.3mm}\dottedline{3}(550,-50)(300,-221)
        \allinethickness{0.3mm}\dottedline{3}(550,-50)(301,-202)
        \allinethickness{0.3mm}\dottedline{3}(550,-50)(304,-185)
        \allinethickness{0.3mm}\dottedline{3}(550,-50)(310,-165)
        \allinethickness{0.3mm}\dottedline{3}(550,-50)(332,-119)
        \allinethickness{0.3mm}\dottedline{3}(550,-50)(342,-103)
        \allinethickness{0.3mm}\dottedline{3}(550,-50)(357,-88)
        \allinethickness{0.3mm}\dottedline{3}(550,-50)(368,-75)

        \allinethickness{0.1mm}\path(250,20)(250,-330)(615,-330)(615,20)(250,20)
        \allinethickness{0.2mm}\path(245,25)(245,-335)(620,-335)(620,25)(245,25)
        \put(428,-360){\shortstack{$(b)$}} 

\end{picture}
\vspace{9cm}\caption{Figure (a) depicts the cone generated by a second stage Cantor construction, $\Omega_2$, on the set of directions given by the curve $\{(1,t,t^2) : 0\leq t\leq C\}$ in the $\{1\}\times\mathbb{R}^2$ plane.  In Figure (b), a point $p = (p_1,p_2,p_3)$ has been fixed and the cone of directions has been projected backward from $p$ onto the coordinate plane, $p-p_1\Omega_2$.  The resulting Poss$(p)$ set is thus given by all cubes $Q_t$, $t\in\mathbb{T}_N$ such that $\widetilde{Q}_t$ intersects a subset of the curve $\{(0,p_2-p_1t,p_3-p_1t^2) : 0\leq t\leq C\}$.}\label{Fig:PossSet}
\end{figure}

\begin{lemma} \label{description of Poss(p) lemma}
\begin{enumerate}[(a)]
\item \label{Poss(p) and sigma} For any slope assignment $\sigma$, 
\[ \bigl\{Q_t: t \in \mathbb T_N, p \in P_{t, \sigma} \bigr\} \subseteq \text{Poss}(p). \] 
\item \label{representation of Poss(p)} For any $p \in [0,10C_0] \times \mathbb R^d$, 
\begin{align} \text{Poss}(p) &= \bigl\{Q_t : \widetilde{Q}_t \cap (p - p_1 \Omega_N) \ne \emptyset  \bigr\} \label{representation of Poss(p) equation}\\ &\subseteq \{ Q_t : t \in \mathbb T_N, Q_t \cap (p - p_1 \Omega_N) \ne \emptyset \}. \label{containment of Poss(p)} \end{align} 
Note that the set in (\ref{representation of Poss(p) equation}) could be empty, but the one in (\ref{containment of Poss(p)}) is not. 
\end{enumerate}
\end{lemma} 
\begin{proof}
If $p \in P_{t, \sigma}$, then $p \in \mathcal P_{t, \sigma(t)}$ with $\sigma(t)$ equal to some $v \in \Omega$. Thus $\mathcal P_{t,v}$ contains $p$ and hence $Q_t \in \text{Poss}(p)$, proving part (\ref{Poss(p) and sigma}). For part (\ref{representation of Poss(p)}), we observe that $p \in \mathcal P_{t,v}$ for some $v \in \Omega_N$ if and only if $p - p_1 v \in \widetilde{Q}_t$, i.e., $\widetilde{Q}_t \cap (p - p_1 \Omega_N) \ne \emptyset$. This proves the relation (\ref{representation of Poss(p) equation}). The containment in (\ref{containment of Poss(p)}) is obvious.
\end{proof} 
We will also need a bound on the cardinality of Poss$(p)$ within a given cube, and on the cardinality of possible slopes that give rise to indistinguishable tubes passing through a given point. We now prescribe these.  Lemmas~\ref{Poss(p) cardinality lemma} and~\ref{Possible slopes cardinality lemma} are not technically needed for the remainder, but can be viewed as steps toward establishing Lemma~\ref{away from root hyperplane lemma} which will prove critical throughout Section~\ref{upperboundsection}.  Not surprisingly, the Cantor-like construction of $\Omega$ plays a role in all these estimates.  
\begin{lemma} \label{Poss(p) cardinality lemma} Given $C_0, C_1 > 0$, there exists $C_2 = C_2(C_0, C_1, M, d) > 0$ with the following property. 
Let $p = (p_1, \cdots, p_{d+1}) \in (0,10C_0] \times \mathbb R^d$, and $Q$ be any cube in $\{0\} \times [0,1)^d$ with sidelength in $[M^{-\ell}, M^{-\ell+1})$ for some $\ell \leq N-1$.  Then
\begin{equation} \# \bigl\{Q_t : t \in \mathbb T_N, Q_t \cap Q \ne \emptyset, \, \text{dist}(Q_t, p-p_1 \Omega_N) \leq C_1 M^{-N} \bigr\} \leq C_2 2^{N - \ell}. \label{counting Poss(p) inequality} \end{equation}   
\end{lemma} 
\begin{proof}
Let $j \in \mathbb Z$ be the index such that $M^{-j} \leq p_1 < M^{-j+1}$. By scaling, the left hand side of \eqref{counting Poss(p) inequality}  is comparable to (i.e., bounded above and below by constant multiples of) the number of $p_1^{-1} M^{-N}$-separated points lying in  \[Q' := \bigl\{ x \in p_1^{-1}Q: \, \text{dist}(x, p_1^{-1}p - \Omega_N) \leq C_1 p_1^{-1} M^{-N} \bigr\}.\] 
But $p_1^{-1} p - \Omega_N = (1, c) - \Omega_N$ is an image of $\Omega_N$ following an inversion and translation. This implies that there is a subset $\Omega_N'$ of $\Omega_N$, depending on $p$ and $p_1^{-1}Q$ and with diameter $O(M^{j - \ell})$, such that $Q'$ is contained in a $O(M^{j-N})$-neighborhood of $-\Omega_N' + (1, c)$. The number of $M^{j-N}$-separated points in $Q'$ is comparable to that in $\Omega_N'$. 

Suppose first that $j \leq \ell$. If $\mathcal C' \subseteq \mathcal C_M^{[N]}$ is defined by the requirement $\Omega_N' = \gamma(\mathcal C')$, then (\ref{Lipschitz}) implies that diam$(\mathcal C') = O(M^{j-\ell})$. Thus $\mathcal C'$ is contained in at most $O(1)$ intervals of length $M^{j-\ell}$ chosen at step $ \ell-j$ in the Cantor-type construction.  Each chosen interval at the $k$th stage gives rise to two chosen subintervals at the next stage, with their centres being separated by at least $M^{-k-1}$. So the number of $M^{j-N}$-separated points in $\mathcal C'$, and hence $\gamma(\mathcal C')$ is $O(2^{(N-j) - (\ell-j)}) = O(2^{N-\ell})$ as claimed. The case $j \geq \ell$ is even simpler, since the number of $M^{j-N}$-separated points in $\mathcal C'$ is trivially bounded by $2^{N-j} \leq 2^{N-\ell}$.  
\end{proof} 

\begin{lemma} \label{Possible slopes cardinality lemma}
Fix $t \in \mathbb T_N$ and $p = (p_1, \cdots, p_{d+1}) \in [M^{-\ell}, M^{-\ell+1}] \times \mathbb R^d$, for some $0 \leq \ell \ll N$. Let $Q$ be a cube centred at $p$ of sidelength $C_1 M^{-N}$. Then 
\[ \# \bigl\{ v \in \Omega_N : Q \cap \mathcal P_{t,v} \ne \emptyset \bigr\} \leq C_2 2^\ell.  \]  
\end{lemma} 
\begin{proof} 
If both $\mathcal P_{t,v}$ and $\mathcal P_{t,v'}$ have nonempty intersection with $Q$, then there exist $q = (q_1, \cdots, q_{d+1}), q' = (q_1', \cdots, q'_{d+1}) \in Q$ such that both $q - q_1v$ and $q' - q_1'v'$ land in $\widetilde{Q}_t$. Thus, \begin{align*} p_1 |v - v'| &\leq |(q - p_1v) - (q' - p_1v')| + |q-q'| \\ &\leq  |(q - q_1v) - (q' - q_1'v')| + |q_1-p_1||v| + |q_1'-p_1||v'| +   |q-q'| \\ &\leq (\kappa_d \sqrt{d} + 10 C_1 \sqrt{d} )M^{-N}. \end{align*} In other words, $|v-v'| \leq (10 C_1+ \kappa_d) \sqrt{d} M^{\ell-N}$. Recalling that $v = \gamma(\alpha)$ and $v' = \gamma(\alpha')$ for some $\alpha, \alpha' \in \mathcal D_M^{[N]}$, combining the last inequality with (\ref{Lipschitz}) implies that $|\alpha - \alpha'| \leq C_2 M^{\ell-N}$. Thus there is a collection of at most $O(1)$ chosen intervals at step $N - \ell$ of the Cantor-type construction which $\alpha$ (and hence $\alpha'$) can belong to. Since each interval gives rise to two chosen intervals at the next stage, the number of possible $\alpha$ and hence $v$ is $O(2^{\ell})$.     
\end{proof} 
A slight modification of the proof above yields a stronger conclusion, stated below, when $p$ is far away from the root hyperplane. We will return to this result several times in the sequel (see for example Lemma \ref{away from root hyperplane lemma - tree version} for a version of it in the language of trees), and make explicit use of it in Section~\ref{upperboundsection}, specifically in the proofs of Lemmas~\ref{percolation description lemma} and~\ref{PossLemma}.
\begin{lemma} \label{away from root hyperplane lemma}
There exists a constant $C_0 \geq 1$ with the following properties.
\begin{enumerate}[(a)]
\item \label{exactly one slope per t} For any $p \in [C_0, C_0+1] \times \mathbb R^d$ and $t \in \mathbb T_N$, there exists at most one $v \in \Omega_N$ such that $p \in \mathcal P_{t,v}$.  In other words, for every $Q_t$ in Poss$(p)$, there is exactly one $\delta$-tube rooted at $t$ that contains $p$.  
\item \label{exactly one basic interval at every stage} For any $p$ as in (\ref{exactly one slope per t}), and $Q_t$, $Q_{t'} \in$ Poss$(p)$, let $v = \gamma(\alpha)$, $v' = \gamma(\alpha')$ be the two unique slopes in $\Omega_N$ guaranteed by (\ref{exactly one slope per t}) such that $p \in \mathcal P_{t,v} \cap \mathcal P_{t',v'}$. If $k$ is the largest integer such that $Q_t$ and $Q_{t'}$ are both contained in the same cube $Q \subseteq \{0\} \times [0,1)^d$ of sidelength $M^{-k}$ whose corners lie in $M^{-k} \mathbb Z^d$, then $\alpha$ and $\alpha'$ belong to the same $k$th stage basic interval in the Cantor construction. 
\end{enumerate} 
\end{lemma}  
\begin{proof}
\begin{enumerate}[(a)]
\item Suppose $v, v' \in \Omega_N$ are such that $p \in \mathcal P_{t,v} \cap \mathcal P_{t,v'}$. Then $p - p_1 v$ and $p - p_1v'$ both lie in $\widetilde{Q}_t$, so that $p_1 |v-v'| \leq \kappa_d \sqrt{d} M^{-N}$. Since $p_1 \geq C_0$ and \eqref{Lipschitz} holds, we find that 
\[ |\alpha - \alpha'| \leq \frac{\kappa_d \sqrt{d}}{cC_0} M^{-N} < M^{-N},\]
where the last inequality holds if $C_0$ is chosen large enough. Let us recall from the description of the Cantor-like construction in Section \ref{results} that any two basic $r$th stage intervals are non-adjacent, and hence any two points in $\mathcal C_{M}$ lying in distinct basic $r$th stage intervals are separated by at least $M^{-r}$. Therefore the inequality above implies that both $\alpha$ and $\alpha'$ belong to the same basic $N$th stage interval in $\mathcal C_M^{[N]}$. But $\mathcal D_M^{[N]}$ contains exactly one element from each such interval. So $\alpha = \alpha'$ and hence $v = v'$.     
\item If $p \in \mathcal P_{t,v} \cap \mathcal P_{t',v'}$, then $p_1|v-v'| \leq \text{diam}(\widetilde{Q}_t \cup \widetilde{Q}_{t'}) \leq \text{diam}(Q) = \sqrt{d} M^{-k}$. Applying \eqref{Lipschitz} again combined with $p_1 \geq C_0$, we find that $|\alpha - \alpha'| \leq \frac{\sqrt{d}}{cC_0} M^{-k} < M^{-k},$ for $C_0$ chosen large enough. By the same property of the Cantor construction as used in (\ref{exactly one slope per t}), we obtain that $\alpha$ and $\alpha'$ lie in the same $k$th stage basic interval in $\mathcal C_M^{[k]}$. 
\end{enumerate} 
\end{proof}

%%%%%%%%%%%%%%%%%%%%%%%%%%%%%%%%%%%%%%%%%%%%%%%%%%%%%%%%%%%%%%%%%%%%%%%

\section{Rooted, labelled trees} \label{tree section} 

\subsection{The terminology of trees}\label{trees}

An undirected graph $\mathcal{G} := (\mathcal{V},\mathcal{E})$ is a pair, where $\mathcal{V}$ is a set of vertices and $\mathcal{E}$ is a symmetric, nonreflexive subset of $\mathcal{V}\times \mathcal{V}$, called the edge set.  By symmetric, here we mean that the pair $(u,v)\in\mathcal{E}$ is unordered; i.e. the pair $(u,v)$ is identical to the pair $(v,u)$.  By nonreflexive, we mean $\mathcal{E}$ does not contain the pair $(v,v)$ for any $v\in\mathcal{V}$.  

A path in a graph is a sequence of vertices such that each successive pair of vertices is a distinct edge in the graph.  A finite path (with at least one edge) whose first and last vertices are the same is called a cycle.  A graph is connected if for each pair of vertices $v\neq u$, there is a path in $\mathcal{G}$ containing $v$ and $u$.  We define a \textit{tree} to be a connected undirected graph with no cycles. 

All our trees will be of a specific structure.  A \textit{rooted, labelled tree} $\mathcal{T}$ is one whose vertex set is a nonempty collection of finite sequences of nonnegative integers such that if $\langle i_1,\ldots,i_n\rangle\in \mathcal{T}$, then
\begin{enumerate}
	\item[(i.)] for any $k$, $0\leq k\leq n$, $\langle i_1,\ldots,i_k\rangle\in \mathcal{T}$, where $k=0$ corresponds to the empty sequence, and
	\item[(ii.)] for every $j\in \{0,1,\ldots,i_n\}$, we have $\langle i_1,\ldots,i_{n-1},j\rangle\in \mathcal{T}$.  
\end{enumerate}
We say that $\langle i_1,\ldots, i_{n-1}\rangle$ is the \textit{parent} of $\langle i_1,\ldots,i_{n-1},j\rangle$ and that $\langle i_1,\ldots,i_{n-1},j\rangle$ is the $(j+1)th$ \textit{child} of $\langle i_1,\ldots,i_{n-1}\rangle$.  If $u$ and $v$ are two sequences in $\mathcal{T}$ such that $u$ is a child of $v$, or a child's child of $v$, or a child's child's child of $v$, etc., then we say that $u$ is a \textit{descendant} of $v$ (or that $v$ is an \textit{ancestor} of $u$), and we write $u \subset v$ (see the remark below).  If $u=\langle i_1,\ldots,i_m\rangle\in\mathcal{T}$, $v = \langle j_1,\ldots,j_n\rangle\in\mathcal{T}$, $m\leq n$, and neither $u$ nor $v$ is a descendant of the other, then the \textit{youngest common ancestor} of $u$ and $v$ is the vertex in $\mathcal{T}$ defined by 
\begin{equation} \label{defn youngest common ancestor} 
	D(u,v) = D(v,u) := \begin{cases}  \emptyset, &\text{ if } i_1\neq j_1\\ \langle i_1,\ldots,i_k\rangle &\text{ if } k = \max\{l : i_l=j_l\}. \end{cases}
\end{equation}
One can similarly define the youngest common ancestor for any finite collection of vertices. \\

\noindent {\em{Remark}}: At first glance, using the notation $u\subset v$ to denote when $u$ is a descendant of $v$ may seem counterintuitive, since $u$ is a descendant of $v$ precisely when $v$ is a subsequence of $u$.  However, we will soon be identifying vertices of rooted labelled trees with certain nested families of cubes in $\mathbb{R}^d$.  Consequently, as will become apparent in the next two subsections, $u$ will be a descendant of $v$ precisely when the cube associated with $u$ is contained within the cube associated with $v$.\\

We designate the empty sequence $\emptyset$ as the \textit{root} of the tree $\mathcal{T}$.  The sequence $\langle i_1,\ldots,i_n\rangle$ should be thought of as the vertex in $\mathcal{T}$ that is the $(i_n+1)th$ child of the $(i_{n-1}+1)th$ child,$\ldots$, of the $(i_1+1)th$ child of the root.  All unordered pairs of the form $(\langle i_1,\ldots,i_{n-1}\rangle,\langle i_1,\ldots,i_{n-1},i_n\rangle)$ describe the edges of the tree $\mathcal{T}$.  We say that the edge originates at the vertex $\langle i_1,\ldots,i_{n-1}\rangle$ and that it terminates at the vertex $\langle i_1,\ldots,i_{n-1},i_n\rangle$.  Note that every vertex in the tree that is not the root is uniquely identified by the edge terminating at that vertex.  Consequently, given an edge $e\in\mathcal{E}$, we define $v(e)$ to be the vertex in $\mathcal{V}$ at which $e$ terminates.  The vertex $\langle i_1,\ldots,i_n\rangle\in \mathcal{T}$ also prescribes a unique path, or \textit{ray}, from the root to this vertex:
\begin{equation}
	\emptyset\rightarrow\langle i_1\rangle\rightarrow\langle i_1,i_2\rangle\rightarrow \cdots \rightarrow \langle i_1,i_2,\ldots,i_n\rangle.\nonumber
\end{equation}
We let $\partial\mathcal{T}$ denote the collection of all rays in $\mathcal{T}$ of maximal (possibly infinite) length.  For a fixed vertex $v = \langle i_1,\ldots,i_m\rangle\in\mathcal{T}$, we also define \textit{the subtree (of $\mathcal{T}$) generated by the vertex} $v$ to be the maximal subtree of $\mathcal{T}$ with $v$ as the root; i.e. it is the subtree $$\{\langle i_1,\ldots,i_m,j_1,\ldots,j_k\rangle\in\mathcal{T} : k\geq 0\}.$$

The \textit{height} of the tree is taken to be the supremum of the lengths of all the sequences in the tree.  Further, we define the height $h(\cdot)$, or \textit{level}, of a vertex $\langle i_1,\ldots,i_n\rangle$ in the tree to be $n$, the length of its identifying sequence.  All vertices of height $n$ are said to be members of the $n$th \textit{generation} of the root, or interchangeably, of the tree.  More explicitly, a member vertex of the $n$th generation has exactly $n$ edges joining it to the root.  The height of the root is always taken to be zero.

If $\mathcal{T}$ is a tree and $n\in\mathbb{Z}^+$, we write the \textit{truncation} of $\mathcal{T}$ to its first $n$ levels as $\mathcal{T}_n = \{\langle i_1,\ldots,i_k\rangle\in \mathcal{T} : 0\leq k\leq n\}.$  This subtree is a tree of height at most $n$.  A tree is called \textit{locally finite} if its truncation to every level is finite; i.e. consists of finitely many vertices.  All of our trees will have this property.  In the remainder of this article, when we speak of a \textit{tree} we will always mean a \textit{locally finite, rooted labelled tree}, unless otherwise specified.

Roughly speaking, two trees are isomorphic if they have the same collection of rays.  To make this precise we define a special kind of map between trees that will turn out to be very important for us later.
\begin{definition}\label{D:stickiness}
	Let $\mathcal{T}_1$ and $\mathcal{T}_2$ be two trees with equal (possibly infinite) heights.  Let $\sigma: \mathcal{T}_1\rightarrow \mathcal{T}_2$; we call $\sigma$ {\bf sticky} if
\begin{enumerate}
	\item[$\bullet$] for all $v\in \mathcal{T}_1$, $h(v) = h(\sigma(v))$, and
	\item[$\bullet$] $u\subset v$ implies $\sigma(u)\subset\sigma(v)$ for all $u,v\in \mathcal{T}_1$.
\end{enumerate}
We often say that $\sigma$ is sticky if it preserves heights and lineages.
\end{definition}
A one-to-one and onto sticky map between two trees, whose inverse is then automatically sticky, is an \textit{isomorphism} and the two trees are said to be \textit{isomorphic}; we will write $\mathcal{T}_1 \cong \mathcal{T}_2$.  Two isomorphic trees can be treated as essentially identical objects.

%%%%%%%%%%%%%%%%%%%%%%%%%%%%%%%%%%%%%%%%%%%%%%%%%%%%%%%%%%%%%%%%%%%%%%%

\subsection{Encoding bounded subsets of the unit interval by trees}\label{tree encoding subsection}

The language of rooted labelled trees is especially convenient for representing bounded sets in Euclidean spaces.  This connection is well-studied in the literature.  We refer the interested reader to~\cite{LyonsPeres} for more information.

We start with $[0,1)\subset\mathbb{R}$.  Fix any positive integer $M\geq 2$.  We define an $M$-adic rational as a number of the form $i/M^k$ for some $i\in\mathbb{Z}$, $k\in\mathbb{Z}^+$, and an $M$-adic interval as $[i\cdot M^{-k},(i+1)\cdot M^{-k})$.  For any nonnegative integer $i$ and positive integer $k$ such that $i<M^k$, there exists a unique representation
\begin{equation}\label{M-adic representation}
	i = i_1M^{k-1} + i_2M^{k-2} + \cdots + i_{k-1}M + i_k,
\end{equation}
where the integers $i_1,\ldots,i_k$ take values in $\mathbb{Z}_M := \{0,1,\ldots,M-1\}$.  These integers should be thought of as the ``digits" of $i$ with respect to its base $M$ expansion.  An easy consequence of \eqref{M-adic representation} is that there is a one-to-one and onto correspondence between $M$-adic rationals in $[0,1)$ of the form $i/M^k$ and finite integer sequences $\langle i_1,\ldots,i_k\rangle$ of length $k$ with $i_j \in \mathbb{Z}_M$ for each $j$.  Naturally then, we define the tree of infinite height 
\begin{equation}\label{unit interval tree}
\mathcal{T}([0,1);M) = \{\langle i_1,\ldots,i_k\rangle : k\geq 0,\ i_j\in\mathbb{Z}_M\}.
\end{equation}
The tree thus defined depends of course on the base $M$; however, if $M$ is fixed, as it will be once we fix the direction set $\Omega = \gamma(\mathcal{C}_M)$ (see Section~\ref{results}), we will omit its usage in our notation, denoting the tree $\mathcal{T}([0,1);M)$ by $\mathcal{T}([0,1))$ instead.  

Identifying the root of the tree defined in \eqref{unit interval tree} with the interval $[0,1)$ and the vertex $\langle i_1,\ldots,i_k\rangle$ with the interval $[i\cdot M^{-k},(i+1)\cdot M^{-k})$, where $i$ and $\langle i_1,\ldots,i_k\rangle$ are related by \eqref{M-adic representation}, we observe that the vertices of $\mathcal{T}([0,1);M)$ at height $k$ yield a partition of $[0,1)$ into $M$-adic subintervals of length $M^{-k}$.  This tree has a self-similar structure: every vertex of $\mathcal{T}([0,1);M)$ has $M$ children and the subtree generated by any vertex as the root is isomorphic to $\mathcal{T}([0,1);M)$.  In the sequel, we will refer to such a tree as a \textit{full M-adic tree}.

Any $x\in[0,1)$ can be realized as the intersection of a nested sequence of $M$-adic intervals, namely $$\{x\} = \bigcap_{k=0}^{\infty}I_k(x),$$ where $I_k(x) = [i_k(x)\cdot M^{-k},(i_k(x)+1)\cdot M^{-k})$.  The point $x$ should be visualized as the destination of the infinite ray
\begin{equation}\label{identifying ray}
	\emptyset\rightarrow\langle i_1(x)\rangle\rightarrow\langle i_1(x),i_2(x)\rangle\rightarrow \cdots\rightarrow\langle i_1(x),i_2(x),\ldots,i_k(x)\rangle\rightarrow\cdots\nonumber
\end{equation}
in $\mathcal{T}([0,1);M)$.  Conversely, every infinite ray $$\emptyset\rightarrow \langle i_1\rangle\rightarrow\langle i_1,i_2\rangle\rightarrow\langle i_1,i_2,i_3\rangle\cdots$$ identifies a unique $x\in[0,1)$ given by the convergent sum $$x = \sum_{j=1}^{\infty} \frac{i_j}{M^j}.$$  Thus the tree $\mathcal{T}([0,1);M)$ can be identified with the interval $[0,1)$ exactly.  Any subset $E\subseteq[0,1)$ is then given by a subtree $\mathcal{T}(E;M)$ of $\mathcal{T}([0,1);M)$ consisting of all infinite rays that identify some $x\in E$.  As before, we will drop the notation for the base $M$ in $\mathcal{T}(E;M)$ once this base has been fixed.

Any \textit{truncation} of $\mathcal{T}(E;M)$, say up to height $k$, will be denoted by $\mathcal{T}_k(E;M)$ and should be visualized as a covering of $E$ by $M$-adic intervals of length $M^{-k}$.  More precisely, $\langle i_1,\ldots,i_k\rangle\in\mathcal{T}_k(E;M)$ if and only if $E\cap[i\cdot M^{-k},(i+1)\cdot M^{-k})\neq\emptyset$, where $i$ and $\langle i_1,\ldots,i_k\rangle$ are related by \eqref{M-adic representation}. 

We now state and prove a key structural result about our sets of interest, the generalized Cantor sets $\mathcal{C}_M$.
\begin{proposition}\label{prop - binary structure}
	Fix any integer $M\geq 3$.  Define $\mathcal{C}_M$ as in Section~\ref{results}.  Then $$\mathcal{T}(\mathcal{C}_M;M)\cong \mathcal{T}([0,1);2).$$  That is, the $M$-adic tree representation of $\mathcal{C}_M$ is isomorphic to the full binary tree, illustrated in Figure~\ref{binary tree}.
\end{proposition}

\begin{figure}[h]
\setlength{\unitlength}{.8mm}
\begin{picture}(-100,70)(-105,-65)

	\path(-90,0)(-36,0)
	\path(-90,2)(-90,-2)
	\path(-36,2)(-36,-2)

	\dottedline{1}(-72,-30)(-54,-30)
	\path(-90,-28)(-90,-32)
	\path(-36,-28)(-36,-32)
	\path(-72,-28)(-72,-32)
	\path(-54,-28)(-54,-32)

	\dottedline{1}(-72,-60)(-54,-60)
	\dottedline{1}(-84,-60)(-78,-60)
	\dottedline{1}(-48,-60)(-42,-60)
	\path(-90,-58)(-90,-62)
	\path(-36,-58)(-36,-62)
	\path(-72,-58)(-72,-62)
	\path(-54,-58)(-54,-62)
	\path(-78,-58)(-78,-62)
	\path(-84,-58)(-84,-62)
	\path(-42,-58)(-42,-62)
	\path(-48,-58)(-48,-62)

{	\allinethickness{.8mm}\path(-90,0)(-36,0)	
	\path(-90,-30)(-72,-30)
	\path(-54,-30)(-36,-30)
	\path(-90,-60)(-84,-60)
	\path(-78,-60)(-72,-60)
	\path(-36,-60)(-42,-60)
	\path(-48,-60)(-54,-60)

}

%%%%%%%%%%%%%%%%

	\special{sh 0.99}\put(30,0){\ellipse{2}{2}}
	\special{sh 0.99}\put(15,-30){\ellipse{2}{2}}
	\special{sh 0.99}\put(45,-30){\ellipse{2}{2}}
	\special{sh 0.01}\put(30,-30){\ellipse{2}{2}}
	\special{sh 0.99}\put(0,-60){\ellipse{2}{2}}
	\special{sh 0.99}\put(24,-60){\ellipse{2}{2}}
	\special{sh 0.99}\put(36,-60){\ellipse{2}{2}}
	\special{sh 0.01}\put(12,-60){\ellipse{2}{2}}
	\special{sh 0.01}\put(48,-60){\ellipse{2}{2}}
	\special{sh 0.99}\put(60,-60){\ellipse{2}{2}}

	\path(45,-30)(30,0)(15,-30)
	\path(0,-60)(15,-30)(24,-60)
	\path(36,-60)(45,-30)(60,-60)
	\dottedline{2}(30,0)(30,-30)
	\dottedline{2}(45,-30)(48,-60)
	\dottedline{2}(15,-30)(12,-60)

%%%%%%%%%%%%%%%%%%

	\dashline{2}(-25,0)(15,0)
	\path(13,-2)(15,0)(13,2)

	\dashline{2}(-25,-30)(0,-30)
	\path(-2,-28)(0,-30)(-2,-32)

	\dashline{2}(-25,-60)(-10,-60)
	\path(-12,-58)(-10,-60)(-12,-62)
\end{picture}
\caption{A pictorial depiction of the isomorphism between a standard middle-thirds Cantor set and its representation as a full binary subtree of the full base $M=3$ tree.}\label{binary tree}
\end{figure}
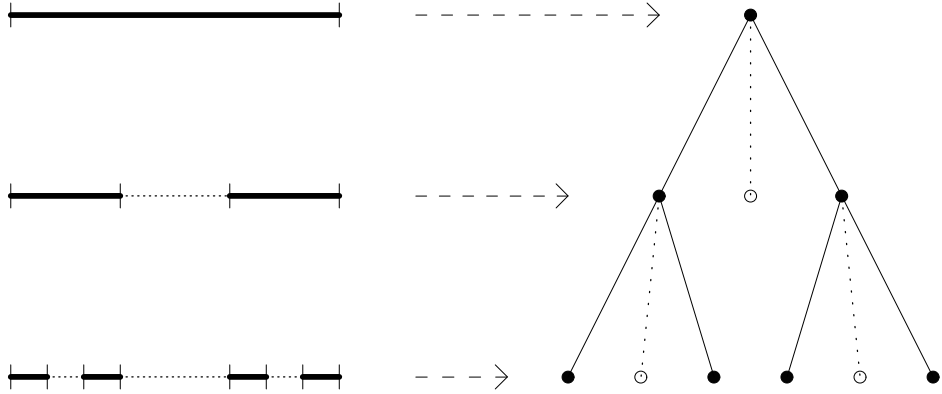

\begin{proof}
	Denote $\mathcal{T} = \mathcal{T}(\mathcal{C}_M;M)$ and $\mathcal{T}' = \mathcal{T}([0,1);2)$.  We must construct a bijective sticky map $\psi: \mathcal{T}\rightarrow\mathcal{T}'$.  First, define $\psi(v_0) = v'_0$, where $v_0$ is the root of $\mathcal{T}$ and $v'_0$ is the root of $\mathcal{T}'$.

Now, for any $k\geq 1$, consider the vertex $\langle i_1,i_2,\ldots,i_k\rangle\in\mathcal{T}$.  We know that $i_j\in\mathbb{Z}_M$ for all $j$.  Furthermore, for any fixed $j$, this vertex corresponds to a $k$th level subinterval of $\mathcal{C}_M^{[k]}$.  Every such $k$-th level interval is replaced by exactly two arbitrary $(k+1)$-th level subintervals in the construction of $\mathcal{C}_M^{[k+1]}$.  Therefore, there exists $N_1 := N_1(\langle i_1,\ldots,i_k\rangle),$ $N_2 := N_2(\langle i_1,\ldots,i_k\rangle)\in\mathbb{Z}_M$, with $N_1<N_2$, such that $\langle i_1,\ldots,i_k,i_{k+1}\rangle\in\mathcal{T}$ if and only if $i_{k+1}=N_1$ or $N_2$.  Consequently, we define 
\begin{equation}\label{Cantor1}
\psi(\langle i_1,i_2,\ldots,i_k\rangle) = \langle l_1,l_2,\ldots,l_k\rangle\in\mathcal{T}',
\end{equation}
where
\begin{equation}\label{Cantor2}
l_{j+1} = \begin{cases} 0 &\text{ if } i_{j+1} = N_1(\langle i_1,\ldots,i_j\rangle),\\
	1 &\text{ if } i_{j+1} = N_2(\langle i_1,\ldots,i_j\rangle).\nonumber
	\end{cases}
\end{equation}
The mapping $\psi$ is injective by construction and surjectivity follows from the binary selection of subintervals at each stage in the construction of $\mathcal{C}_M$.  Moreover, $\psi$ is sticky by \eqref{Cantor1}.
\end{proof}
The following corollary is an easy consequence of the above and left to the reader. 
\begin{corollary} \label{corollary - finite binary structure}
Recall the definition of $\mathcal D_M^{[N]}$ from Section \ref{preliminary construction subsection}. Then $$\mathcal T_N(\mathcal D_M^{[N]};M) \cong \mathcal T_N([0,1);2).$$
\end{corollary} 

Proposition~\ref{prop - binary structure} and Corollary \ref{corollary - finite binary structure} guarantee that the tree encoding our set of directions will retain a certain binary structure.  This fact will prove vital to establishing Theorem~\ref{main theorem}.

%%%%%%%%%%%%%%%%%%%%%%%%%%%%%%%%%%%%%%%%%%%%%%%%%%%%%%%%%%%%%%%%%%%%%%%

\subsection{Encoding higher dimensional bounded subsets of Euclidean space by trees}\label{tree encoding higher dimensions}

The approach to encoding a bounded subset of Euclidean space by a tree extends readily to higher dimensions.  For any $\mathbf{i} = \langle j_1,\ldots,j_d\rangle\in\mathbb{Z}^d$ such that $\mathbf{i}\cdot M^{-k}\in [0,1)^d$, we can apply \eqref{M-adic representation} to each component of $\mathbf{i}$ to obtain $$\frac {\mathbf{i}}{M^k} = \frac {\mathbf{i}_1}{M} + \frac {\mathbf{i}_2}{M^2} + \cdots + \frac {\mathbf{i}_k}{M^k},$$ with $\mathbf{i}_j\in\mathbb{Z}_M^d$ for all $j$.  As before, we identify $\mathbf{i}$ with $\langle \mathbf{i}_1,\ldots,\mathbf{i}_k\rangle$.

Let $\phi : \mathbb{Z}_M^d \rightarrow \{0,1,\ldots, M^d-1\}$ be an enumeration of $\mathbb{Z}_M^d$.  Define the full $M^d$-adic tree
\begin{equation}\label{tree encoding}
	\mathcal{T}([0,1)^d;M,\phi) = \left\{\langle \phi(\mathbf{i}_1),\ldots,\phi(\mathbf{i}_k)\rangle : k\geq 0,\ \mathbf{i}_j \in\mathbb{Z}_M^d\right\}.
\end{equation}
The collection of $k$th generation vertices of this tree may be thought of as the $d$-fold Cartesian product of the $k$th generation vertices of $\mathcal{T}([0,1);M)$.  For our purposes, it will suffice to fix $\phi$ to be the lexicographic ordering, and so we will omit the notation for $\phi$ in \eqref{tree encoding}, writing simply, and with a slight abuse of notation, 
\begin{equation}\label{better tree encoding}
	\mathcal{T}([0,1)^d;M) = \left\{\langle \mathbf{i}_1,\ldots,\mathbf{i}_k\rangle : k\geq 0,\ \mathbf{i}_j \in\mathbb{Z}_M^d\right\}.
\end{equation}
As before, we will refer to the tree in \eqref{better tree encoding} by the notation $\mathcal{T}([0,1)^d)$ once the base $M$ has been fixed.

By a direct generalization of our one-dimensional results, each vertex $\langle \mathbf{i}_1,\ldots,\mathbf{i}_k\rangle$ of $\mathcal{T}([0,1)^d;M)$ at height $k$ represents the unique $M$-adic cube in $[0,1)^d$ of sidelength $M^{-k}$, containing $\mathbf{i}\cdot M^{-k}$, of the form $$\left[\frac {j_1}{M^k},\frac {j_1+1}{M^k}\right)\times\cdots\times \left[\frac {j_d}{M^k},\frac {j_d+1}{M^k}\right).$$  As in the one-dimensional setting, any $x\in[0,1)^d$ can be realized as the intersection of a nested sequence of $M$-adic cubes.  Thus, we view the tree in \eqref{better tree encoding} as an encoding of the set $[0,1)^d$ with respect to base $M$.  As before, any subset $E\subseteq[0,1)^d$ then corresponds to a subtree of $\mathcal{T}([0,1)^d;M)$. 

The connection between sets and trees encoding them leads to the following easy observations that we record for future use in Lemma~\ref{Resistance}. 
\begin{lemma} \label{lemma Omega_N vertex count}
Let $\Omega_N$ be the set defined in \eqref{defn Omega_N}. 
\begin{enumerate}[(a)]
\item Given $\Omega_N$, there is a constant $C_1 > 0$ (depending only on $d$ and $C,c$ from \eqref{Lipschitz}) such that for any $1 \leq k \leq N$, the number of $k$th generation vertices in $\mathcal T_N(\Omega_N;M)$ is  $\leq C_1 2^k$.  
\item \label{affine Omega_N} For any compact set $\mathbb K \subseteq \mathbb R^{d+1}$, there exists a constant $C(\mathbb K) > 0$ with the following property. For any $x = (x_1, \cdots, x_{d+1}) \in \mathbb K$, and $1 \leq k \leq N$, the number of $k$th generation vertices in $\mathcal T_N(E(x);M)$ is $\leq C(\mathbb K) 2^k$, where $E(x) := (x - x_1 \Omega_N) \cap \{0\} \times [0,1)^d$. 
\end{enumerate}
\end{lemma} 
\begin{proof}
There are exactly $2^k$ basic intervals of level $k$ that comprise $\mathcal C_M^{[k]}$. Under $\gamma$, each such basic interval maps into a set of diameter at most $CM^{-k}$. Since $\Omega_N = \gamma(\mathcal D_M^{[N]}) \subseteq \gamma(\mathcal C_M^{[k]})$, the number of $k$th generation vertices  in $\mathcal T_N(\Omega_N;M)$, which is also the number of $k$th level $M$-adic cubes needed to cover $\Omega_N$, is at most $C_1 2^k$.  This proves (a).   

Let $Q$ be any $k$th generation $M$-adic cube such that $Q \cap \Omega_N \ne \emptyset$. Then on one hand, $(x - x_1 Q) \cap (x - x_1 \Omega_N) \ne \emptyset$; on the other hand, the number of $k$th level $M$-adic cubes covering $(x - x_1 Q)$ is $\leq C(\mathbb K)$, and part (\ref{affine Omega_N}) follows.  
\end{proof} 

{\em{Notation: }} We end this section with a notational update. In light of the discussion above and for simplicity, we will henceforth identify a vertex $u = \langle i_1, i_2, \cdots, i_k \rangle  \in \mathcal T([0,1)^d)$ with the corresponding cube $\{ 0\} \times u$ lying on the root hyperplane $\{ 0 \} \times [0,1)^d$. In this parlance, a vertex $t \in \mathcal T_N([0,1)^d)$ of height $N$ is the same as a root cube $Q_t$ (or $\widetilde{Q}_t$) defined in \eqref{defn Q_t}, and the notation $t \subseteq u$ stands both for set containment as well as tree ancestry. 
%%%%%%%%%%%%%%%%%%%%%%%%%%%%%%%%%%%%%%%%%%%%%%%%%%%%%%%%%%%%%%%%%%%%%%%

\section{Electrical circuits and percolation on trees}\label{percolation section}

\subsection{The percolation process associated to a tree}\label{percolation}

The proof of Theorem~\ref{main theorem} will require consideration of a special probabilistic process on certain trees called a (bond) percolation.  Imagine a liquid that is poured on top of some porous material.  How will the liquid flow - or \textit{percolate} - through the holes of the material?  How likely is it that the liquid will flow from hole to hole in at least one uninterrupted path all the way to the bottom?  The first question forms the intuition behind a formal percolation process, whereas the second question turns out to be of critical importance to the proof of Theorem~\ref{main theorem}; this idea plays a key role in establishing the planar analogue of that theorem in Bateman and Katz~\cite{BatemanKatz}, and again in the more general framework of~\cite{Bateman}.

Although it is possible to speak of percolation processes in far more general terms (see~\cite{Grimmett}), we will only be concerned with a percolation process on a tree.  Accordingly, given some tree $\mathcal{T}$ with vertex set $\mathcal{V}$ and edge set $\mathcal{E}$, we define an \textit{edge-dependent Bernoulli (bond) percolation process} to be any collection of random variables $\{X_e : e\in\mathcal{E}\}$, where $X_e$ is Bernoulli$(p_e)$ with $p_e<1$.  The parameter $p_e$ is called the \textit{survival probability of the edge} $e$.  We will always be concerned with a particular type of percolation on our trees: we define a \textit{standard Bernoulli}$(p)$ \textit{percolation} to be one where the random variables $\{X_e : e\in\mathcal{E}\}$ are mutually independent and identically distributed Bernoulli$(p)$ random variables, for some $p<1$.  In fact, for our purposes, it will suffice to consider only standard Bernoulli$(\frac 12)$ percolations. 

Rather than imagining a tree with a percolation process as the behaviour of a liquid acted upon by gravity in a porous material, it will be useful to think of the percolation process as acting more directly on the mathematical object of the tree itself.  Given some percolation process on a tree $\mathcal{T}$, we will think of the event $\{X_e=0\}$ as the event that we \textit{remove} the edge $e$ from the edge set $\mathcal{E}$, and the event $\{X_e=1\}$ as the event that we \textit{retain} this edge; denote the random set of retained edges by $\mathcal{E}^*$.  Notice that with this interpretation, after percolation there is no guarantee that $\mathcal{E}^*$, the subset of edges that remain after percolation, defines a subtree of $\mathcal{T}$.  In fact, it can be quite likely that the subgraph that remains after percolation is a union of many disconnected subgraphs of $\mathcal{T}$.

For a given edge $e\in\mathcal{E}$, we think of $p = \text{Pr}(X_e=1)$ as the probability that we retain this edge after percolation.  The probability that at least one uninterrupted path remains from the root of the tree to its bottommost level is given by the \textit{survival probability} of the corresponding percolation process.  More explicitly, given a percolation on a tree $\mathcal{T}$, the survival probability after percolation is the probability that the random variables associated to all edges of at least one ray in $\mathcal{T}$ take the value $1$; i.e.
\begin{equation}\label{survival probability}
	\text{Pr}\left(\text{survival after percolation on } \mathcal{T}\right) := \text{Pr}\left(\bigcup_{R\in\partial\mathcal{T}}\bigcap_{e\in\mathcal{E}\cap R} \{X_e=1\}\right).
\end{equation}
Estimation of this probability will prove to be a valuable tool in the proof of Theorem~\ref{main theorem}.  This estimation will require reimagining a tree as an electrical network.

%%%%%%%%%%%%%%%%%%%%%%%%%%%%%%%%%%%%%%%%%%%%%%%%%%%%%%%%%%%%%%%%%%%%%%%

\subsection{Trees as electrical networks}\label{electrical networks}

Formally, an electrical network is a particular kind of weighted graph.  The weights of the edges are called \textit{conductances} and their reciprocals are called \textit{resistances}.  In his seminal works on the subject, Lyons visualizes percolation on a tree as a certain electrical network.  In~\cite{Lyons1}, he lays the groundwork for this correspondence.  While his results hold in great generality, we describe his results in the context of standard Bernoulli percolation on a locally finite, rooted labelled tree only.  We briefly review the concepts relevant to our application here.

A percolation process on the truncation of any given tree $\mathcal{T}$ is naturally associated to a particular electrical network.  To see this, we truncate the tree $\mathcal{T}$ at height $N$ and place the positive node of a battery at the root of $\mathcal{T}_N$.  Then, for every ray in $\partial\mathcal{T}_N$, there is a unique terminating vertex; we connect each of these vertices to the negative node of the battery.  A resistor is placed on every edge $e$ of $\mathcal{T}_N$ with resistance $R_e$ defined by
\begin{equation}\label{resistance}
	\frac 1{R_e} = \frac 1{1-p_e}\prod_{\emptyset\subset v(e')\subseteq v(e)} p_{e'}.
\end{equation}
Notice that the resistance for the edge $e$ is essentially the reciprocal of the probability that a path remains from the root of the tree to the vertex $v(e)$ after percolation.  For standard Bernoulli$(\frac 12)$ percolation, we have 
\begin{equation}\label{resistances}
	R_e = 2^{h(v(e))-1}.
\end{equation}

One fact that will prove useful for us later is that connecting any two vertices at a given height by an ideal conductor (i.e. one with zero resistance) only decreases the overall resistance of the circuit.  This will allow us to more easily estimate the total resistance of a generic tree.  
\begin{proposition}\label{resistance prop}
	Let $\mathcal{T}_N$ be a truncated tree of height $N$ with corresponding electrical network generated by a standard Bernoulli$(\frac 12)$ percolation process.  Suppose at height $k<N$ we connect two vertices by a conductor with zero resistance.  Then the resulting electrical network has a total resistance no greater than that of the original network.
\end{proposition}

\begin{proof}
	Let $u$ and $v$ be the two vertices at height $k$ that we will connect with an ideal conductor.  Let $R_1$ denote the resistance between $u$ and $D(u,v)$, the youngest common ancestor of $u$ and $v$; let $R_2$ denote the resistance between $v$ and $D(u,v)$.  Let $R_3$ denote the total resistance of the subtree of $\mathcal{T}_N$ generated by the root $u$ and let $R_4$ denote the total resistance of the subtree of $\mathcal{T}_N$ generated by the root $v$.  These four connections define a subnetwork of our tree, depicted in Figure~\ref{subnetworks}(a).  The connection of $u$ and $v$ by an ideal conductor, as pictured in Figure~\ref{subnetworks}(b), can only change the total resistance of this subnetwork, as that action leaves all other connections unaltered. It therefore suffices to prove that the total resistance of the subnetwork comprising of the resistors $R_1$, $R_2$, $R_3$ and $R_4$ can only decrease if $u$ and $v$ are joined by an ideal conductor. 

\begin{figure}[h]
\setlength{\unitlength}{0.8mm}
\begin{picture}(-50,5)(-106,5)

        \put(-65,-70){\shortstack{$(a)$}} 

	\special{sh 0.99}\put(-60,0){\ellipse{2}{2}}
	\put(-60,3){\shortstack{$D(u,v)$}}
	\special{sh 0.99}\put(-75,-30){\ellipse{2}{2}}
	\put(-81,-31){\shortstack{$u$}}
	\special{sh 0.99}\put(-45,-30){\ellipse{2}{2}}
	\put(-41,-31){\shortstack{$v$}}
	\special{sh 0.99}\put(-60,-60){\ellipse{2}{2}}

	\path(-60,0)(-90,0)(-90,-45)
	\path(-60,-60)(-90,-60)(-90,-50)
	\path(-95,-45)(-85,-45)
	\path(-92,-50)(-88,-50)
	\put(-97,-41){\large\shortstack{$+$}}
	\put(-97,-55){\large\shortstack{$-$}} 

	\path(-60,0)(-60,-5)
	\path(-75,-13)(-75,-5)(-45,-5)(-45,-13)
	\path(-75,-23)(-75,-37)
	\path(-45,-23)(-45,-37)
	\path(-75,-47)(-75,-55)(-45,-55)(-45,-47)
	\path(-60,-60)(-60,-55)

	\path(-75,-13)(-77,-14)(-73,-15)(-77,-16)(-73,-17)(-77,-18)(-73,-19)(-77,-20)(-73,-21)(-77,-22)(-75,-23)
	\path(-45,-13)(-47,-14)(-43,-15)(-47,-16)(-43,-17)(-47,-18)(-43,-19)(-47,-20)(-43,-21)(-47,-22)(-45,-23)
	\path(-75,-37)(-77,-38)(-73,-39)(-77,-40)(-73,-41)(-77,-42)(-73,-43)(-77,-44)(-73,-45)(-77,-46)(-75,-47)
	\path(-45,-37)(-47,-38)(-43,-39)(-47,-40)(-43,-41)(-47,-42)(-43,-43)(-47,-44)(-43,-45)(-47,-46)(-45,-47)
	\put(-70,-20){\large\shortstack{$R_1$}} 
	\put(-40,-20){\large\shortstack{$R_2$}}
	\put(-70,-43){\large\shortstack{$R_3$}}
	\put(-40,-43){\large\shortstack{$R_4$}}

%%%%%%%%%%%%%%%%

        \put(25,-70){\shortstack{$(b)$}}

	\special{sh 0.99}\put(30,0){\ellipse{2}{2}}
	\put(30,3){\shortstack{$D(u,v)$}}
	\special{sh 0.99}\put(30,-30){\ellipse{2}{2}}
	\put(33,-31){\shortstack{$u \sim v$}}
	\special{sh 0.99}\put(30,-60){\ellipse{2}{2}}

	\path(30,0)(0,0)(0,-45)
	\path(30,-60)(0,-60)(0,-50)
	\path(-5,-45)(5,-45)
	\path(-2,-50)(2,-50)
	\put(-7,-41){\large\shortstack{$+$}}
	\put(-7,-55){\large\shortstack{$-$}} 

	\path(30,0)(30,-5)
	\path(15,-10)(15,-5)(45,-5)(45,-10)
	\path(30,-30)(30,-25)
	\path(15,-20)(15,-25)(45,-25)(45,-20)
	\path(30,-30)(30,-35)
	\path(15,-40)(15,-35)(45,-35)(45,-40)
	\path(15,-50)(15,-55)(45,-55)(45,-50)
	\path(30,-60)(30,-55) 

	\path(15,-10)(17,-11)(13,-12)(17,-13)(13,-14)(17,-15)(13,-16)(17,-17)(13,-18)(17,-19)(15,-20)
	\path(45,-10)(47,-11)(43,-12)(47,-13)(43,-14)(47,-15)(43,-16)(47,-17)(43,-18)(47,-19)(45,-20)
	\path(15,-40)(17,-41)(13,-42)(17,-43)(13,-44)(17,-45)(13,-46)(17,-47)(13,-48)(17,-49)(15,-50)
	\path(45,-40)(47,-41)(43,-42)(47,-43)(43,-44)(47,-45)(43,-46)(47,-47)(43,-48)(47,-49)(45,-50)
	\put(20,-17){\large\shortstack{$R_1$}} 
	\put(50,-17){\large\shortstack{$R_2$}}
	\put(20,-47){\large\shortstack{$R_3$}}
	\put(50,-47){\large\shortstack{$R_4$}}

\end{picture}
\vspace{6.25cm}
\caption{(a) The original subnetwork with the resistors $R_1$, $R_3$ and $R_2$, $R_4$ in series; (b) the new subnetwork obtained by connecting vertices $u$ and $v$ by an ideal conductor.}\label{subnetworks}
\end{figure}
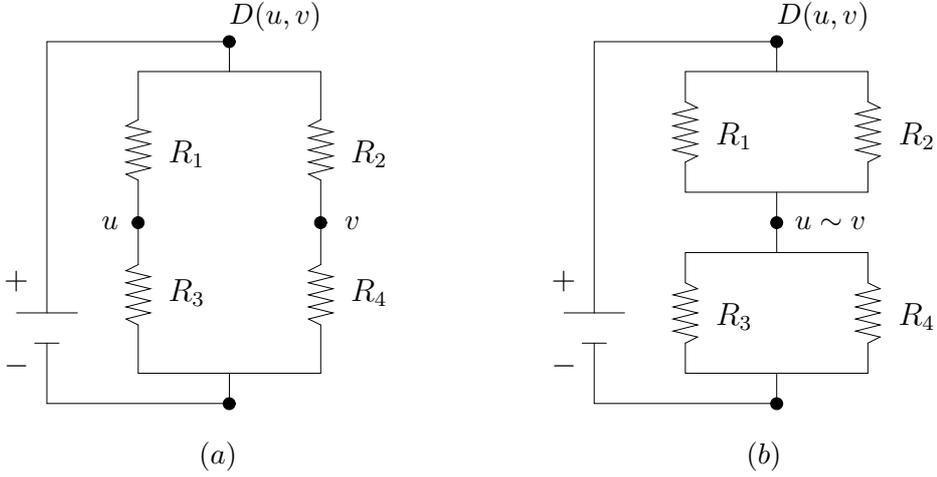

In the original subnetwork, the resistors $R_1$ and $R_3$ are in series, as are the resistors $R_2$ and $R_4$.  These pairs of resistors are also in parallel with each other.  Thus, we calculate the total resistance of this subnetwork, $R_{\text{original}}$:
\begin{align}\label{original resistance}
	R_{\text{original}} &= \left(\frac 1{R_1+R_3} + \frac 1{R_2+R_4}\right)^{-1}\nonumber\\
	&= \frac {(R_1+R_3)(R_2+R_4)}{R_1+R_2+R_3+R_4}.
\end{align}
After connecting vertices $u$ and $v$ by an ideal conductor, the structure of our subnetwork is inverted as follows.  The resistors $R_1$ and $R_2$ are in parallel, as are the resistors $R_3$ and $R_4$, and these pairs of resistors are also in series with each other.  Therefore, we calculate the new total resistance of this subnetwork, $R_{\text{new}}$, as
\begin{align}\label{new resistance}
	R_{\text{new}} &= \left(\frac 1{R_1} + \frac 1{R_2}\right)^{-1} + \left(\frac 1{R_3} + \frac 1{R_4}\right)^{-1}\nonumber\\
	&= \frac {R_1R_2(R_3+R_4) + R_3R_4(R_1+R_2)}{(R_1+R_2)(R_3+R_4)}.
\end{align}
We claim that \eqref{original resistance} is greater than or equal to \eqref{new resistance}.  To see this, simply cross-multiply these expressions.  After cancellation of common terms, our claim reduces to $$R_1^2R_4^2 + R_2^2R_3^2 \geq 2R_1R_2R_3R_4.$$  But this is trivially satisfied since $(a-b)^2\geq 0$ for any real numbers $a$ and $b$.
\end{proof}

%%%%%%%%%%%%%%%%%%%%%%%%%%%%%%%%%%%%%%%%%%%%%%%%%%%%%%%%%%%%%%%%%%%%%%%

\subsection{Estimating the survival probability after percolation}\label{survival probability section}

We now present Lyons' pivotal result linking the total resistance of an electrical network and the survival probability under the associated percolation process.  

\begin{theorem}[Lyons, Theorem 2.1 of~\cite{Lyons2}]\label{Lyons theorem}
	Let $\mathcal{T}$ be a tree with mutually associated percolation process and electrical network, and let $R(\mathcal{T})$ denote the total resistance of this network.  If the percolation is Bernoulli, then $$\frac 1{1+R(\mathcal{T})} \leq \text{Pr}(\mathcal{T}) \leq \frac 2{1+R(\mathcal{T})},$$ where $\text{Pr}(\mathcal{T})$ denotes the survival probability after percolation on $\mathcal{T}$.
\end{theorem}
We will not require the full strength of this theorem.  A reasonable upper bound on the survival probability coupled with the result of Proposition~\ref{resistance prop} will suffice for our applications.  For completeness, we state and prove a sufficient simpler version of Theorem~\ref{Lyons theorem} as essentially formulated by Bateman and Katz~\cite{BatemanKatz}.

\begin{proposition}\label{survival prob}
	Let $M\geq 2$ and let $\mathcal{T}$ be a subtree of a full $M$-adic tree.  Let $R(\mathcal{T})$ and $\text{Pr}(\mathcal{T})$ be as in Theorem~\ref{Lyons theorem}.  Then under Bernoulli percolation, we have
\begin{equation}\label{survival}
	\text{Pr}(\mathcal{T}) \leq \frac 2{1+R(\mathcal{T})}.
\end{equation}
\end{proposition}

\begin{proof}
We will only focus on the case when $R(\mathcal T) \geq 1$, since otherwise (\ref{survival}) holds trivially.  We prove this by induction on the height of the tree $N$.  When $N=0$, then \eqref{survival} is trivially satisfied.  Now suppose that up to height $N-1$, we have $$\text{Pr}(\mathcal{T}) \leq \frac 2{1+R(\mathcal{T})}.$$  

Suppose $\mathcal{T}$ is of height $N$.  We can view the tree $\mathcal{T}$ as its root together with at most $M$ edges connecting the root to the subtrees $\mathcal{T}_1,\ldots,\mathcal{T}_M$ of height $N-1$ generated by the terminating vertices of these edges.  If there are $k < M$ edges originating from the root, then we take $M-k$ of these subtrees to be empty.  Note that by the induction hypothesis, (\ref{survival}) holds for each $\mathcal T_j$. To simplify notation, we denote $$\text{Pr}(\mathcal{T}_j) = P_j\text{ and }R(\mathcal{T}_j) = R_j,$$ taking $P_j=0$ and $R_j=\infty$ if $\mathcal{T}_j$ is empty.  

Using independence and recasting $\text{Pr}(\mathcal{T})$ as one minus the probability of \textit{not} surviving after percolation on $\mathcal{T}$, we have the formula:
\begin{equation}\label{prob recursion}
	\text{Pr}(\mathcal{T}) = 1 - \prod_{k=1}^M\left(1 - \frac 12 P_k\right).\nonumber
\end{equation}
Note that the function $F(x_1,\ldots,x_M) = 1 - (1-x_1/2)(1-x_2/2)\cdots(1-x_M/2)$ is monotone increasing in each variable on $[0,2]^M$.  Now define $$Q_j := \frac 2{1+R_j}.$$  Since resistances are nonnegative, we know that $Q_j\leq 2$ for all $j$.  Therefore, 
\begin{align}
	\text{Pr}(\mathcal{T}) &= F(P_1,\ldots,P_M)\nonumber\\
	&\leq F(Q_1,\ldots,Q_M)\nonumber\\
	&\leq \frac 12 \sum_{k=1}^M Q_k.\nonumber
\end{align}
Here, the first inequality follows by monotonicity and the induction hypothesis.  Plugging in the definition of $Q_k$, we find that $$\text{Pr}(\mathcal{T}) \leq \sum_{k=1}^M \frac 1{1+R_k}.$$  But since each resistor $R_j$ is in parallel, we know that $$\frac 1{R(\mathcal{T})} = \sum_{k=1}^M\frac 1{1 + R_k}.$$  Combining this formula with the previous inequality and recalling that $R(\mathcal T) \geq 1$, we have $$\text{Pr}(\mathcal{T}) \leq \frac 1{R(\mathcal{T})} \leq \frac 2{1+R(\mathcal{T})},$$ as required.

\end{proof}

%%%%%%%%%%%%%%%%%%%%%%%%%%%%%%%%%%%%%%%%%%%%%%%%%%%%%%%%%%%%%%%%%%%%%%%

\section{The random mechanism and the property of stickiness}\label{stickiness section}

As discussed in the introduction of this paper, the construction of a Kakeya-type set with orientations given by $\Omega$ will require a certain random mechanism.  We now describe this mechanism in detail.

In order to assign a slope $\sigma(\cdot)$ to the tubes $P_{t,\sigma}:= \mathcal P_{t, \sigma(t)}$ given by \eqref{defn-Ptsigma}, we want to define a collection of random variables $\{X_{\langle i_1,\ldots, i_k\rangle} : \langle i_1,\ldots,i_k\rangle\in \mathcal{T}([0,1)^d)\}$, one on each edge of the tree used to identify the roots of these tubes.  The tree $\mathcal{T}_1( [0,1)^d)$ consists of all first generation edges of $\mathcal{T}([0,1)^d)$.  It has exactly $M^d$ many edges and we place (independently) a Bernoulli$(\frac 12)$ random variable on each edge: $X_{\langle 0\rangle},X_{\langle 1\rangle},\ldots,X_{\langle M^d-1\rangle}$.  Now, the tree $\mathcal{T}_2( [0,1)^d)$ consists of all first and second generation edges of $\mathcal{T}([0,1)^d)$.  It has $M^d+M^{2d}$ many edges and we place (independently) a new Bernoulli$(\frac 12)$ random variable on each of the $M^{2d}$ second generation edges. We label these $X_{\langle i_1,i_2\rangle}$ where $0\leq i_1,i_2<M^d$.  We proceed in this way, eventually assigning an ordered collection of independent Bernoulli$(\frac 12)$ random variables to the tree $\mathcal{T}_N([0,1)^d)$: $$\mathbb{X}_N :=  \left\{X_{\langle i_1,\ldots,i_k\rangle} : \langle i_1,\ldots,i_k\rangle\in \mathcal{T}_N([0,1)^d),\ 1\leq k\leq N\right\},$$ where $X_{\langle i_1,\ldots,i_k\rangle}$ is assigned to the unique edge identifying $\langle i_1, i_2, \cdots, i_k\rangle$, namely the edge joining $\langle i_1, i_2, \cdots, i_{k-1} \rangle$ to $\langle i_1,i_2,\ldots,i_k\rangle$. Each realization of $\mathbb X_N$ is a finite ordered collection of cardinality $M^d + M^{2d} + \cdots + M^{Nd}$ with entries either $0$ or $1$. 

We will now establish that every realization of the random variable $\mathbb{X}_N$ defines a sticky map between the truncated position tree $\mathcal{T}_N([0,1)^d)$ and the truncated {\em{binary}} tree $\mathcal{T}_N([0,1);2)$, as defined in Definition~\ref{D:stickiness}.  Fix a particular realization $\mathbb{X}_N = \mathbf x = \{ x_{\langle i_1, \cdots, i_k\rangle}\}$.  Define a map $\tau_{\mathbf x} : \mathcal{T}_N([0,1)^d)\rightarrow \mathcal{T}_N([0,1);2)$, where 
\begin{equation}\label{tau map}
	\tau_{\mathbf x}(\langle i_1,i_2,\ldots,i_k\rangle) = \left\langle x_{\langle i_1\rangle},x_{\langle i_1,i_2\rangle},\ldots, x_{\langle i_1,i_2,\ldots,i_k\rangle}\right\rangle.
\end{equation}
We then have the following key proposition.

\begin{proposition}\label{sticky assignment}
	The map $\tau_{\mathbf x}$ just defined is sticky for every realization $\mathbf x$ of $\mathbb X_N$.  Conversely, any sticky map $\tau$ between $\mathcal{T}_N([0,1)^d)$ and $\mathcal{T}_N([0,1); 2)$ can be written as $\tau = \tau_{\mathbf x}$ for some realization $\mathbf x$ of $\mathbb X_N$.  
\end{proposition}

\begin{proof} 
Recalling Definition~\ref{D:stickiness}, we need to verify that $\tau_{\mathbf x}$ preserves heights and lineages. By \eqref{tau map}, any finite sequence $v = \langle i_1, i_2, \cdots, i_k \rangle$ in $\mathcal T([0,1)^d)$ is mapped to a sequence of the same length in $\mathcal T([0,1);2)$. Therefore $h(v) = h(\tau_{\mathbf x}(v))$ for every $v \in \mathcal T([0,1)^d)$. 
	
Next suppose $u\supset v$. Then $u = \langle i_1,\ldots,i_{h(u)}\rangle$, with $h(u) \leq k$.  So again by \eqref{tau map}, $$\tau_{\mathbf x}(u) = \left\langle x_{\langle i_1\rangle},\ldots,x_{\langle i_1,\ldots,i_{h(u)}\rangle}\right\rangle \supset\left\langle x_{\langle i_1\rangle},\ldots,x_{\langle i_1,\ldots,i_{h(u)}\rangle},\ldots,x_{\langle i_1,\ldots,i_k\rangle}\right\rangle = \tau_{\mathbf x}(v).$$  Thus, $\tau_{\mathbf x}$ preserves lineages, establishing the first claim in Proposition~\ref{sticky assignment}.

For the second, fix a sticky map $\tau :  \mathcal{T}_N([0,1)^d) \rightarrow \mathcal{T}_N([0,1); 2)$. Define $x_{\langle i_1 \rangle} := \tau(\langle i_1 \rangle)$, $x_{\langle i_1, i_2\rangle} := \pi_2 \circ \tau(\langle i_1, i_2 \rangle)$, and in general 
\[ x_{\langle i_1, \cdots, i_k\rangle} := \pi_k \circ \tau(\langle i_1, i_2, \cdots, i_k \rangle), \qquad k \geq 1, \] 
where $\pi_k$ denotes the projection map whose image is the $k$th coordinate of the input sequence. The collection $\mathbf x = \{x_{\langle i_1, i_2, \cdots, i_k\rangle} \}$ is the unique realization of $\mathbb X_N$ that verifies the second claim.  
\end{proof}

\subsection{Slope assignment algorithm} 
Recall from Sections \ref{results} and \ref{preliminary construction subsection} that $\Omega := \gamma(\mathcal C_M)$ and $\Omega_N := \gamma(\mathcal D_{M}^{[N]})$, where $\mathcal C_M$ is the generalized Cantor-type set and $\mathcal D_M^{[N]}$ a finitary version of it. 
In order to exploit the binary structure of the trees $\mathcal{T}(\mathcal{C}_M) := \mathcal{T}(\mathcal{C}_M;M)$ and $\mathcal{T}(\mathcal{D}^{[N]}_M) := \mathcal{T}(\mathcal{D}^{[N]}_M;M)$ advanced in Proposition~\ref{prop - binary structure} and Corollary~\ref{corollary - finite binary structure}, we need to map traditional binary sequences onto the subsequences of $\{0,\ldots,M-1\}^{\infty}$ defined by $\mathcal{C}_M$ or $\mathcal D_{M}^{[N]}$.  

\begin{proposition}\label{slope assignment}
	Every sticky map $\tau$ as in \eqref{tau map} that maps $\mathcal{T}_N([0,1)^d;M)$ to $\mathcal{T}_N([0,1);2)$ induces a natural mapping $\sigma = \sigma_{\tau}$ from $\mathcal{T}_N([0,1)^d)$ into $\Omega_N$. The maps $\sigma_\tau$ obey a uniform Lipschitz-type condition: for any $t, t' \in \mathcal T_N([0,1)^d)$, $t \ne t'$, 
\begin{equation} \label{sigma Lipschitz-type}
\bigl| \sigma_{\tau}(t) - \sigma_{\tau}(t') \bigr| \leq CM^{-h(D(\tau(t),\tau(t')))},    
\end{equation}  
where $C$ is as in \eqref{Lipschitz}.
\end{proposition}

\noindent {\em{Remark}}: While the choice of $\mathcal D_M^{[N]}$ for a given $\mathcal C_M^{[N]}$ is not unique, the mapping  $\tau \mapsto \sigma_{\tau}$ is unique given a specific choice. Moreover, if $\mathcal D_M^{[N]}$ and $\overline{\mathcal D}_M^{[N]}$ are two selections of finitary direction sets at scale $M^{-N}$, then the corresponding maps $\sigma_{\tau}$ and $\overline{\sigma}_{\tau}$ must obey
\begin{equation} \bigl|\sigma_{\tau}(v) - \overline{\sigma}_{\tau}(v)\bigr| \leq C M^{-h(v)} \quad \text{ for every } v \in \mathcal T_N([0,1)^d), \label{error-slope}\end{equation} 
where $C$ is as in \eqref{Lipschitz}.  Thus given $\tau$, the slope in $\Omega$ that is assigned by $\sigma_{\tau}$ to an $M$-adic cube in $\{0\} \times [0,1)^d$ of sidelength $M^{-N}$ is unique up to an error of $O(M^{-N})$. As a consequence $P_{t, \sigma_\tau}$ and $P_{t, \overline{\sigma}_{\tau}}$ are comparable, in the sense that each is contained in a $O(M^{-N})$-thickening of the other.

\begin{proof}
There are two links that allow passage of $\tau$ to $\sigma$. The first of these is the isomorphism $\psi$ constructed in  Proposition~\ref{prop - binary structure} that maps $\mathcal T(\mathcal C_M;M)$ onto $\mathcal T([0,1);2)$. Under this isomorphism, the pre-image of any $k$-long sequence of 0's and 1's is a vertex $w$ of height $k$ in $\mathcal T(\mathcal C_M;M)$, in other words one of the $2^k$ chosen $M$-adic intervals of length $M^{-k}$ that constitute  $\mathcal C_M^{[k]}$. The second link is a mapping $\Phi: \mathcal T_N(\mathcal C_M;M) \rightarrow \mathcal D_M^{[N]}$ that sends every vertex $w$ to a point in $\mathcal C_M \cap w$, where, per our notational agreement at the end of Section \ref{tree section}, we have also let $w$ denote the particular $M$-adic interval that it identifies. While the choice of the image point, i.e., $\mathcal D_{M}^{[N]}$ is not unique, any two candidates $\Phi$, $\overline{\Phi}$ satisfy
\begin{equation} \label{Phi_k error} |\Phi(w) - \overline{\Phi}(w) \bigr| \leq \text{diam}(w) = M^{-h(w)} \quad \text{ for every } w \in \mathcal T_N(\mathcal C_M;M).  
\end{equation}  
 
We are now ready to describe the assignment $\tau \mapsto \sigma = \sigma_{\tau}$. Given a sticky map $\tau:\mathcal T_N([0,1)^d;M) \rightarrow \mathcal T_N([0,1);2)$ such that
\[ \tau(\langle i_1, i_2, \cdots, i_k \rangle) = \langle X_{\langle i_1\rangle}, \cdots, X_{\langle i_1, i_2, \cdots, i_k\rangle}\rangle, \]  the transformed random variable
\begin{equation}\label{slope random variable}
	Y_{\langle i_1,i_2\ldots,i_k\rangle} := \gamma\circ\Phi \circ \psi^{-1} \left(\langle X_{\langle i_1\rangle}, X_{\langle i_1,i_2\rangle},\ldots, X_{\langle i_1,i_2,\ldots,i_k\rangle}\rangle\right)\nonumber
\end{equation}
 associates a random direction in $\Omega_N = \gamma(\mathcal{D}^{[N]}_M)$ to the sequence $t = \langle i_1,\ldots,i_k\rangle$ identified with a unique vertex  $t \in \mathcal{T}_N([0,1)^d)$.  Thus, defining
\begin{equation}\label{sticky direction assignment}
	\sigma := \gamma\circ\Phi\circ\psi^{-1} \circ \tau
\end{equation}
gives the appropriate (random) mapping claimed by the proposition. The weak Lipschitz condition (\ref{sigma Lipschitz-type}) is verified as follows,
\begin{align*}
\bigl| \sigma_{\tau}(t) - \sigma_{\tau}(t') \bigr| &= \bigl|\gamma\circ\Phi\circ\psi^{-1} \circ \tau(t) - \gamma\circ\Phi\circ\psi^{-1} \circ \tau(t') \bigr| \\ &\leq C \bigl| \Phi\circ\psi^{-1} \circ \tau(t) - \Phi\circ\psi^{-1} \circ \tau(t') \bigr|   \\ &\leq C M^{-h(D(\psi^{-1} \circ \tau(t), \psi^{-1} \circ \tau(t')))} \\ &= C M^{-h(D(\tau(t), \tau(t')))}. 
\end{align*}
Here the first inequality follows from \eqref{Lipschitz}, the second from the definition of $\Phi$. The third step uses the fact that $\psi$ is an isomorphism, so that $h(D(\tau(t),\tau(t'))) = h(D(\psi^{-1} \circ \tau(t), \psi^{-1} \circ \tau(t')))$.  Finally, any non-uniqueness in the definition of $\sigma$ comes from $\Phi$, hence (\ref{error-slope}) follows from \eqref{Phi_k error} and \eqref{Lipschitz}. 
\end{proof}
The stickiness of the maps $\tau_{\mathbf x}$ is built into their definition \eqref{tau map}. The reader may be interested in observing that there is a naturally sticky map already introduced in this article, which should be viewed as the inspiration for the construction of $\tau$ and $\sigma_{\tau}$. We refer to the geometric content of Lemma \ref{away from root hyperplane lemma}, which in the language of trees has a particularly succinct reformulation. We record this below. 
\begin{lemma}\label{away from root hyperplane lemma - tree version}  
For $C_0$ obeying the requirement of Lemma \ref{away from root hyperplane lemma} and $p \in [C_0, C_0+1] \times \mathbb R^d$, let Poss$(p)$ be as in \eqref{defn Poss(p)}. Let $\Phi$ and $\psi$ be the maps used in  Proposition \ref{slope assignment}. Then the map $t \mapsto \beta(t)$ which maps every $t \in \text{Poss}(p)$ to the unique $\beta(t) \in [0,1)$ such that 
\begin{equation} \label{defn beta(t)}
p \in \mathcal P_{t, v(t)} \quad  \text{where} \quad v(t) = \gamma \circ \Phi \circ \psi^{-1} \circ \beta(t),
\end{equation}
extends as a well-defined sticky map from $\mathcal T_N(\text{Poss}(p);M)$ to $\mathcal T_N([0,1);2)$.  
\end{lemma}
\begin{proof} 
By Lemma \ref{away from root hyperplane lemma}(\ref{exactly one slope per t}), there exists for every $t \in \text{Poss}(p)$ a unique $v(t) \in \Omega_N$ such that $p \in \mathcal P_{t, v(t)}$. Let us therefore define for $1 \leq k \leq N$,  
\begin{equation}  \label{defn beta(t) tree version} \beta(\pi_1(t), \cdots, \pi_k(t)) =  (\pi_1 \circ \beta(t), \cdots, \pi_k \circ \beta(t)) \end{equation}   
where $\beta(t)$ is as in \eqref{defn beta(t)} and as always $\pi_k$ denotes the projection to the $k$th coordinate of an input sequence.  More precisely, $\pi_k(t)$ represents the unique $k$th level $M$-adic cube that contains $t$.  Similarly $\pi_k(\beta(t))$ is the $k$th component of the $N$-long binary sequence that identifies $\beta(t)$. The function $\beta$ defined in \eqref{defn beta(t) tree version} maps $\mathcal T_N(\text{Poss}(p);M)$ to $\mathcal T_N([0,1);2)$, and agrees with $\beta$ as in \eqref{defn beta(t)} if $k=N$. 

To check that the map is consistently defined, we pick $t \ne t'$ in Poss$(p)$ with $u = D(t,t')$ and aim to show that $\beta(\pi_1(t), \cdots, \pi_k(t)) = \beta(\pi_1(t'), \cdots, \pi_k(t'))$ for all $k$ such that $k \leq h(u)$. But by definition \eqref{defn beta(t)}, $v(t)$ and $v(t')$ have the property that $p \in \mathcal P_{t,v(t)} \cap \mathcal P_{t',v(t')}$. Hence Lemma \ref{away from root hyperplane lemma}(\ref{exactly one basic interval at every stage}) asserts that $\alpha(t) = \gamma^{-1}(v(t))$ and $\alpha(t') = \gamma^{-1}(v(t'))$ share the same basic interval at step $h(u)$ of the Cantor construction. Thus $\beta(t) = \psi \circ \Phi^{-1} \circ \alpha(t)$ and $\beta(t') = \psi \circ \Phi^{-1} \circ \alpha(t')$ have a common ancestor in $\mathcal T_N([0,1);2)$ at height $h(u)$, and hence $\pi_k(\beta(t)) = \pi_k(\beta(t'))$ for all $k \leq h(u)$, as claimed. Preservation of heights and lineages is a consequence of the definition \eqref{defn beta(t) tree version}, and stickiness follows.  
\end{proof} 

\subsection{Construction of Kakeya-type sets revisited} 
As $\tau$ ranges over all sticky maps $\tau_{\mathbf x}: \mathcal T_N([0,1)^d) \rightarrow \mathcal T_N([0,1);2)$ with $\mathbf x \in \mathbb X_N$,  we now have for every vertex $t \in\mathcal{T}_N([0,1)^d)$ with $h(t)=N$ a random sticky slope assignment $\sigma(t) \in \Omega_N$ defined as above.  For all such $t$, this generates a randomly oriented tube $P_{t,\sigma}$ given by \eqref{defn-Ptsigma} rooted at the $M$-adic cube $Q_t$ identified by $t$, with sidelength $\kappa_d\cdot M^{-N}$ in the $x_1=0$ plane.  We may rewrite the collection of such tubes from \eqref{Kakeya set original} as
\begin{equation}\label{Kakeya sets}
K_N(\sigma) := \bigcup_{\begin{subarray}{c}t\in\mathcal{T}_N([0,1)^d)\\ h(t) = N\end{subarray}} P_{t,\sigma}.
\end{equation}

On average, a random collection of tubes with the above described sticky slope assignment will comprise a Kakeya-type set, as per \eqref{Kakeya-type condition}.  Specifically, we will show in the next section that the following proposition holds. In view of Proposition \ref{generic upper and lower bounds}, this will suffice to prove Theorem \ref{main theorem}.   
\begin{proposition}\label{Kakeya on average}
	Suppose $(\Sigma_N,\mathfrak{P}(\Sigma_N),\text{Pr})$ is the probability space of sticky maps described above, equipped with the uniform probability measure.  For every $\sigma\in\Sigma_N$, there exists a set $K_N(\sigma)$ as defined in \eqref{Kakeya sets}, with tubes oriented in directions from $\Omega_N = \gamma(\mathcal D^{[N]}_M)$. Then these random sets obey 
the hypotheses of Proposition~\ref{generic upper and lower bounds} with \begin{equation} a_N =  c_M\frac{\sqrt{\log N}}{N} \qquad \text{ and } \qquad b_N = \frac{C_M}{N} , \label{our a_N b_N}\end{equation}  where $c_M$ and $C_M$ are fixed positive constants depending only on $M$ and $d$.  The content of Proposition~\ref{generic upper and lower bounds} allows us to conclude that $\Omega$ admits Kakeya-type sets.
\end{proposition}

%%%%%%%%%%%%%%%%%%%%%%%%%%%%%%%%%%%%%%%%%%%%%%%%%%%%%%%%%%%%%%%%%%%%%%%

\section{Slope probabilities and root configurations} \label{probability estimation section}
Having established the randomization method for assigning slopes to tubes, we are now in a position to apply this toward the estimation of probabilities of certain events that will be of interest in the next section. Roughly speaking, we wish to compute conditional probabilities that one or more cubes on the root hyperplane are assigned prescribed slopes, provided  similar information is available for other cubes.  
\begin{lemma}\label{probability estimate}
Let us fix $v_1, v_2 \in \Omega_N$, so that $v_1 = \gamma(\alpha_1)$ and $v_2 = \gamma(\alpha_2)$ for unique $\alpha_1, \alpha_2 \in \mathcal D_M^{[N]}$. We also fix $t_1, t_2 \in \mathcal T_N([0,1)^d)$, $h(t_1) = h(t_2) = N$, $t_1 \ne t_2$. Let us denote by $u \in \mathcal T_N([0,1)^d)$ and $\alpha \in \mathcal T_N(\mathcal D_M^{[N]})$ the youngest common ancestors of $(t_1, t_2)$ and $(\alpha_1, \alpha_2)$ respectively, i.e., $u = D(t_1, t_2)$, $\alpha = D(\alpha_1, \alpha_2)$. Then   
\begin{equation}\label{prob est}
	\text{Pr}\bigl(\sigma(t_2)=v_2 \bigl| \sigma(t_1)=v_1\bigr) = \begin{cases} 2^{-(N-h(u))} &\text{ if } h(u) \leq h(\alpha), \\  0 &\text{ otherwise. }\end{cases}
\end{equation}
\end{lemma}

\begin{proof}
Keeping in mind the slope assignment as described in \eqref{sticky direction assignment}, and the stickiness of the map $\tau$ as given in Proposition~\ref{sticky assignment}, the proof can be summarized as in Figure~\ref{Fig:sticky tree}.  Since $t_1$ and $t_2$ must map to $v_1 = \gamma(\alpha_1)$ and $v_2 = \gamma(\alpha_2)$ under $\sigma = \sigma_{\tau}$, the sticky map $\psi^{-1} \circ \tau$ must map $t_1$ and $t_2$ to the $N$th stage basic intervals in the Cantor construction containing $\alpha_1$ and $\alpha_2$ respectively. Since sticky maps preserve heights and lineages, we must have $h(\alpha) \geq h(u)$. Assuming this, we simply count the number of distinct edges on the ray defining $t_2$ that are not common with $t_1$.  The map $\tau$ generating $\sigma = \sigma_{\tau}$ is defined by a binary choice on every edge in $\mathcal{T}_N([0,1)^d)$, and the rays given by $t_1$ and $t_2$ agree on their first $h(u)$ edges, so we have exactly $N-h(u)$ binary choices to make.  This is precisely \eqref{prob est}.\\

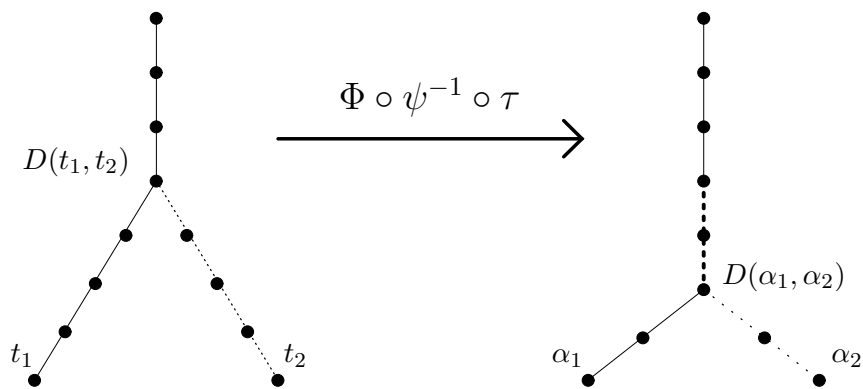
\begin{figure}[h!]
\setlength{\unitlength}{0.8mm}
\begin{picture}(-50,0)(-106,0)

	\special{sh 0.99}\put(-60,0){\ellipse{2}{2}}
	\special{sh 0.99}\put(-60,-27){\ellipse{2}{2}}
	\special{sh 0.99}\put(-60,-9){\ellipse{2}{2}}
	\special{sh 0.99}\put(-60,-18){\ellipse{2}{2}}
	\special{sh 0.99}\put(-65,-36){\ellipse{2}{2}}
	\special{sh 0.99}\put(-70,-44){\ellipse{2}{2}}
	\special{sh 0.99}\put(-75,-52){\ellipse{2}{2}}
	\special{sh 0.99}\put(-55,-36){\ellipse{2}{2}}
	\special{sh 0.99}\put(-50,-44){\ellipse{2}{2}}
	\special{sh 0.99}\put(-45,-52){\ellipse{2}{2}}
	\put(-82,-25){\shortstack{$D(t_1,t_2)$}}
	\put(-39,-57){\shortstack{$t_2$}}
	\put(-84,-57){\shortstack{$t_1$}}
	\special{sh 0.99}\put(-80,-60){\ellipse{2}{2}}
	\special{sh 0.99}\put(-40,-60){\ellipse{2}{2}}

	\path(-60,0)(-60,-27)
	\path(-80,-60)(-60,-27)
	\dottedline{1}(-60,-27)(-40,-60)

%%%%%%%%%%%%%%%%

	\special{sh 0.99}\put(30,0){\ellipse{2}{2}}
	\special{sh 0.99}\put(30,-36){\ellipse{2}{2}}
	\special{sh 0.99}\put(30,-9){\ellipse{2}{2}}
	\special{sh 0.99}\put(30,-18){\ellipse{2}{2}}
	\special{sh 0.99}\put(30,-27){\ellipse{2}{2}}
	\special{sh 0.99}\put(30,-45){\ellipse{2}{2}}
	\special{sh 0.99}\put(20,-53){\ellipse{2}{2}}
	\special{sh 0.99}\put(40,-53){\ellipse{2}{2}}
	\put(33,-44){\shortstack{$D(\alpha_1,\alpha_2)$}}
	\put(51,-57){\shortstack{$\alpha_2$}}
	\put(5,-57){\shortstack{$\alpha_1$}}
	\special{sh 0.99}\put(11,-60){\ellipse{2}{2}}
	\special{sh 0.99}\put(49,-60){\ellipse{2}{2}}

	\path(30,0)(30,-27)
	\path(11,-60)(30,-45)
	\dottedline{2}(30,-45)(49,-60)
	\allinethickness{.5mm}\dottedline{2}(30,-27)(30,-45)

%%%%%%%%%%%%%%%%%%

	\path(-40,-20)(10,-20)
	\path(7,-17)(10,-20)(7,-23)
	\put(-30,-15){\Large\shortstack{$\Phi\circ\psi^{-1}\circ\tau$}}

%%%%%%%%%%%%%%%%

\end{picture}
\vspace{5cm}
\caption{\label{Fig:sticky tree} Diagram of the sticky assignment between the two rays defining $t_1,t_2\in\mathcal{T}_N([0,1)^d)$ and the two rays defining their assigned slopes $\alpha_1,\alpha_2\in\mathcal{D}^{[N]}_M$.  The bold edges defining $t_1$ are fixed to map to the corresponding bold edges at the same height defining $\alpha_1$.  This leaves a binary choice to be made at each of the dotted edges along the path between $D(t_1,t_2)$ and $t_2$.  We see that $t_2$ is assigned the slope $v_2$ under $\sigma$ if and only if these dotted edges are assigned via $\Phi\circ\psi^{-1}\circ\tau$ to the dotted edges on the ray defining $\alpha_2$.}
\end{figure}

More explicitly, if $t_1 = \langle i_1, i_2, \cdots, i_N \rangle$ and $t_2 = \langle j_1, \cdots, j_N \rangle$, then  
\begin{equation} \label{agreement between t_1 and t_2}
\langle i_1, \cdots, i_{h(u)} \rangle = \langle j_1, \cdots, j_{h(u)} \rangle. 
\end{equation} 
The event of interest may therefore be recast as 
\begin{align}
\bigl\{ \sigma(t_2) &= v_2 \bigl| \sigma(t_1) = v_1 \bigr\} \\ &= \Bigl\{ \tau(j_1, \cdots, j_N) = \psi \circ \Phi^{-1} (\alpha_2) \Bigl| \tau(i_1, \cdots, i_N) = \psi \circ \Phi^{-1} (\alpha_1)\Bigr\}  \nonumber \\ &= \Bigl\{ \langle X_{\langle j_1\rangle}, \cdots, X_{\langle j_1, \cdots, j_N\rangle} \rangle = \psi \circ \Phi^{-1} (\alpha_2) \Bigl| \langle X_{\langle i_1\rangle}, \cdots, X_{\langle i_1, \cdots, i_N\rangle} \rangle = \psi \circ \Phi^{-1} (\alpha_1) \Bigr\} \nonumber \\ &= \Bigl\{ X_{\langle j_1, \cdots, j_k \rangle} = \pi_k \circ \psi \circ \Phi^{-1}(\alpha_2) \text{ for } h(u) + 1 \leq k \leq N \Bigr\}, \label{2nd matching event}
\end{align}
where $\pi_k$ denotes the $k$th component of the input sequence. At the second step above we have used \eqref{tau map} and Proposition~\ref{slope assignment}, and the third step uses \eqref{agreement between t_1 and t_2}. The event in \eqref{2nd matching event} then amounts to the agreement of two $(N-h(u))$-long binary sequences, with an independent, 1/2 chance of agreement at each sequential component.  The probability of such an event is $2^{-(N-h(u))}$, as claimed.
\end{proof} 
The same idea can be iterated to compute more general probabilities. To exclude configurations that are not compatible with stickiness, let us agree to call a collection \begin{equation} \label{sticky admissible collection} \{(t, \alpha_t) : t \in A, \; h(t) = h(\alpha_t) = N \} \subseteq \mathcal T_N([0,1)^d) \times \mathcal D_M^{[N]} \end{equation} of point-slope combinations {\em{sticky-admissible}} if there exists a sticky map $\tau$ such that $\psi^{-1} \circ \tau$ maps $t$ to $\alpha_t$ for every $t \in A$. Notice that existence of a sticky $\tau$ imposes certain consistency requirements on a sticky-admissible collection \eqref{sticky admissible collection}; for example $h(D(\alpha_t, \alpha_{t'})) \geq h(D(t,t'))$, and more generally $h(D(\alpha_t : t \in A')) \geq h(D(A'))$ for any finite subset $A' \subseteq A$.  

For sticky-admissible configurations, we summarize the main conditional probability of interest, leaving the proof to the interested reader. 
\begin{lemma} \label{probability estimate general}
Let $A$ and $B$ be finite disjoint collections of vertices in $\mathcal T_N([0,1)^d)$ of height $N$. Then for any choice of slopes $\{ v_t = \gamma(\alpha_t) : t \in A \cup B\} \subseteq \Omega_N$ such that the collection $\{(t, \alpha_t) : t \in A \cup B \}$ is sticky-admissible, the following equation holds: 
\[ \text{Pr} \bigl(\sigma(t) = v_t \text{ for all } t \in B \; \big| \; \sigma(t) = v_t \text{ for all } t \in A \bigr) = \left( \frac{1}{2}\right)^{k(A,B)}, \] 
where $k(A,B)$ is the number of distinct edges in the tree identifying $B$ that are not common with the tree identifying $A$. If $\{(t,\alpha_t) : t\in A\cup B\}$ is not sticky-admissible, then the probability is zero.   
\end{lemma} 
For the remainder of this section, we focus on some special events of the form dealt with in Lemma \ref{probability estimate general} that will be critical to the proof of (\ref{generic lower bound}). In all these cases of interest $\#(A), \#(B) \leq 2$.  As is reasonable to expect, the configuration of the root cubes within the tree $\mathcal T_N([0,1)^d)$ plays a role in determining $k(A,B)$. While there is a large number of possible configurations, we isolate certain structures that will turn out to be generic enough  for our purposes.
\subsection{Four point root configurations} 
\begin{definition} \label{preferred configuration definition}
Let $\mathbb I = \{(t_1, t_2); (t_1', t_2')\}$ be an ordered tuple of four distinct points in $\mathcal T_N([0,1)^d)$ of height $N$ such that 
\begin{equation} \label{u u' height relation} h(u) \leq h(u') \quad \text{ where } u = D(t_1, t_2), \; u' = D(t_1', t_2'). \end{equation} 
 We say that $\mathbb I$ is in type 1 configuration if exactly one of the following conditions is satisfied:
\begin{enumerate}[(a)]
\item either $u \cap u' = \emptyset$, or
\item $u' \subsetneq u$, or
\item $u = u' = D(t_i, t_j')$ for all $i, j = 1,2$
\end{enumerate}
If $\mathbb I$ satisfying \eqref{u u' height relation} is not of type 1, we call it of type 2.  An ordered tuple $\mathbb I$ not satisfying the inequality in \eqref{u u' height relation} is said to be of type $j = 1, 2$ if $\mathbb I' = \{(t_1', t_2'); (t_1, t_2) \}$ is of the same type.  
\end{definition}     
The different structural possibilities are listed in Figure~\ref{Fig:configurations}. 
\begin{figure}[h!]
\setlength{\unitlength}{0.6mm}
\begin{picture}(-55,-10)(-135,35)

        \allinethickness{0.1mm}\path(-100,20)(-100,-100)(71,-100)(71,20)(-100,20)
	\path(-43,20)(-43,-100)
	\path(14,20)(14,-100)
	\path(-100,-40)(71,-40)
	\path(-100,-100)(71,-100)
	\path(-100,20)(-100,30)(71,30)(71,20)

	\put(-46,23){\shortstack{Type 1 Configurations}}

	\special{sh 0.99}\put(-71,10){\ellipse{2}{2}}
	\special{sh 0.99}\put(-71,0){\ellipse{2}{2}}
	\put(-87,-10){\shortstack{$u$}}
	\special{sh 0.99}\put(-81,-10){\ellipse{2}{2}}
	\put(-58,-17){\shortstack{$u'$}}
	\special{sh 0.99}\put(-61,-17){\ellipse{2}{2}}
	\put(-89,-36){\shortstack{$t_1$}}
	\special{sh 0.99}\put(-87,-30){\ellipse{2}{2}}
	\put(-79,-36){\shortstack{$t_2$}}
	\special{sh 0.99}\put(-78,-30){\ellipse{2}{2}}
	\put(-63,-36){\shortstack{$t_1'$}}
	\special{sh 0.99}\put(-62,-30){\ellipse{2}{2}}
	\put(-53,-36){\shortstack{$t_2'$}}
	\special{sh 0.99}\put(-52,-30){\ellipse{2}{2}}

	\path(-71,10)(-71,0)
	\path(-81,-10)(-71,0)(-61,-17)
	\path(-87,-30)(-81,-10)(-78,-30)
	\path(-62,-30)(-61,-17)(-52,-30)
	\put(-97,13){\shortstack{(a)}}

%%%%%%%%%%%%%%%

	\special{sh 0.99}\put(43,10){\ellipse{2}{2}}
	\put(37,-1){\shortstack{$u$}}
	\special{sh 0.99}\put(43,0){\ellipse{2}{2}}
	\put(37,-16){\shortstack{$u'$}}
	\special{sh 0.99}\put(43,-15){\ellipse{2}{2}}
	\put(61,-36){\shortstack{$t_2$}}
	\special{sh 0.99}\put(23,-30){\ellipse{2}{2}}
	\put(21,-36){\shortstack{$t_1$}}
	\special{sh 0.99}\put(38,-30){\ellipse{2}{2}}
	\put(36,-36){\shortstack{$t_1'$}}
	\special{sh 0.99}\put(50,-30){\ellipse{2}{2}}
	\put(48,-36){\shortstack{$t_2'$}}
	\special{sh 0.99}\put(63,-30){\ellipse{2}{2}}

	\path(43,10)(43,0)(43,-15)
	\path(50,-30)(43,-15)(38,-30)
	\path(63,-30)(43,0)(23,-30)
	\put(17,13){\shortstack{(c)}}

%%%%%%%%%%%%%%%

	\special{sh 0.99}\put(43,-50){\ellipse{2}{2}}
	\special{sh 0.99}\put(43,-60){\ellipse{2}{2}}
	\put(46,-61){\shortstack{$u$}}
	\special{sh 0.99}\put(43,-70){\ellipse{2}{2}}
	\put(46,-71){\shortstack{$u'$}}
	\special{sh 0.99}\put(52,-80){\ellipse{2}{2}}
	\put(25,-96){\shortstack{$t_2$}}
	\special{sh 0.99}\put(27,-90){\ellipse{2}{2}}
	\put(35,-96){\shortstack{$t_2'$}}
	\special{sh 0.99}\put(36,-90){\ellipse{2}{2}}
	\put(51,-96){\shortstack{$t_1'$}}
	\special{sh 0.99}\put(52,-90){\ellipse{2}{2}}
	\put(61,-96){\shortstack{$t_1$}}
	\special{sh 0.99}\put(62,-90){\ellipse{2}{2}}

	\path(43,-50)(43,-70)
	\path(43,-60)(27,-90)
	\path(36,-90)(43,-70)(52,-80)(52,-90)
	\path(52,-80)(62,-90)
	\put(17,-47){\shortstack{(f)}}

%%%%%%%%%%%%%%%

	\special{sh 0.99}\put(-71,-50){\ellipse{2}{2}}
	\special{sh 0.99}\put(-71,-60){\ellipse{2}{2}}
	\put(-68,-61){\shortstack{$u$}}
	\special{sh 0.99}\put(-71,-70){\ellipse{2}{2}}
	\special{sh 0.99}\put(-62,-80){\ellipse{2}{2}}
	\put(-59,-81){\shortstack{$u'$}}
	\put(-89,-96){\shortstack{$t_2$}}
	\special{sh 0.99}\put(-87,-90){\ellipse{2}{2}}
	\put(-80,-96){\shortstack{$t_1$}}
	\special{sh 0.99}\put(-78,-90){\ellipse{2}{2}}
	\put(-63,-96){\shortstack{$t_1'$}}
	\special{sh 0.99}\put(-62,-90){\ellipse{2}{2}}
	\put(-53,-96){\shortstack{$t_2'$}}
	\special{sh 0.99}\put(-52,-90){\ellipse{2}{2}}

	\path(-71,-50)(-71,-70)
	\path(-71,-60)(-87,-90)
	\path(-78,-90)(-71,-70)(-62,-80)(-62,-90)
	\path(-62,-80)(-52,-90)
	\put(-97,-47){\shortstack{(d)}}

%%%%%%%%%%%%%%%

	\special{sh 0.99}\put(-14,10){\ellipse{2}{2}}
	\put(-11,-11){\shortstack{$u = u'$}}
	\special{sh 0.99}\put(-14,-10){\ellipse{2}{2}}
	\put(-36,-36){\shortstack{$t_1$}}
	\special{sh 0.99}\put(-34,-30){\ellipse{2}{2}}
	\put(-23,-36){\shortstack{$t_1'$}}
	\special{sh 0.99}\put(-21,-30){\ellipse{2}{2}}
	\put(-9,-36){\shortstack{$t_2$}}
	\special{sh 0.99}\put(-7,-30){\ellipse{2}{2}}
	\put(4,-36){\shortstack{$t_2'$}}
	\special{sh 0.99}\put(6,-30){\ellipse{2}{2}}

	\path(-14,10)(-14,-10)
	\path(-34,-30)(-14,-10)(-21,-30)
	\path(-7,-30)(-14,-10)(6,-30)
	\put(-40,13){\shortstack{(b)}}

%%%%%%%%%%%%%%%

	\special{sh 0.99}\put(-14,-50){\ellipse{2}{2}}
	\special{sh 0.99}\put(-14,-60){\ellipse{2}{2}}
	\put(-11,-61){\shortstack{$u$}}
	\special{sh 0.99}\put(-14,-70){\ellipse{2}{2}}
	\put(-11,-71){\shortstack{$u'$}}
	\put(-34,-96){\shortstack{$t_2$}}
	\special{sh 0.99}\put(-32,-90){\ellipse{2}{2}}
	\put(-23,-96){\shortstack{$t_1$}}
	\special{sh 0.99}\put(-21,-90){\ellipse{2}{2}}
	\put(-10,-96){\shortstack{$t_1'$}}
	\special{sh 0.99}\put(-9,-90){\ellipse{2}{2}}
	\put(4,-96){\shortstack{$t_2'$}}
	\special{sh 0.99}\put(5,-90){\ellipse{2}{2}}

	\path(-14,-50)(-14,-70)
	\path(-14,-60)(-32,-90)
	\path(-21,-90)(-14,-70)(-14,-70)(-9,-90)
	\path(-14,-70)(5,-90)
	\put(-40,-47){\shortstack{(e)}}

%%%%%%%%%%%%%%%
\end{picture}
\vspace{8.5cm}
\caption{\label{Fig:configurations} All possible four point configurations of type 1, up to permutations.}
\end{figure}
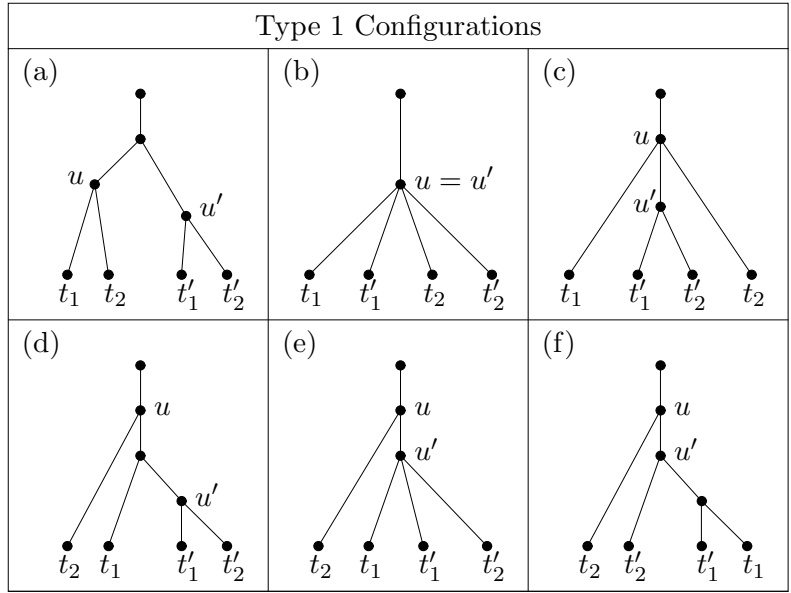
The advantage of a type 1 configuration is that, in addition to being overwhelmingly popular, it allows (up to permutations) an easy computation of the quantity $k(A,B)$ described in Lemma \ref{probability estimate general} if $\#(A) = \#(B) = 2$,  $A \cup B = \{ t_1, t_1', t_2, t_2' \}$ and $\#(A \cap \{ t_1, t_2\}) = \#(B \cap \{ t_1, t_2\}) =  1$. 
\begin{lemma}\label{probability estimate for second moment lemma}
Let $\mathbb I = \{(t_1, t_2) ; (t_1', t_2')\}$ obeying \eqref{u u' height relation} be in type 1 configuration. Let $v_i = \gamma(\alpha_i)$, $v_i' = \gamma(\alpha_i')$, $i = 1,2$, be (not necessarily distinct) points in $\Omega_N$. Then there exist two permutations $\{i_1, i_2\}$ and $\{ j_1, j_2\}$ of $\{1,2\}$ such that 
\[ \text{Pr} \bigl( \sigma(t_{i_2}) = v_{i_2}, \sigma(t'_{j_2}) = v'_{j_2} \bigl| \sigma(t_{i_1}) = v_{i_1}, \sigma(t'_{j_1}) = v'_{j_1} \bigr) = \left(\frac{1}{2} \right)^{2N - h(u) - h(u')}. \] 
provided the collection $\{(t_i, \alpha_i), (t_i', \alpha_i') ; i=1,2\}$ is sticky-admissible. If the admissibility requirement is not met, then the probability is zero.     
\end{lemma}
\begin{proof}
The proof is best illustrated by referring to the above diagram, Figure~\ref{Fig:configurations}.
If $u \cap u' = \emptyset$, then any two permutations will satisfy the conclusion of the lemma, Figure~\ref{Fig:configurations}(a). In particular, choosing $i_1 = j_1 = 1$, $i_2 = j_2 = 2$, we see that the number of edges in $B = \{ t_2, t_2 '\}$ not shared by $A = \{ t_1, t_1' \}$ is $k(A,B) = (N - h(u)) + (N - h(u')) = 2N - h(u) - h(u')$.  The same argument applies if $u = u' = D(t_i, t_j')$ for all $i, j = 1,2$, Figure~\ref{Fig:configurations}(b).

We turn to the remaining case where $u' \subsetneq u$. Here there are several possiblities for the relative positions of $t_1, t_2$. Suppose first that there is no vertex $w$ on the ray joining $u$ and $u'$ with $h(u) < h(w) < h(u')$ such that $w$ is an ancestor of $t_1$ or $t_2$. This means that the rays of $t_1$, $t_2$ and $u'$ follow disjoint paths starting from $u$, so any choice of permutation suffices, Figure~\ref{Fig:configurations}(c). Suppose next that there is a vertex $w$ on the ray joining $u$ and $u'$ with $h(u) < h(w) < h(u')$ such that $w$ is an ancestor of exactly one of $t_1, t_2$, but no descendant of $w$ on this path is an ancestor of either $t_1$ or $t_2$, Figure~\ref{Fig:configurations}(d). In this case, we choose $t_{i_1}$ to be the unique element of $\{t_1, t_2\}$ whose ancestor is $w$. Note that the ray for $t_{i_2}$ must have split off from $u$ in this case. Any permutation of $\{ t_1', t_2'\}$ will then give rise to the desired estimate. If neither of the previous two cases hold, then exactly one of $\{ t_1, t_2 \}$, say $t_{i_1}$, is a descendant of $u'$. If $u' = D(t_{i_1}, t_j')$ for both $j=1,2$, then again any permutation of $\{ t_1', t_2' \}$ works, Figure~\ref{Fig:configurations}(e). Thus the only remaining scenario is where there exists exactly one element in $\{ t_1', t_2' \}$, call it $t_{j_1}'$, such that $h(D(t_{i_1}, t_{j_1}')) > h(u')$. In this case, we choose $A = \{t_{i_1}, t_{j_1}' \}$ and $B = \{t_{i_2}, t_{j_2}' \}$, Figure~\ref{Fig:configurations}(f).  All cases now result in $k(A,B) = 2N - h(u) - h(u')$, completing the proof.      
\end{proof} 
\begin{lemma}
\label{type 2 config lemma}
Let $\mathbb I = \{(t_1, t_2) ; (t_1', t_2') \}$ obeying \eqref{u u' height relation} be in type 2 configuration. Then there exist permutations $\{i_1, i_2 \}$ and $\{ j_1, j_2\}$ of $\{1, 2\}$ for which we have the relations 
\[ \begin{aligned} u_1 &\subseteq u, \; u_2 \subsetneq u \text{ with }  h(u) \leq h(u_1) \leq h(u_2), \; \text{ where } \\ u_1 &= D(t_{i_1}, t_{j_1}'), \; u_2 = D(t_{i_2}, t_{j_2}'), \end{aligned} \] 
and for which the following equality holds:
\[ \text{Pr}\bigl(\sigma(t_{i_1})=v_{i_1}, \sigma(t_{j_1}') =  v_{j_1}' \; \bigl| \; \sigma(t_{i_2})=v_{i_2}, \sigma(t_{j_2}') =  v_{j_2}'\bigr) = \left(\frac{1}{2} \right)^{2N - h(u) - h(u_1)}\]  
for any choice of slopes $v_1, v_1', v_2, v_2' \in \Omega_N$ for which $\{(t_i, \alpha_i), (t_i', \alpha_i'); i=1,2\}$ is sticky-admissible. 
\end{lemma}

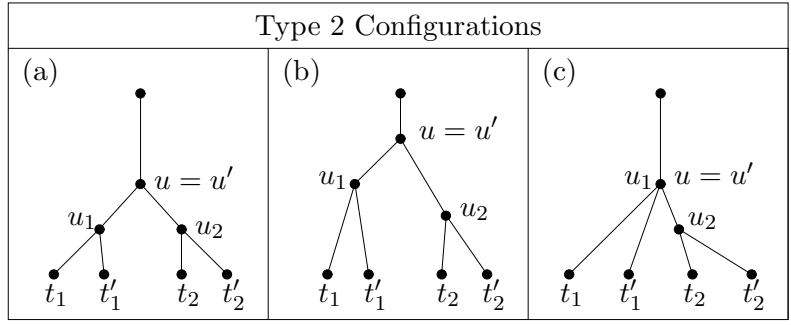
\begin{figure}[h]
\setlength{\unitlength}{0.6mm}
\begin{picture}(-50,0)(-135,30)

	\path(-100,20)(-100,30)(71,30)(71,20)
        \allinethickness{0.1mm}\path(-100,20)(-100,-40)(71,-40)(71,20)(-100,20)
	\path(-43,20)(-43,-40)
	\path(14,20)(14,-40)
	\path(71,20)(71,-40)

	\put(-46,23){\shortstack{Type 2 Configurations}}

	\special{sh 0.99}\put(-14,10){\ellipse{2}{2}}
	\special{sh 0.99}\put(-14,0){\ellipse{2}{2}}
	\put(-10,0){\shortstack{$u = u'$}}
	\put(-32,-10){\shortstack{$u_1$}}
	\special{sh 0.99}\put(-24,-10){\ellipse{2}{2}}
	\put(-1,-17){\shortstack{$u_2$}}
	\special{sh 0.99}\put(-4,-17){\ellipse{2}{2}}
	\put(-32,-36){\shortstack{$t_1$}}
	\special{sh 0.99}\put(-30,-30){\ellipse{2}{2}}
	\put(-22,-36){\shortstack{$t_1'$}}
	\special{sh 0.99}\put(-21,-30){\ellipse{2}{2}}
	\put(-6,-36){\shortstack{$t_2$}}
	\special{sh 0.99}\put(-5,-30){\ellipse{2}{2}}
	\put(4,-36){\shortstack{$t_2'$}}
	\special{sh 0.99}\put(5,-30){\ellipse{2}{2}}

	\path(-14,10)(-14,0)
	\path(-24,-10)(-14,0)(-4,-17)
	\path(-30,-30)(-24,-10)(-21,-30)
	\path(-5,-30)(-4,-17)(5,-30)
	\put(-40,13){\shortstack{(b)}}

%%%%%%%%%%%%%%%

	\special{sh 0.99}\put(-71,10){\ellipse{2}{2}}
	\special{sh 0.99}\put(-80,-20){\ellipse{2}{2}}
	\put(-68,-11){\shortstack{$u = u'$}}
	\special{sh 0.99}\put(-71,-10){\ellipse{2}{2}}
	\put(-87,-19){\shortstack{$u_1$}}
	\special{sh 0.99}\put(-62,-20){\ellipse{2}{2}}
	\put(-59,-21){\shortstack{$u_2$}}
	\put(-92,-36){\shortstack{$t_1$}}
	\special{sh 0.99}\put(-90,-30){\ellipse{2}{2}}
	\put(-80,-36){\shortstack{$t_1'$}}
	\special{sh 0.99}\put(-79,-30){\ellipse{2}{2}}
	\put(-63,-36){\shortstack{$t_2$}}
	\special{sh 0.99}\put(-62,-30){\ellipse{2}{2}}
	\put(-53,-36){\shortstack{$t_2'$}}
	\special{sh 0.99}\put(-52,-30){\ellipse{2}{2}}

	\path(-71,10)(-71,-10)
	\path(-79,-30)(-80,-20)(-90,-30)
	\path(-80,-20)(-71,-10)(-62,-20)(-62,-30)
	\path(-62,-20)(-52,-30)
	\put(-97,13){\shortstack{(a)}}

%%%%%%%%%%%%%%%

	\special{sh 0.99}\put(43,10){\ellipse{2}{2}}
	\put(46,-10){\shortstack{$u = u'$}}
	\put(35,-10){\shortstack{$u_1$}}
	\put(49,-19){\shortstack{$u_2$}}
	\special{sh 0.99}\put(43,-10){\ellipse{2}{2}}
	\special{sh 0.99}\put(47,-20){\ellipse{2}{2}}
	\put(21,-36){\shortstack{$t_1$}}
	\special{sh 0.99}\put(23,-30){\ellipse{2}{2}}
	\put(34,-36){\shortstack{$t_1'$}}
	\special{sh 0.99}\put(36,-30){\ellipse{2}{2}}
	\put(48,-36){\shortstack{$t_2$}}
	\special{sh 0.99}\put(50,-30){\ellipse{2}{2}}
	\put(61,-36){\shortstack{$t_2'$}}
	\special{sh 0.99}\put(63,-30){\ellipse{2}{2}}

	\path(43,10)(43,-10)
	\path(23,-30)(43,-10)(36,-30)
	\path(50,-30)(47,-20)(63,-30)
	\path(47,-20)(43,-10)
	\put(17,13){\shortstack{(c)}}

%%%%%%%%%%%%%%%

\end{picture}
\vspace{4.5cm}
\caption{\label{Fig:Type 2 configurations} All possible four point configurations of type 2, up to permutations.}
\end{figure}
\begin{proof}
Since $\mathbb I$ is of type 2, we know that $u=u'$, and hence all pairwise youngest common ancestors of $\{ t_1, t_1', t_2, t_2' \}$ must lie within $u$, but that there exist $i, j \in \{1,2\}$ such that $h(D(t_i, t_j')) > h(u)$.  Let us set $(i_2, j_2)$ to be a tuple for which $h(D(t_{i_2}, t_{j_2}'))$ is maximal. The height inequalities and containment relations are now obvious, and Figure~\ref{Fig:Type 2 configurations} shows that $k(A,B) = (N - h(u)) + (N - h(u_1))$ if $A = \{t_{i_2}, t_{j_2}' \}$ and $B = \{ t_{i_1}, t_{j_1}'\}$.   
\end{proof}

\subsection{Three point root configurations} 
The arguments in the previous section simplify considerably when there are three root cubes instead of four. Since the proofs here are essentially identical to those presented in Lemmas \ref{probability estimate for second moment lemma} and \ref{type 2 config lemma}, we simply record the necessary facts with the accompanying diagram of Figure~\ref{Fig:3point configurations}, leaving their verification to the interested reader. 
\begin{definition}\label{three point defn} 
Let $\mathbb I = \{(t_1, t_2); (t_1, t_2') \}$ be an ordered tuple of three distinct points in $\mathcal T_N([0,1)^d)$ of height $N$ such that $h(u) \leq h(u')$, where $u = D(t_1, t_2)$, $u' = D(t_1, t_2')$. We say that $\mathbb I$ is in type 1 configuration if exactly one of the following two conditions holds:
\begin{enumerate}[(a)]
\item $u' \subsetneq u$, or
\item $u = u' = D(t_2, t_2')$.
\end{enumerate}
Else $\mathbb I$ is of type 2, in which case one necessarily has $u = u'$ and $u_2 = D(t_2, t_2')$ obeys $u_2 \subsetneq u$. If $h(u) > h(u')$, then the type $\mathbb I$ is the same as that of $\mathbb I' = \{ (t_1, t_2'); (t_1, t_2) \}$. 
\end{definition}  
\begin{figure}[h!]
\setlength{\unitlength}{0.6mm}
\begin{picture}(-55,-10)(-135,25)

        \allinethickness{0.1mm}\path(-100,27)(-100,-40)(71,-40)(71,27)(-100,27)
	\path(-43,17)(-43,-40)
	\path(14,27)(14,-40)
	\path(-100,-40)(71,-40)
	\path(-100,17)(71,17)

	\put(-52,20){\large\shortstack{Type 1}}
	\put(30,20){\large\shortstack{Type 2}}

	\special{sh 0.99}\put(-72,10){\ellipse{2}{2}}
	\put(-83,-1){\shortstack{$u$}}
	\special{sh 0.99}\put(-77,0){\ellipse{2}{2}}
	\put(-88,-15){\shortstack{$u'$}}
	\special{sh 0.99}\put(-82,-15){\ellipse{2}{2}}
	\put(-90,-37){\shortstack{$t_1$}}
	\special{sh 0.99}\put(-88,-30){\ellipse{2}{2}}
	\put(-79,-37){\shortstack{$t_2'$}}
	\special{sh 0.99}\put(-79,-30){\ellipse{2}{2}}
	\put(-61,-37){\shortstack{$t_2$}}
	\special{sh 0.99}\put(-60,-30){\ellipse{2}{2}}

	\path(-72,10)(-77,0)
	\path(-82,-15)(-77,0)(-60,-30)
	\path(-88,-30)(-82,-15)(-79,-30)

%%%%%%%%%%%%%%%

	\special{sh 0.99}\put(-14,10){\ellipse{2}{2}}
	\put(-16,-6){\shortstack{$u = u'$}}
	\special{sh 0.99}\put(-19,-5){\ellipse{2}{2}}
	\put(-2,-37){\shortstack{$t_2'$}}
	\special{sh 0.99}\put(-30,-30){\ellipse{2}{2}}
	\put(-31,-37){\shortstack{$t_1$}}
	\special{sh 0.99}\put(-17,-30){\ellipse{2}{2}}
	\put(-18,-37){\shortstack{$t_2$}}
	\special{sh 0.99}\put(-1,-30){\ellipse{2}{2}}

	\path(-14,10)(-19,-5)(-30,-30)
	\path(-17,-30)(-19,-5)(-1,-30)

%%%%%%%%%%%%%%%

	\special{sh 0.99}\put(43,10){\ellipse{2}{2}}
	\special{sh 0.99}\put(38,-5){\ellipse{2}{2}}
	\put(41,-6){\shortstack{$u = u'$}}
	\put(48,-20){\shortstack{$u_2$}}
	\special{sh 0.99}\put(45,-20){\ellipse{2}{2}}
	\put(24,-37){\shortstack{$t_1$}}
	\special{sh 0.99}\put(25,-30){\ellipse{2}{2}}
	\put(39,-37){\shortstack{$t_2$}}
	\special{sh 0.99}\put(41,-30){\ellipse{2}{2}}
	\put(52,-37){\shortstack{$t_2'$}}
	\special{sh 0.99}\put(53,-30){\ellipse{2}{2}}

	\path(43,10)(38,-5)(25,-30)
	\path(53,-30)(45,-20)(41,-30)
	\path(38,-5)(45,-20)
\end{picture}
\vspace{4.3cm}
\caption{\label{Fig:3point configurations} Structural possibilities for three point root configurations}
\end{figure}
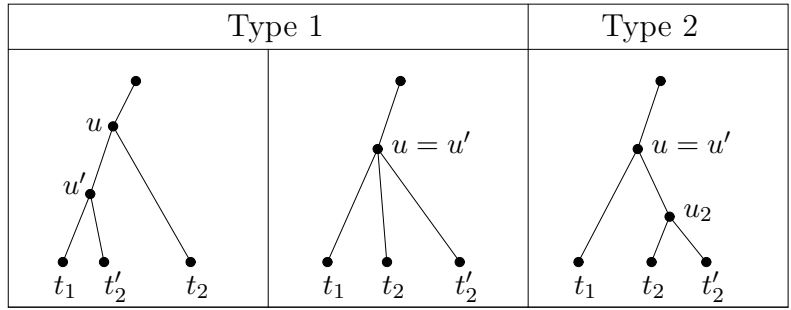

\begin{lemma}\label{three point lemma} 
Let $\mathbb I = \{(t_1, t_2); (t_1, t_2') \}$ be any three-point configuration with $h(u) \leq h(u')$ in the notation of Definition \ref{three point defn}, and let $v_1 = \gamma(\alpha_1)$, $v_2 = \gamma(\alpha_2)$ $v_2' = \gamma(\alpha_2')$ be slopes in $\Omega_N$. Then 
\[
\text{Pr}\bigl(\sigma(t_2) = v_2, \; \sigma(t_2') = v_2' \bigl| \sigma(t_1) = v_1 \bigr) = \begin{cases} \left( \frac{1}{2}\right)^{2N-h(u) - h(u')} &\text{ if $\mathbb I$ is of type 1}, \\ \left( \frac{1}{2}\right)^{2N-h(u) - h(u_2)} &\text{ if $\mathbb I$ is of type 2}, \end{cases} 
\] 
provided the point-slope combination $\{(t_1, \alpha_1), (t_2, \alpha_2), (t_2', \alpha_2') \}$ is sticky-admissible.
\end{lemma} 

%%%%%%%%%%%%%%%%%%%%%%%%%%%%%%%%%%%%%%%%%%%%%%%%%%%%%%%%%%%%%%%%%%%%%%%%

\section{Proposition \ref{Kakeya on average}: Proof of the lower bound \eqref{generic lower bound}}\label{lowerboundsection}

If a collection of many thin tubes is to have a large volume, then it is sensible to expect that the intersection of most pairs of tubes should be small.  The following measure-theoretic lemma of Bateman and Katz~\cite{BatemanKatz} quantifies this phenomenon generally.

\begin{lemma}[\cite{BatemanKatz}, Proposition 2, page 75]\label{MeasureTheory}
	Suppose $(X,\mathcal{A},\mu)$ is a measure space and $A_1,\ldots,A_n\in\mathcal{A}$ are sets with $\mu(A_j) = \alpha$ for every $j$.  If $$\sum_{i=1}^n\sum_{j=1}^n\mu(A_i\cap A_j) \leq L,$$ then $$\mu\left(\bigcup_{i=1}^n A_i\right) \geq \frac {\alpha^2 n^2}{16L}.$$
\end{lemma}

We defer the proof of this fact to reference \cite{Bateman} or \cite{BatemanKatz}.  Using it, we reduce the derivation of inequality \eqref{generic lower bound} with the $a_N$ specified in \eqref{our a_N b_N} via the following lemma.  Throughout this subsection, all probability statements are understood to take place on the probability space $(\Sigma_N,\mathfrak{P}(\Sigma_N),\text{Pr})$ identified in Proposition~\ref{Kakeya on average}.

\begin{proposition} \label{proposition slab estimate} 
Fix integers $N$ and $R$ with $N \gg M$ and $N - \frac{1}{10}\log_M N \leq R \leq N - 10$. Define $P^{\ast}_{t, \sigma, R}$ to be the portion of $P_{t, \sigma}$ contained in the vertical slab $[M^{R-N}, M^{R+1-N}] \times \mathbb R^d$. Then 
\begin{equation} \mathbb E_{\sigma} \Bigl[ \sum_{t_1 \ne t_2} \left| P^{\ast}_{t_1, \sigma, R} \cap P^{\ast}_{t_2, \sigma, R}\right| \Bigr] \lesssim N  M^{-2N + 2R},  \label{generic slab estimate} \end{equation}  
where the implicit constant depends only on $M$ and $d$. 
\end{proposition} 
If one can show that with large probability and for all $R$ specified in Proposition \ref{proposition slab estimate}, the quantity $\sum_{t_1 \ne t_2} \big| P^{\ast}_{t_1, \sigma, R} \cap P^{\ast}_{t_2, \sigma, R}\bigr|$ is bounded above by the right hand side of \eqref{generic slab estimate}, then Lemma \ref{MeasureTheory} would imply (\ref{generic lower bound}) with $a_N = \sqrt{\log N}/N$. Unfortunately, \eqref{generic slab estimate} only shows this on average for every $R$, and hence is too weak a statement to permit such a conclusion. However, with some additional work we are able to upgrade the statement in Proposition \ref{proposition slab estimate} to a second moment estimate, given below. While still not as strong as the statement mentioned above, this suffices for our purposes with a smaller choice of $a_N$.   
\begin{proposition} \label{proposition slab estimate second moment}
Under the same hypotheses as Proposition \ref{proposition slab estimate}, there exists a constant $C_{M,d} > 0$ such that
\begin{equation} \mathbb E_{\sigma} \Bigl[ \Bigl(\sum_{t_1 \ne t_2} \left| P^{\ast}_{t_1, \sigma, R} \cap P^{\ast}_{t_2, \sigma, R}\right| \Bigr)^2\Bigr] \leq C_{M,d}^2 \Bigl(N  M^{-2N + 2R} \Bigr)^2.  \label{generic slab estimate second moment} \end{equation} 
\end{proposition} 
\begin{corollary}
Proposition \ref{proposition slab estimate second moment} implies (\ref{generic lower bound}) with $a_N$ as in  (\ref{our a_N b_N}). 
\end{corollary} 
\begin{proof}
Fix a small constant $c_1 > 0$ such that $2c_1 < \frac{1}{10}$. By Chebyshev's inequality, (\ref{generic slab estimate second moment}) implies that there exists a large constant $C_{M,d} > 0$ such that for every $R$ with $c_1 \log N \leq N - R \leq 2c_1 \log N$, 
\begin{align*} \text{Pr} \Bigl( \Bigl\{ \sigma : \sum_{t_1 \ne t_2} \bigl| P^{\ast}_{t_1, \sigma, R} \cap P^{\ast}_{t_2, \sigma, R}\bigr| \geq 2C_{M,d} N &\sqrt{\log N}  M^{-2N + 2R} \Bigr\} \Bigr) \\ &\leq \frac{\mathbb E_{\sigma} \Bigl[ \Bigl(\sum_{t_1 \ne t_2} \left| P^{\ast}_{t_1, \sigma, R} \cap P^{\ast}_{t_2, \sigma, R}\right| \Bigr)^2\Bigr]}{\bigl(2C_{M,d} N \sqrt{\log N}  M^{-2N + 2R} \bigr)^2}  \\ &\leq \frac{1}{4 \log N}. \end{align*}
Therefore, 
\begin{align*} \text{Pr} \Bigl( \bigcup_{N -R = c_1 \log N}^{2c_1 \log N}\Bigl\{ \sigma : \sum_{t_1 \ne t_2} \bigl| P^{\ast}_{t_1, \sigma, R} \cap P^{\ast}_{t_2, \sigma, R}\bigr| \geq C_{M,d} N &\sqrt{\log N}  M^{-2N + 2R} \Bigr\} \Bigr) \\ &\leq \frac{c_1 \log N}{4 \log N} < \frac{1}{4}. \end{align*}
In other words, for a class of $\sigma$ with probability at least $\frac34$, 
\[\sum_{t_1 \ne t_2} \bigl| P^{\ast}_{t_1, \sigma, R} \cap P^{\ast}_{t_2, \sigma, R}\bigr| \leq C_{M,d} N \sqrt{\log N}  M^{-2N + 2R}  \]
for every $N - R \in \bigl[c_1 \log N, 2c_1 \log N \bigr]$. For such $\sigma$ and the chosen range of $R$, we apply Lemma \ref{MeasureTheory} with $A_t = P^{\ast}_{t, \sigma, R}$, $n = M^{Nd}$, for which $\alpha = C_d M^{R-N} M^{-Nd}$,  and \begin{align*}\sum_{t_1, t_2} \bigl| P^{\ast}_{t_1, \sigma, R} \cap P^{\ast}_{t_2, \sigma, R}\bigr|  &= \Bigl[\sum_{t_1=t_2} + \sum_{t_1 \ne t_2} \Bigr] \bigl| P^{\ast}_{t_1, \sigma, R} \cap P^{\ast}_{t_2, \sigma, R}\bigr|\\ &\leq \alpha n +  C_{M,d} N \sqrt{\log N} M^{-2N + 2R} \\ &\lesssim M^{R-N} + N \sqrt{\log N} M^{-2N + 2R}  \\ &\lesssim N \sqrt{\log N} M^{-2N + 2R} =: L. \end{align*} 
The last step above uses the specified range of $R$. Lemma \ref{MeasureTheory} now yields that 
\[ \Bigl| \bigcup_{t} P^{\ast}_{t, \sigma, R} \Bigr| \gtrsim \frac{(M^{R-N})^2}{L} \sim \frac{1}{N \sqrt{\log N}} \]
for every $N - R \in \bigl[c_1\log N, 2c_1 \log N \bigr]$. Since $\{\cup_t P^{\ast}_{t, \sigma, R} : R \geq 0\}$ is a disjoint collection, we obtain 
\[\bigl| K_N(\sigma) \cap [0,1] \times \mathbb R^d \bigr| \geq \sum_{R = N - 2c_1\log N}^{N - c_1 \log N} \Bigl| \bigcup_{t} P^{\ast}_{t, \sigma, R} \Bigr| \gtrsim \log N \frac{1}{N \sqrt{\log N}}  = a_N, \] 
which is the desired conclusion (\ref{generic lower bound}).  
\end{proof}

\subsection{Proof of Proposition \ref{proposition slab estimate}} 
Thus, we are charged with proving Proposition \ref{proposition slab estimate second moment}.  We will prove Proposition \ref{proposition slab estimate} first, since it involves many of the same ideas as in the proof of the main proposition, but in a simpler setting. We will need to take advantage of several geometric facts, counting arguments and probability estimates prepared in Sections \ref{geometry section} and \ref{probability estimation section} that will be described shortly.  For now, we prescribe the main issues in establishing the bound in \eqref{generic slab estimate}.

\begin{proof}
Given $N$ and $R$ as in the statement of the proposition, we decompose the slab $[M^{R-N}, M^{R+1-N}] \times \mathbb R^d$ into thinner slices $Z_k$, where
\[ Z_k:= \left[ \frac{k}{M^N}, \frac{k+1}{M^N}\right] \times \mathbb R^d, \qquad M^R \leq k \leq M^{R+1}-1.\] 
Setting $P_{t, \sigma, k} := P_{t, \sigma} \cap Z_k$, we observe that $P^{\ast}_{t, \sigma, R}$ is an essentially disjoint union of $\{ P_{t, \sigma, k}\}$. Since $P^{\ast}_{t, \sigma, R}$ is transverse to $Z_k$, we arrive at the estimate 
\begin{align}
\sum_{t_1 \ne t_2} \left| P^{\ast}_{t_1, \sigma, R} \cap P^{\ast}_{t_2, \sigma, R}\right| &= \sum_{M^R \leq k < M^{R+1}} \sum_{t_1 \ne t_2} \left| P_{t_1, \sigma, k} \cap P_{t_2, \sigma, k} \right| \nonumber \\ &\lesssim M^{-(d+1)N} \sum_{M^R \leq k < M^{R+1}} \sum_{t_1 \ne t_2} T_{t_1 t_2} (k) \label{intersection estimate - before u} \\   &\lesssim  M^{-(d+1)N} \sum_{M^R \leq k < M^{R+1}} \sum_{\begin{subarray}{c} u \in \mathcal T_N([0,1)^d) \\ h(u) < N \end{subarray}}\sum_{(t_1, t_2) \in \mathcal S_u} T_{t_1 t_2}(k), \label{intersection estimate - step 1}  
\end{align} 
where $T_{t_1t_2}(k)$ is a random variable that equals one if $P_{t_1, \sigma, k} \cap P_{t_2, \sigma, k} \ne \emptyset$, and is zero otherwise. At the last step in the above string of inequalities, we have further stratified the sum in $(t_1, t_2)$ in terms of their youngest common ancestor $u = D(t_1, t_2)$ in the tree $\mathcal T_N([0,1)^d)$, with the index set $\mathcal S_u$ of the innermost sum being defined by 
\[ \mathcal S_u := \left\{ (t_1, t_2) : t_1, t_2  \in \mathcal T_N([0,1)^d), \; h(t_1) = h(t_2) = N, \; D(t_1, t_2) = u \right\}. \]   

We will prove below in Lemma \ref{proof of step 2 estimate lemma} that 
\begin{equation} \label{intersection estimate step 2}
\mathbb E_{\sigma} \Bigl[ \sum_{(t_1, t_2) \in \mathcal S_u} T_{t_1 t_2}(k) \Bigr] \lesssim M^{R-N} M^{-d h(u) + Nd} = M^{R - d h(u) + N(d-1)}.
\end{equation} 

Plugging this expected count into  the last step of (\ref{intersection estimate - step 1}) and simplifying, we obtain
\[
\begin{aligned} 
\mathbb E_{\sigma} \Bigl[\sum_{t_1 \ne t_2} \left| P^{\ast}_{t_1, \sigma, R} \cap P^{\ast}_{t_2, \sigma, R}\right| \Bigr] &\lesssim \sum_{M^R \leq k < M^{R+1}-1} M^{R-2N} \sum_{\begin{subarray}{c} u \in \mathcal T_N([0,1)^d) \\ h(u) < N \end{subarray}} M^{-dh(u)} \\ 
&\lesssim \sum_{M^R \leq k < M^{R+1}-1} M^{R-2N} N \lesssim NM^{2R-2N},
\end{aligned} 
\]
which is the estimate claimed by Proposition \ref{proposition slab estimate}. At the penultimate step, we have used the fact that there are $M^{dr}$ vertices $u$ in $\mathcal T_N([0,1)^d)$ of height $r$, resulting in \begin{equation} \sum_u M^{-d h(u)} = \sum_{0 \leq r < N} M^{-dr} M^{dr} = N. \label{summation in u}\end{equation}  
\end{proof}

\subsection{Proof of Proposition \ref{proposition slab estimate second moment}} 
\begin{proof}
To establish \eqref{generic slab estimate second moment}, we take a similar route, with some extra care in summing over the (now more numerous) indices. Squaring the  expression in \eqref{intersection estimate - before u}, we obtain
\begin{equation} \label{intersection estimate second moment - step 1} 
\begin{aligned} 
\Bigl[\sum_{t_1 \ne t_2} \left| P^{\ast}_{t_1, \sigma, R} \cap P^{\ast}_{t_2, \sigma, R}\right| \Bigr]^2 &\leq M^{-2(d+1)N}\sum_{k,k' \in [M^R, M^{R+1})}  \sum_{\begin{subarray}{c} t_1 \ne t_2 \\ t_1' \ne t_2' \end{subarray}} T_{t_1t_2}(k) T_{t_1' t_2'}(k') \\ &\leq \mathfrak S_2 + \mathfrak S_3 + \mathfrak S_4, 
\end{aligned} 
\end{equation} 
where the index $i$ in $\mathfrak S_i$ corresponds to the number of distinct points in the tuple $\{(t_1, t_2); (t_1', t_2')\}$. More precisely, for $i=2,3,4$, 
\begin{align} \label{definition of S_i}
\mathfrak S_i &:= M^{-2(d+1)N} \sum_{k,k'} \sum_{\mathbb I \in \mathfrak I_i} T_{t_1t_2}(k) T_{t_1't_2'}(k'), \quad \text{ where }   \\
\mathfrak I_i &:= \left\{\mathbb I = \{(t_1, t_2); (t_1', t_2')\} \Biggl| \begin{aligned} & t_j, t_j' \in \mathcal T_N([0,1)^d), h(t_j) = h(t_j') =N \; \forall j=1,2, \\ &t_1 \ne t_2, \; t_1' \ne t_2',  \; \#(\{t_1, t_1', t_2, t_2'\}) = i  \end{aligned} \right\}.
\end{align} 
The main contribution to the left hand side of \eqref{generic slab estimate second moment} will be from $\mathbb E_{\sigma}(\mathfrak S_4)$, and we will discuss its estimation in detail. The other terms, whose treatment will be briefly sketched, will turn out to be of smaller size. 

We decompose $\mathfrak I_4 = \mathfrak I_{41} \cup \mathfrak I_{42}$, where $\mathfrak I_{4j}$ is the collection of 4-tuples of distinct points $\{(t_1, t_2); (t_1', t_2')\}$ that are in configuration of type $j=1,2$, as explained in Definition \ref{preferred configuration definition}. This results in a corresponding decomposition $\mathfrak S_4 = \mathfrak S_{41} + \mathfrak S_{42}$. For $\mathfrak S_{41}$, we further stratify the sum in terms of $u = D(t_1, t_2)$ and $u' = D(t_1', t_2')$, where we may assume without loss of generality that $h(u) \leq h(u')$. Thus,  
\begin{align} \mathbb E_{\sigma} \bigl(\mathfrak S_{41} \bigr) &= \sum_{k,k'}\sum_{\begin{subarray}{c} u,u' \in \mathcal T_N([0,1)^d) \\ h(u) \leq h(u') < N \end{subarray}} \mathbb E_{\sigma} \bigl(\mathfrak S_{41}(u,u'; k, k') \bigr)\quad \text{ where } \label{S_{41} decomposition}\\ \mathfrak S_{41}(u,u';k,k') &:= M^{-2(d+1)N}\sum_{\mathbb I \in \mathfrak I_{41}(u,u')} T_{t_1t_2}(k) T_{t_1't_2'}(k'), \text{ and }\nonumber \\\mathfrak I_{41}(u,u') &:= \{\mathbb I \in \mathfrak I_{41} : u= D(t_1, t_2), u' = D(t_1', t_2') \}. \nonumber \end{align}   
In Lemma \ref{proof of step 2 estimate second moment lemma} below, we will show that 
\begin{equation}
\begin{aligned} 
\mathbb E_{\sigma} \bigl[ \mathfrak S_{41}(u,u';k,k') \bigr] & \lesssim M^{-2(d+1)N}  M^{2R-d (h(u) + h(u')) + 2N(d-1)} \\ &= M^{2R - 4N-d(h(u) + h(u'))}.
\end{aligned} \label{intersection estimate second moment step 2} 
\end{equation} 
Inserting this back into \eqref{S_{41} decomposition}, we now follow the same summation steps that led to \eqref{generic slab estimate} from \eqref{intersection estimate step 2}. Specifically, applying \eqref{summation in u} twice, we obtain 
\begin{align*}
\mathbb E_{\sigma} (\mathfrak S_{41}) &\lesssim M^{2R - 4N} \sum_{k,k'} \sum_{u,u'} M^{-d(h(u) + h(u'))} \\  &\lesssim \sum_{k,k'} N^2 M^{2R-4N} \lesssim N^2 M^{4R-4N}, 
\end{align*} 
which is the right hand side of \eqref{generic slab estimate second moment}. 

Next we turn to $\mathfrak S_{42}$. Motivated by the configuration type, and after permutations of $\{ t_1, t_2 \}$ and of $\{t_1', t_2' \}$ if necessary (so that the conclusion of Lemma \ref{type 2 config lemma} holds), we stratify this sum in terms of $u = u' = D(t_1, t_2) = D(t_1', t_2')$, $u_1 = D(t_1, t_1')$, $u_2 = D(t_2, t_2')$, writing 
\begin{align} \mathfrak S_{42} &= \sum_{k,k'} \sum_{\begin{subarray}{c}  u, u_1, u_2 \in \mathcal T_N([0,1)^d) \\ u_1, u_2 \subseteq u  \end{subarray}} \mathfrak S_{42}(u,u_1, u_2; k,k'), \text{ where } \nonumber \\ \mathfrak S_{42}(u,u_1, u_2; k,k') &:= M^{-2(d+1)N} \sum_{\mathbb I \in \mathfrak I_{42}(u,u_1, u_2)} T_{t_1t_2}(k) T_{t_1't_2'}(k'),  \text{ and } \nonumber \\ \mathfrak I_{42}(u,u_1, u_2) &:= \left\{\mathbb I \in \mathfrak I_{42} \Bigl| \begin{aligned} &u = D(t_1, t_2) = D(t_1' t_2'), \; \\ &u_1 = D(t_1, t_1'), \; u_2 = D(t_2, t_2') \end{aligned}  \right\} \label{defn I_{42}(u,u_1,u_2)}\end{align}  
for given $u_1, u_2 \subseteq u$ with $h(u) \leq h(u_1) \leq h(u_2)$. For such $u, u_1, u_2$, we will prove in Lemma \ref{second moment lemma S_{42} estimate} below that 
\begin{equation} \label{S_{42}(u, u_1, u_2) estimate}
\mathbb E_{\sigma} \bigl( \mathfrak S_{42}(u,u_1, u_2; k,k') \bigr) \lesssim M^{-2N-2d h(u_2)}.
\end{equation}
Accepting this estimate for the time being, we complete the estimation of $\mathbb E_{\sigma}(\mathfrak S_{42})$ as follows, 
\begin{align}
\mathbb E_{\sigma}(\mathfrak S_{42}) &\lesssim \sum_{k,k'} \sum_{u, u_1, u_2} M^{-2N-2d h(u_2)} \nonumber \\ &\lesssim M^{-2N} \sum_{k,k'} \sum_u  \sum_{u_2 \subseteq u} M^{-2d h(u_2)} \sum_{\begin{subarray}{c}u_1 \subseteq u \\ h(u_1) \leq h(u_2) \end{subarray}} 1 \nonumber \\ &\lesssim M^{-2N} \sum_{k,k'} \sum_u  \sum_{u_2 \subseteq u} M^{-2d h(u_2)} \Bigl[  M^{d(h(u_2) - h(u))}\Bigr] \label{S_{42} completed} \\ &\lesssim M^{-2N} \sum_{k,k'} \sum_u M^{-d h(u)} \sum_{u_2 \subseteq u} M^{-dh(u_2)} \nonumber \\ &\lesssim NM^{-2N} \sum_{k,k'} \sum_u M^{-2dh(u)} \label{summation in u_2} \\ &\lesssim N M^{2R-2N} \label{summation in u after u_2}.  
 \end{align}  
For the range $N-R \leq \frac{1}{2}\log_M N$ assured by Proposition \ref{proposition slab estimate second moment}, the last quantity above is smaller than $(N M^{2R-2N})^2$. The string of inequalities displayed above involve repeated applications of the fact used to prove \eqref{summation in u}, namely there are $M^{dj - dh(u)}$ cubes of sidelength $M^{-j}$ contained in $u$. Thus the estimates 
\begin{align*}
\sum_{\begin{subarray}{c}u_1 \subseteq u \\ h(u_1) \leq h(u_2) \end{subarray}} 1 &\lesssim \sum_{j = h(u)}^{h(u_2)} M^{d(j-h(u))} \lesssim M^{d(h(u_2) - h(u))}, \\
\sum_{u_2 \subseteq u} M^{-d h(u_2)} &\lesssim \sum_{N \geq j \geq h(u)} M^{-dj} M^{d(j-h(u))} \lesssim NM^{-dh(u)}, \text{ and } \\
\sum_{u} M^{-2d h(u)} &= \sum_{j = 0}^N M^{dj}  M^{-2dj} = \sum_{j=0}^N M^{-dj} \lesssim 1
\end{align*}
were used in \eqref{S_{42} completed} \eqref{summation in u_2} and \eqref{summation in u after u_2} respectively, completing the estimation of $\mathbb E(\mathfrak S_4)$.  

Arguments similar to and in fact simpler than those above lead to the following estimates for $\mathbb E(\mathfrak S_3)$ and $\mathbb E(\mathfrak S_2)$:
\begin{align} 
\mathbb E(\mathfrak S_3) & = \mathbb E(\mathfrak S_{31}) + \mathbb E(\mathfrak S_{32}) \nonumber \\ &\lesssim N M^{3R-3N} + M^{3R-3N} \lesssim NM^{3R-3N}, \text{ and } \label{proof omitted estimate - three point}  \\ 
\mathbb E(\mathfrak S_2) &\lesssim N M^{3R-(d+3)N}. \label{proof omitted estimate - two point} 
\end{align} 
Here without loss of generality and after a permutation if necessary, we have assumed that $\mathbb I = \{(t_1, t_2); (t_1, t_2') \}\in \mathfrak I_3$, with $h(D(t_1, t_2)) \leq h(D(t_1, t_2'))$. The subsum $\mathfrak S_{3i}$ then corresponds to tuples $\mathbb I$ that are in type $i$ configuration in the sense of Definition \ref{three point defn}. There is only one possible configuration of pairs in $\mathfrak I_2$. The derivation of the expectation estimates \eqref{proof omitted estimate - three point} and \eqref{proof omitted estimate - two point} closely follow the estimation of $\mathfrak S_4$, with appropriate adjustments in the probability counts; for instance, \eqref{proof omitted estimate - three point} uses Lemma \ref{three point lemma} and \eqref{proof omitted estimate - two point} uses Lemma \ref{probability estimate}. To avoid repetition, we leave the details of \eqref{proof omitted estimate - three point} and \eqref{proof omitted estimate - two point} to the reader, noting that the right hand term in each case is dominated by $(NM^{2R-2N})^2$ by our conditions on $R$.
\end{proof}

%%%%%%%%%%%%%%%%%%%%%%%%%%%%%%%%%%%%%%%%%%%%%%%%%%%%%%%%%%%%%%%%%%%%%%%

\subsection{Expected intersection counts} 
It remains to establish \eqref{intersection estimate step 2}, \eqref{intersection estimate second moment step 2} and \eqref{S_{42}(u, u_1, u_2) estimate}. The necessary steps for this are laid out in the following sequence of lemmas. Unless otherwise stated, we will be using the notation introduced in the proof of Propositions~\ref{proposition slab estimate} and ~\ref{proposition slab estimate second moment}.

\begin{lemma} \label{dist to bdry lemma} 
Fix $Z_k$. Let us define $\mathcal A_u = \mathcal A_u(k)$ to be the (deterministic) collection of all $t_1 \in \mathcal T_N([0,1)^d)$, $h(t_1) = N$ that are contained in the cube $u$ and whose distance from the boundary of some child of $u$ is $ \lesssim kM^{-N - h(u)}$. \\

For $t_1 \in \mathcal A_u$, let $\mathcal B_{t_1} = \mathcal B_{t_1}(k)$ denote the (also deterministic) collection of $t_2 \in \mathcal T_N([0,1)^d)$ with $h(t_2) = N$ and $D(t_1,t_2) = u$ such that the distance between the centres of $t_1$ and $t_2$ is $\lesssim kM^{-N - h(u)}$.  
\begin{enumerate}[(a)]
\item  \label{reduction to A_u} Then for any slope assignment $\sigma$, the random variable $T_{t_1t_2}(k) = 0$ unless $t_1 \in \mathcal A_{u}$ and $t_2 \in \mathcal B_{t_1}$. In other words, 
\begin{align} 
\sum_{(t_1, t_2) \in \mathcal S_u} T_{t_1t_2}(k) &= \sum_{t_1 \in \mathcal  A_u} \sum_{t_2 \in \mathcal B_{t_1}} T_{t_1t_2}(k), \text{ so that } \nonumber  \\
 \mathbb E_{\sigma} \Bigl[\sum_{(t_1, t_2) \in \mathcal S_u} T_{t_1t_2}(k) \Bigr] &= \sum_{t_1 \in \mathcal  A_u} \mathbb E_{\sigma} \Bigl[ \sum_{t_2 \in \mathcal B_{t_1}} T_{t_1t_2}(k) \Bigr]. \label{first order reduction} 
\end{align} 
\item \label{size of A_u} The description of $\mathcal A_u$ yields the following bound on its cardinality:  
\[ \# (\mathcal A_u)\lesssim \Bigl(\frac{k}{M^N}\Bigr) M^{d(N-h(u))} \lesssim M^{R-dh(u) + (d-1)N}.  \] 
\end{enumerate}
\end{lemma} 

\begin{figure}[h!]
\setlength{\unitlength}{0.8mm}
\begin{picture}(-50,10)(-96,10)

        \allinethickness{0.5mm}\path(-45,20)(-45,-70)(45,-70)(45,20)(-45,20)
	\allinethickness{0.5mm}\path(-45,-10)(45,-10)
	\path(-15,20)(-15,-70)
	\path(-45,-40)(45,-40)
	\path(15,20)(15,-70)
	\allinethickness{0.1mm}\path(-17,20)(-17,-70)
	\path(-19,20)(-19,-70)
	\path(-13,20)(-13,-70)
	\path(-11,20)(-11,-70)
	\path(-43,20)(-43,-70)
	\path(-41,20)(-41,-70)
	\path(13,20)(13,-70)
	\path(11,20)(11,-70)
	\path(17,20)(17,-70)
	\path(19,20)(19,-70)
	\path(41,20)(41,-70)
	\path(43,20)(43,-70)
	\path(-45,-42)(45,-42)
	\path(-45,-44)(45,-44)
	\path(-45,-66)(45,-66)
	\path(-45,-68)(45,-68)

	\path(-45,18)(45,18)
	\path(-45,16)(45,16)
	\path(-45,-38)(45,-38)
	\path(-45,-36)(45,-36)
	\path(-45,-8)(45,-8)
	\path(-45,-6)(45,-6)
	\path(-45,-12)(45,-12)
	\path(-45,-14)(45,-14)

	\path(-39,20)(-39,16)
	\path(-37,20)(-37,16)
	\path(-35,20)(-35,16)
	\path(-33,20)(-33,16)
	\path(-31,20)(-31,16)
	\path(-29,20)(-29,16)
	\path(-27,20)(-27,16)
	\path(-25,20)(-25,16)
	\path(-23,20)(-23,16)
	\path(-21,20)(-21,16)
	\path(-19,20)(-19,16)
	\path(-17,20)(-17,16)

	\path(-39,-44)(-39,-36)
	\path(-37,-44)(-37,-36)
	\path(-35,-44)(-35,-36)
	\path(-33,-44)(-33,-36)
	\path(-31,-44)(-31,-36)
	\path(-29,-44)(-29,-36)
	\path(-27,-44)(-27,-36)
	\path(-25,-44)(-25,-36)
	\path(-23,-44)(-23,-36)
	\path(-21,-44)(-21,-36)
	\path(-19,-44)(-19,-36)
	\path(-17,-44)(-17,-36)

	\path(-39,-70)(-39,-66)
	\path(-37,-70)(-37,-66)
	\path(-35,-70)(-35,-66)
	\path(-33,-70)(-33,-66)
	\path(-31,-70)(-31,-66)
	\path(-29,-70)(-29,-66)
	\path(-27,-70)(-27,-66)
	\path(-25,-70)(-25,-66)
	\path(-23,-70)(-23,-66)
	\path(-21,-70)(-21,-66)
	\path(-19,-70)(-19,-66)
	\path(-17,-70)(-17,-66)

	\path(9,20)(9,16)
	\path(7,20)(7,16)
	\path(5,20)(5,16)
	\path(3,20)(3,16)
	\path(1,20)(1,16)
	\path(-1,20)(-1,16)
	\path(-3,20)(-3,16)
	\path(-5,20)(-5,16)
	\path(-7,20)(-7,16)
	\path(-9,20)(-9,16)
	\path(-11,20)(-11,16)
	\path(-13,20)(-13,16)

	\path(39,20)(39,16)
	\path(37,20)(37,16)
	\path(35,20)(35,16)
	\path(33,20)(33,16)
	\path(31,20)(31,16)
	\path(29,20)(29,16)
	\path(27,20)(27,16)
	\path(25,20)(25,16)
	\path(23,20)(23,16)
	\path(21,20)(21,16)

	\path(39,-6)(39,-14)
	\path(37,-6)(37,-14)
	\path(35,-6)(35,-14)
	\path(33,-6)(33,-14)
	\path(31,-6)(31,-14)
	\path(29,-6)(29,-14)
	\path(27,-6)(27,-14)
	\path(25,-6)(25,-14)
	\path(23,-6)(23,-14)
	\path(21,-6)(21,-14)

	\path(9,-44)(9,-36)
	\path(7,-44)(7,-36)
	\path(5,-44)(5,-36)
	\path(3,-44)(3,-36)
	\path(1,-44)(1,-36)
	\path(-1,-44)(-1,-36)
	\path(-3,-44)(-3,-36)
	\path(-5,-44)(-5,-36)
	\path(-7,-44)(-7,-36)
	\path(-9,-44)(-9,-36)
	\path(-11,-44)(-11,-36)
	\path(-13,-44)(-13,-36)

	\path(9,-70)(9,-66)
	\path(7,-70)(7,-66)
	\path(5,-70)(5,-66)
	\path(3,-70)(3,-66)
	\path(1,-70)(1,-66)
	\path(-1,-70)(-1,-66)
	\path(-3,-70)(-3,-66)
	\path(-5,-70)(-5,-66)
	\path(-7,-70)(-7,-66)
	\path(-9,-70)(-9,-66)
	\path(-11,-70)(-11,-66)
	\path(-13,-70)(-13,-66)

	\path(-45,14)(-41,14)
	\path(-45,12)(-41,12)
	\path(-45,10)(-41,10)
	\path(-45,8)(-41,8)
	\path(-45,6)(-41,6)
	\path(-45,4)(-41,4)
	\path(-45,2)(-41,2)
	\path(-45,0)(-41,0)
	\path(-45,-2)(-41,-2)
	\path(-45,-4)(-41,-4)

	\path(17,-44)(17,-36)
	\path(19,-44)(19,-36)
	\path(21,-44)(21,-36)
	\path(23,-44)(23,-36)
	\path(25,-44)(25,-36)
	\path(27,-44)(27,-36)
	\path(29,-44)(29,-36)
	\path(31,-44)(31,-36)
	\path(33,-44)(33,-36)
	\path(35,-44)(35,-36)
	\path(37,-44)(37,-36)
	\path(39,-44)(39,-36)

	\path(17,-70)(17,-66)
	\path(19,-70)(19,-66)
	\path(21,-70)(21,-66)
	\path(23,-70)(23,-66)
	\path(25,-70)(25,-66)
	\path(27,-70)(27,-66)
	\path(29,-70)(29,-66)
	\path(31,-70)(31,-66)
	\path(33,-70)(33,-66)
	\path(35,-70)(35,-66)
	\path(37,-70)(37,-66)
	\path(39,-70)(39,-66)

	\path(-45,-16)(-41,-16)
	\path(-45,-18)(-41,-18)
	\path(-45,-20)(-41,-20)
	\path(-45,-22)(-41,-22)
	\path(-45,-24)(-41,-24)
	\path(-45,-26)(-41,-26)
	\path(-45,-28)(-41,-28)
	\path(-45,-30)(-41,-30)
	\path(-45,-32)(-41,-32)
	\path(-45,-34)(-41,-34)

	\path(-45,-46)(-41,-46)
	\path(-45,-48)(-41,-48)
	\path(-45,-50)(-41,-50)
	\path(-45,-52)(-41,-52)
	\path(-45,-54)(-41,-54)
	\path(-45,-56)(-41,-56)
	\path(-45,-58)(-41,-58)
	\path(-45,-60)(-41,-60)
	\path(-45,-62)(-41,-62)
	\path(-45,-64)(-41,-64)

	\path(19,14)(11,14)
	\path(19,12)(11,12)
	\path(19,10)(11,10)
	\path(19,8)(11,8)
	\path(19,6)(11,6)
	\path(19,4)(11,4)
	\path(19,2)(11,2)
	\path(19,0)(11,0)
	\path(19,-2)(11,-2)
	\path(19,-4)(11,-4)

	\path(45,14)(41,14)
	\path(45,12)(41,12)
	\path(45,10)(41,10)
	\path(45,8)(41,8)
	\path(45,6)(41,6)
	\path(45,4)(41,4)
	\path(45,2)(41,2)
	\path(45,0)(41,0)
	\path(45,-2)(41,-2)
	\path(45,-4)(41,-4)

	\path(19,-16)(11,-16)
	\path(19,-18)(11,-18)
	\path(19,-20)(11,-20)
	\path(19,-22)(11,-22)
	\path(19,-24)(11,-24)
	\path(19,-26)(11,-26)
	\path(19,-28)(11,-28)
	\path(19,-30)(11,-30)
	\path(19,-32)(11,-32)
	\path(19,-34)(11,-34)

	\path(45,-16)(41,-16)
	\path(45,-18)(41,-18)
	\path(45,-20)(41,-20)
	\path(45,-22)(41,-22)
	\path(45,-24)(41,-24)
	\path(45,-26)(41,-26)
	\path(45,-28)(41,-28)
	\path(45,-30)(41,-30)
	\path(45,-32)(41,-32)
	\path(45,-34)(41,-34)

	\path(19,-46)(11,-46)
	\path(19,-48)(11,-48)
	\path(19,-50)(11,-50)
	\path(19,-52)(11,-52)
	\path(19,-54)(11,-54)
	\path(19,-56)(11,-56)
	\path(19,-58)(11,-58)
	\path(19,-60)(11,-60)
	\path(19,-62)(11,-62)
	\path(19,-64)(11,-64)

	\path(45,-46)(41,-46)
	\path(45,-48)(41,-48)
	\path(45,-50)(41,-50)
	\path(45,-52)(41,-52)
	\path(45,-54)(41,-54)
	\path(45,-56)(41,-56)
	\path(45,-58)(41,-58)
	\path(45,-60)(41,-60)
	\path(45,-62)(41,-62)
	\path(45,-64)(41,-64)

	\path(-19,14)(-11,14)
	\path(-19,12)(-11,12)
	\path(-19,10)(-11,10)
	\path(-19,8)(-11,8)
	\path(-19,6)(-11,6)
	\path(-19,4)(-11,4)
	\path(-19,2)(-11,2)
	\path(-19,0)(-11,0)
	\path(-19,-2)(-11,-2)
	\path(-19,-4)(-11,-4)

	\path(-19,-16)(-11,-16)
	\path(-19,-18)(-11,-18)
	\path(-19,-20)(-11,-20)
	\path(-19,-22)(-11,-22)
	\path(-19,-24)(-11,-24)
	\path(-19,-26)(-11,-26)
	\path(-19,-28)(-11,-28)
	\path(-19,-30)(-11,-30)
	\path(-19,-32)(-11,-32)
	\path(-19,-34)(-11,-34)

	\path(-19,-46)(-11,-46)
	\path(-19,-48)(-11,-48)
	\path(-19,-50)(-11,-50)
	\path(-19,-52)(-11,-52)
	\path(-19,-54)(-11,-54)
	\path(-19,-56)(-11,-56)
	\path(-19,-58)(-11,-58)
	\path(-19,-60)(-11,-60)
	\path(-19,-62)(-11,-62)
	\path(-19,-64)(-11,-64)

	\path(-39,-6)(-39,-14)
	\path(-37,-6)(-37,-14)
	\path(-35,-6)(-35,-14)
	\path(-33,-6)(-33,-14)
	\path(-31,-6)(-31,-14)
	\path(-29,-6)(-29,-14)
	\path(-27,-6)(-27,-14)
	\path(-25,-6)(-25,-14)
	\path(-23,-6)(-23,-14)
	\path(-21,-6)(-21,-14)
	\path(-19,-6)(-19,-14)
	\path(-17,-6)(-17,-14)

	\path(9,-6)(9,-14)
	\path(7,-6)(7,-14)
	\path(5,-6)(5,-14)
	\path(3,-6)(3,-14)
	\path(1,-6)(1,-14)
	\path(-1,-6)(-1,-14)
	\path(-3,-6)(-3,-14)
	\path(-5,-6)(-5,-14)
	\path(-7,-6)(-7,-14)
	\path(-9,-6)(-9,-14)
	\path(-11,-6)(-11,-14)
	\path(-13,-6)(-13,-14)

%%%%%%%%%%%%%%%

	\path(-48,20)(-50,20)(-50,-40)
	\path(-50,-40)(-50,-70)(-48,-70)
	\put(-70,-27){\shortstack{$M^{-h(u)}$}}
	\path(47,-6)(49,-6)(49,-10)(47,-10)
	\put(52,-10){\shortstack{$kM^{-N-h(u)}$}}
\end{picture}
\vspace{7cm}
\caption{\label{A_u figure}
A diagram of $\mathcal A_u$ when $d = 2$, $M = 3$. Here the largest square is $u$. The shaded area depicts $\mathcal A_u$. The finest squares are the root cubes contained in $\mathcal A_u$.} 
\end{figure}
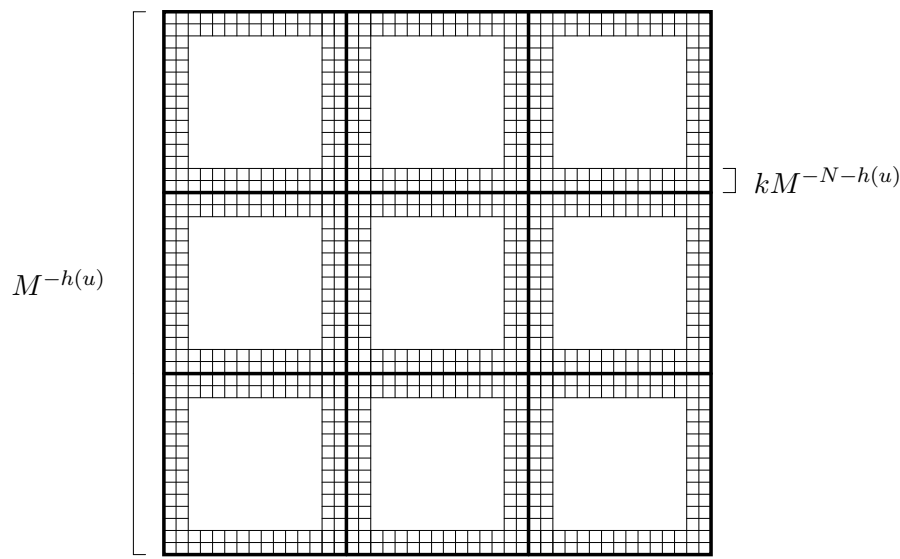

\begin{proof} 
We observe that $T_{t_1t_2}(k) = 1$ if and only if there exists a point $p = (p_1, \cdots, p_{d+1}) \in Z_k$ and $v_1, v_2 \in \Omega_N$ such that $p \in \mathcal P_{t_1, v_1} \cap \mathcal P_{t_2, v_2}$, and $\sigma(t_1) = v_1$, $\sigma(t_2) = v_2$. By Lemma \ref{intersection criterion lemma}, this implies that 
\begin{equation} \label{tube intersection condition}
|\text{cen}(t_1) - \text{cen}(t_2) + p_1(\sigma(t_1) - \sigma(t_2))| \leq 2\kappa_d \sqrt{d} M^{-N},  
\end{equation}   
where $\text{cen}(t_i)$ denotes the centre of the cube $t_i$. For $p \in Z_k$, \eqref{tube intersection condition} yields
\begin{align} 
|\text{cen}(t_1) - \text{cen}(t_2)| &\leq p_1|\sigma(t_1) - \sigma(t_2)| + 2\kappa_d \sqrt{d} M^{-N}  \lesssim p_1|\sigma(t_1) - \sigma(t_2)| \nonumber \\ &\lesssim \Bigl(\frac{k+1}{M^N} \Bigr) |\sigma(t_1) - \sigma(t_2)| \lesssim \Bigl(\frac{k}{M^N}\Bigr) M^{-h(D(\tau(t_1), \tau(t_2) ))} \nonumber \\ &\lesssim k M^{-N-h(u)}. \label{A_u B_t main inequality}
\end{align} 
The second inequality in the steps above follows from Corollary \ref{which is bigger corollary}, the third from the definition of $Z_k$ and the fourth from the property (\ref{sigma Lipschitz-type}) of the slope assignment. Here $\tau$ is the unique sticky map that generates $\sigma$, as specified in Proposition \ref{slope assignment}. Since $\tau$ preserves heights and lineages, $h(D(\tau(t_1), \tau(t_2))) \geq h(D(t_1, t_2)) = h(u)$, and the last step follows. 

The inequality in (\ref{A_u B_t main inequality}) implies that $T_{t_1t_2}(k)= 0$ unless $t_2 \in \mathcal B_{t_1}$. Further, $t_1, t_2$ lie in distinct children of $u$, so $t_1$ must satisfy 
\[ \text{dist}(t_1, \partial u') \lesssim \frac{k}{M^N} M^{-h(u)} \quad \text{ for some child $u'$ of $u$}, \] 
to allow for the existence of some $t_2$ obeying (\ref{A_u B_t main inequality}).   This means $t_1 \in \mathcal A_u$, proving (\ref{reduction to A_u}).

For (\ref{size of A_u}) we observe that $u$ has $M^d$ children. The Lebesgue measure of the set
\begin{equation} \label{Euclidean A_u}
\bigcup_{u'} \Bigl\{ x \in u' : \text{dist}(x, \partial u') \lesssim kM^{-N-h(u)}, u' \text{ is a child of } u  \Bigr\} 
\end{equation} 
is therefore $\lesssim (M^d) kM^{-N-h(u)} M^{-(d-1)h(u)}$. The cardinality of $\mathcal A_u$ is comparable to the number of $M^{-N}$-separated points in the set (\ref{Euclidean A_u}), and (\ref{size of A_u}) follows.     
\end{proof}
Our next task is to make further reductions to the expression on the right hand side of \eqref{first order reduction} that will enable us to invoke the probability estimates from Section \ref{probability estimation section}. To this end,  let us fix $Z_k$, $t_1 \in \mathcal A_u(k)$, $v_1 = \gamma(\alpha_1) \in \Omega_N$, and define a collection of point-slope pairs 
\begin{equation}  \mathcal E_u(t_1, v_1; k) :=  \left\{(t_2, v_2) \Biggl| \begin{aligned} &t_2 \in \mathcal T_N([0,1)^d) \cap \mathcal B_{t_1}, \; v_2 = \gamma(\alpha_2) \in \Omega_N,  \\ &\; h(t_2) = h(\alpha_2) = N, \; u = D(t_1, t_2), \\ &\mathcal P_{t_1, v_1} \cap \mathcal P_{t_2, v_2} \cap Z_k \ne \emptyset, \; h(D(\alpha_1, \alpha_2)) \geq h(u) \end{aligned} \right\}. \label{defn E_u}\end{equation} 
Thus $\mathcal E_u(t_1, v_1; k)$ is non-random as well. The significance of this collection is clarified in the next lemma.  
\begin{lemma} \label{defn of E(t_1, v_1) lemma} 
For $(t_2, v_2) \in \mathcal E_u(t_1, v_1;k)$ described as in \eqref{defn E_u}, define a random variable $\overline{T}_{t_2v_2}(t_1, v_1;k)$ as follows: 
\begin{equation} 
\overline{T}_{t_2v_2}(t_1, v_1;k) := \begin{cases} 1 &\text{ if }  \sigma(t_2) = v_2, \\ 0 &\text{ otherwise. }\end{cases}  \label{defn Tbar} \end{equation}  
\begin{enumerate}[(a)]
\item \label{reduction to E(t_1, v_1)} The random variables $T_{t_1t_2}(k)$ and $\overline{T}_{t_2v_2}(t_1, v_1;k)$ are related as follows: given $\sigma(t_1) = v_1$, 
\begin{equation} \label{T and Tbar}
T_{t_1t_2}(k) = \sup \bigl\{\overline{T}_{t_2v_2}(t_1, v_1;k) : (t_2, v_2) \in \mathcal E_u(t_1, v_1;k) \bigr\}.
\end{equation} 
In particular under the same conditional hypothesis $\sigma(t_1) = v_1$,  one obtains the bound
\begin{equation} T_{t_1t_2}(k) \leq \sum_{\begin{subarray}{c}v_2 \in \Omega_N \\ (t_2, v_2) \in \mathcal E_u(t_1, v_1;k) \end{subarray}} \overline{T}_{t_2v_2}(t_1, v_1;k), \label{second order reduction} \end{equation} 
which in turn implies  
\begin{equation} \mathbb E_{\sigma} \Bigl[ \sum_{t_2 \in \mathcal B_{t_1}} T_{t_1t_2}(k) \Bigl| \sigma(t_1) = v_1 \Bigr] \leq \sum_{(t_2, v_2) \in \mathcal E_u(t_1, v_1;k)}  \text{Pr}(\sigma(t_2) = v_2 \bigl| \sigma(t_1) = v_1\bigr). \label{conditional expectation estimate} \end{equation}  
\item \label{size of E(t_1, v_1)} The cardinality of $\mathcal E_u(t_1, v_1; k)$ is $\lesssim 2^{N-h(u)}$. 
\end{enumerate}
\end{lemma} 
\begin{proof}
We already know from Lemma \ref{dist to bdry lemma} that $T_{t_1t_2}(k) = 0$ unless $t_2 \in \mathcal B_{t_1}$. Further, if $\sigma(t_1) = v_1$ is known, then it is clear that $T_{t_1t_2} (k) = 1$ if and only if there exists $v_2 \in \Omega_N$ such that $\mathcal P_{t_1, v_1} \cap \mathcal P_{t_2, v_2} \cap Z_k \ne \emptyset$ and $\sigma(t_2) = v_2$. But this means that the sticky map $\tau$ that generates $\sigma$ must map $t_2$ to the $N$-long binary sequence that identifies $\alpha_2$. Stickiness dictates that $h(D(\alpha_1, \alpha_2)) = h(D(\tau(t_1), \tau(t_2)))\geq h(D(t_1, t_2)) = h(u)$, explaining the constraints that define $\mathcal E_u(t_1, v_1;k)$. Rephrasing the discussion above, given $\sigma(t_1) = v_1$, the event $T_{t_1t_2}(k) = 1$ holds if and only if there exists $v_2 \in \Omega_N$ such that $(t_2, v_2) \in \mathcal E_u(t_1, v_1;k)$ and $\sigma(t_2) = v_2$. This is the identity claimed in \eqref{T and Tbar} of part (\ref{reduction to E(t_1, v_1)}). The bound in \eqref{second order reduction} follows easily from \eqref{T and Tbar} since the supremum is dominated by the sum. The final estimate \eqref{conditional expectation estimate} in part (\ref{reduction to E(t_1, v_1)}) follows by taking conditional expectation of both sides of \eqref{second order reduction}, and observing that $\mathbb E_{\sigma}(\overline{T}_{t_2v_2}(t_1, v_1;k) |\sigma(t_1) = v_1) =   \text{Pr}(\sigma(t_2) = v_2 \bigl| \sigma(t_1) = v_1\bigr)$.

We turn to (\ref{size of E(t_1, v_1)}). If $v_2 \in \Omega_N$ is fixed, then it follows from Corollary \ref{counting t_2 given t_1, v_1, v_2} (taking $Q$ in that corollary to be the cube of sidelength $O(M^{-N})$ containing $\mathcal P_{t_1, v_1} \cap Z_k$) that there exist at most a constant number of choices of $t_2$ such that $(t_2, v_2) \in \mathcal E_u(t_1, v_1;k)$. But by Corollary \ref{corollary - finite binary structure} the number of points $\alpha_2 \in \mathcal D_M^{[N]}$ (and hence slopes $v_2 \in \Omega_N$) that obey $h(D(\alpha_1, \alpha_2)) \geq h(u)$ is no more than $2^{N - h(u)}$, proving the claim.    
\end{proof} 
The same argument above applied twice yields the following conclusion, the verification of which is left to the reader. 
\begin{corollary} \label{two copies corollary}
Given $t_1 \in \mathcal A_u(k)$, $t_1' \in \mathcal A_{u'}(k')$, $v_1, v_1' \in \Omega_N$, define $\mathcal E_u(t_1, v_1;k)$ and $\mathcal E_{u'}(t_1', v_1';k')$ as in \eqref{defn E_u} and the random variables $\overline{T}_{t_2v_2}(t_1, v_1;k)$, $\overline{T}_{t_2'v_2'}(t_1', v_1';k')$ as in \eqref{defn Tbar}. Then given $\sigma(t_1) = v_1$ and $\sigma(t_1') = v_1'$, 
\begin{equation*} 
\sum_{\begin{subarray}{c} t_2 \in \mathcal B_{t_1} \\ t_2' \in \mathcal B_{t_1'} \end{subarray}} T_{t_1t_2}(k) T_{t_1't_2'}(k') \leq \overset{\ast}{\sum} \overline{T}_{t_2v_2}(t_1, v_1;k) \overline{T}_{t_2'v_2'}(t_1', v_1';k'),
\end{equation*}  
where the notation $\overset{\ast}{\sum}$ represents the sum over all indices $\{(t_2, v_2); (t_2', v_2')\} \in \mathcal E_u(t_1,v_1; k) \times \mathcal E_{u'}(t_1',v_1'; k')$.
\end{corollary} 
We are now ready to establish the key estimates in the proofs of Propositions \ref{proposition slab estimate} and \ref{proposition slab estimate second moment}. 
\begin{lemma} \label{proof of step 2 estimate lemma} 
The estimate in \eqref{intersection estimate step 2} holds. 
\end{lemma} 
\begin{proof}
We combine the steps outlined in Lemmas \ref{dist to bdry lemma}, \ref{defn of E(t_1, v_1) lemma} and \ref{probability estimate}. By Lemma \ref{dist to bdry lemma}(\ref{reduction to A_u}), 
\begin{equation} \label{towards step 2 estimate}
\begin{aligned}
\mathbb E_{\sigma} \Bigl[ \sum_{(t_1, t_2) \in \mathcal S_u} T_{t_1t_2}(k)\Bigr] &= \sum_{t_1 \in \mathcal A_u} \mathbb E_{\sigma} \Bigl[ \sum_{t_2 \in \mathcal B_{t_1}} T_{t_1t_2}(k)\Bigr] \\ &=   \sum_{t_1 \in \mathcal A_u} \mathbb E_{v_1} \mathbb E_{\sigma} \Bigl[ \sum_{t_2 \in \mathcal B_{t_1}} T_{t_1t_2}(k) \Bigl| \sigma(t_1) = v_1 \Bigr]. 
\end{aligned}  
\end{equation} 
Applying \eqref{conditional expectation estimate} from Lemma \ref{defn of E(t_1, v_1) lemma} followed by Lemma \ref{probability estimate}, we find that the inner expectation above obeys the bound
\begin{align*} \mathbb E_{\sigma} \Bigl[ \sum_{t_2 \in \mathcal B_{t_1}} T_{t_1t_2}(k) \bigl| \sigma(t_1) = v_1 \Bigr]  &\leq \sum_{(t_2,v_2) \in \mathcal E_u(t_1, v_1; k)} \text{Pr}(\sigma(t_2) = v_2 | \sigma(t_1) = v_1 ) \\ &\leq \#(\mathcal E_u(t_1, v_1;k)) \times \underbrace{2^{-N + h(u)}}_{\text{Lemma \ref{probability estimate}}} \\&\lesssim \underbrace{2^{N- h(u)}}_{\text{Lemma \ref{defn of E(t_1, v_1) lemma}(\ref{size of E(t_1, v_1)})}} \times 2^{-N+h(u)} \lesssim 1, \end{align*}
uniformly in $v_1$. Inserting this back into \eqref{towards step 2 estimate}, we arrive at
\[\mathbb E_{\sigma} \Bigl[ \sum_{(t_1, t_2) \in \mathcal S_u} T_{t_1t_2}(k)\Bigr] \lesssim \#(\mathcal A_u), \] 
which according to Lemma \ref{dist to bdry lemma}(\ref{size of A_u}) is the bound claimed in \eqref{intersection estimate step 2}. 
\end{proof}

\begin{lemma}\label{proof of step 2 estimate second moment lemma}
The estimate in \eqref{intersection estimate second moment step 2} holds. 
\end{lemma}
\begin{proof} 
The proof of \eqref{intersection estimate second moment step 2} shares many similarities with that of Lemma \ref{proof of step 2 estimate lemma}, except that there are now two copies of each of the objects appearing in the proof of \eqref{intersection estimate step 2} and the probability estimate comes from Lemma \ref{probability estimate for second moment lemma} instead of Lemma \ref{probability estimate}.  We outline the main steps below.  

In view of Lemma \ref{probability estimate for second moment lemma} and after a permutation of $(t_1, t_2)$ and of $(t_1', t_2')$ if necessary, we may assume that for every $\mathbb I = \{(t_1, t_2); (t_1', t_2') \} \in \mathfrak I_{41}(u,u')$, 
\begin{equation} \label{main prob est for S_{42}} 
\text{Pr}\bigl( \sigma(t_2) = v_2, \; \sigma(t_2') = v_2'  | \sigma(t_1) = v_1, \; \sigma(t_1') = v_1' \bigr) = \left(\frac{1}{2} \right)^{2N - h(u) - h(u')}.
\end{equation}  
Now, 
\begin{align*}
\mathbb E_{\sigma} &\bigl( \mathfrak S_{41}(u,u'; k,k'\bigr) \\ &\leq M^{-2(d+1)N}\mathbb E_{\sigma}\Bigl[ \sum_{\mathbb I \in \mathfrak I_{41}(u,u')} T_{t_1t_2}(k)  T_{t_1't_2'}(k')\Bigr] \\ &= M^{-2(d+1)N}\sum_{\begin{subarray}{c}  t_1 \in \mathcal A_u(k) \\ t_1' \in \mathcal A_{u'}(k')\end{subarray}} \mathbb E_{v_1, v_1'} \mathbb E_{\sigma} \Bigl[ \sum_{\begin{subarray}{c} t_2 \in \mathcal B_{t_1} \\ t_2' \in \mathcal B_{t_1'}\end{subarray}} T_{t_1t_2}(k) T_{t_1't_2'}(k') \Bigl| \sigma(t_1) = v_1, \sigma(t_1') = v_1' \Bigr] \\ 
& \lesssim M^{-2(d+1)N}\underbrace{\left( \frac{kk'}{M^{2N}} M^{d(2N-h(u) - h(u'))}\right)}_{\#(t_1, t_1') \text{ from Lemma } \ref{dist to bdry lemma}} \lesssim M^{2R-4N - d(h(u) + h(u'))},  
\end{align*} 
since according to Corollary \ref{two copies corollary}
\begin{align*} 
\mathbb E_{\sigma} \Bigl[ \sum_{\begin{subarray}{c} (t_2, t_2') \in \mathcal B_{t_1} \times  \mathcal B_{t_1'}\end{subarray}} &T_{t_1t_2}(k) T_{t_1't_2'}(k') \Bigl|\sigma(t_1)  = v_1, \sigma(t_1') = v_1' \Bigr]  \\ &\leq \mathbb E_{\sigma} \Bigl[\sum^{\ast} \overline{T}_{t_2 v_2}(t_1, v_1;k) \overline{T}_{t_2', v_2'}(t_1', v_1';k') \Bigl| \sigma(t_1) = v_1, \; \sigma(t_1') = v_1'\Bigr]\\&\lesssim \sum^{\ast} \text{Pr}(\sigma(t_2) = v_2, \; \sigma(t_2') = v_2'\; | \; \sigma(t_1) = v_1, \sigma(t_1') = v_1' ) \\ 
&\lesssim \underbrace{{(2^{N- h(u)})}}_{\#(\mathcal E_{u}(t_1, v_1;k))} \,
 \times \, \underbrace{{(2^{N- h(u')})}}_{\#(\mathcal E_{u'}(t_1', v_1';k'))} \,
 \times \, \underbrace{{(2^{-2N+h(u)+h(u')})}}_{\eqref{main prob est for S_{42}} \text{via Lemma} \ref{probability estimate for second moment lemma}}  \\
 &\lesssim 1, \quad \text{ uniformly in } v_1, v_1'.
\end{align*} 
The proof is therefore complete.
\end{proof} 
\begin{lemma}\label{second moment lemma S_{42} estimate} 
The estimate in \eqref{S_{42}(u, u_1, u_2) estimate} holds. 
\end{lemma}
\begin{proof}
The proof of \eqref{S_{42}(u, u_1, u_2) estimate} is similar to \eqref{intersection estimate second moment step 2}, and in certain respects simpler. But the configuration type dictates that we set up a different class $\mathcal E^{\ast}$ of point-slope tuples that will play a  role analogous to $\mathcal E(t_1, v_1; k)$ in the preceding lemmas. Recall the structure of a type 2 configuration from Figure~\ref{Fig:Type 2 configurations} and the definition of $\mathfrak I_{42}(u,u_1, u_2)$ from \eqref{defn I_{42}(u,u_1,u_2)}. Given root cubes $t_2, t_2'$, and $u, u_1, u_2 \in \mathcal T_N([0,1)^d)$ with the property that
\[ u_1 \subseteq u,\; u_2 \subsetneq u, \quad u_2 = D(t_2, t_2'), \quad h(u) \leq h(u_1) \leq h(u_2) \leq N = h(t_2) = h(t_2'), \] 
and slopes $v_2 = \gamma(\alpha_2)$, $v_2' = \gamma(\alpha_2') \in \Omega_N$, we define $\mathcal E^{\ast}$ (depending on all these objects) to be the following collection of root-slope tuples: 
\begin{equation} \mathcal E^{\ast} := \left\{\{(t_1, v_1);(t_1', v_1')\}  \Biggl| 
\begin{aligned} &\mathbb I = \{(t_1, t_2);(t_1', t_2') \} \in \mathfrak I_{42}(u, u_1, u_2), \\ & v_1 = \gamma(\alpha_1), \; v_1' = \gamma(\alpha_1') \text{ for some } \alpha_1, \alpha_1' \in \mathcal D_M^{[N]},  \\ &\mathcal P_{t_1, v_1} \cap \mathcal P_{t_2, v_2} \cap Z_k \ne \emptyset, \; \mathcal P_{t_1', v_1'} \cap \mathcal P_{t_2', v_2'} \cap Z_{k'} \ne \emptyset, \\ &\text{$\{(t_i, \alpha_i), (t_i', \alpha_i') : i=1, 2 \}$ is sticky-admissible }. \end{aligned}  \right\} 
\label{double intersection} \end{equation} 
The relevance of $\mathcal E^{\ast}$ is this: if $\sigma(t_2) = v_2$ and $\sigma(t_2') = v_2'$ are given, then $T_{t_1t_2}(k) T_{t_1't_2'}(k') = 0$ unless there exist $v_1, v_1' \in \Omega_N$ with $\{(t_1, v_1);(t_1', v_1') \} \in \mathcal E^{\ast}$ and $\sigma(t_1) = v_1$, $\sigma(t_1') = v_1'$. 

We first set about obtaining a bound on the size of $\mathcal E^{\ast}$ that we will need momentarily. Stickiness dictates that $h(D(\alpha_1, \alpha_2)) \geq h(u)$, and that $\alpha_1$ is an $N$th level descendant of $\alpha$, the ancestor of $\alpha_2$ at height $h(u)$. Thus the number of possible $\alpha_1$ (and hence $v_1$) is $\leq 2^{N-h(u)}$, by Corollary \ref{corollary - finite binary structure}. Again by stickiness, $h(D(\alpha_1, \alpha_1')) \geq h(u_1)$, so for a given $\alpha_1$, the number of $\alpha_1'$ (hence $v_1'$) is no more than the number of possible descendants of $\alpha^{\ast}$, the ancestor of $\alpha_1$ at height $h(u_1)$. This number is thus $\leq 2^{N-h(u_1)}$. Once $v_1, v_1'$ have been fixed (recall that $v_2, v_2', t_2, t_2'$ are already fixed), it follows from  Corollary \ref{counting t_2 given t_1, v_1, v_2} that the number of $t_1, t_1'$ obeying the intersection conditions in \eqref{double intersection} is $\lesssim 1$. Combining these, we arrive at the following bound on the cardinality of $\mathcal E^{\ast}$:  
\begin{equation}  \label{Estar size} 
\#(\mathcal E^{\ast}) \lesssim \bigl(2^{N-h(u)} \bigr) \bigl( 2^{N-h(u_1)}\bigr) = 2^{2N - h(u) - h(u_1)}. 
\end{equation}  

We use this bound on the size of $\mathcal E^{\ast}$ to estimate a conditional expectation, essentially the same way as in the previous two lemmas. 
\begin{align}
\mathbb E_{\sigma} \Bigl[ \sum_{\begin{subarray}{c} t_1, t_1' \\ \mathbb I \in \mathfrak I_{42}(u, u_1, u_2)\end{subarray}} &T_{t_1t_2}(k) T_{t_1't_2'}(k') \bigl| \sigma(t_2) = v_2, \sigma(t_2') = v_2' \Bigr]  \nonumber \\ &= \sum_{\mathcal E^{\ast}}\text{Pr}(\sigma(t_1) = v_1, \sigma(t_1') = v_1' | \sigma(t_2) = v_2, \sigma(t_2') = v_2' ) \nonumber  \\ 
&\lesssim \#(\mathcal E^{\ast}) \left( \frac{1}{2}\right)^{2N - h(u) - h(u_1)} \lesssim 1, \label{conditional exp for S_{42}}
\end{align}
where the last step follows by combining Lemma \ref{type 2 config lemma} with \eqref{Estar size}. As a result, we obtain
\begin{align*} 
\mathbb E_{\sigma}&\bigl(\mathfrak S_{42}(u,u_1, u_2; k, k') \bigr) \\ &= M^{-2(d+1)N} \mathbb E_{\sigma} \Bigl[\sum_{\mathbb I \in \mathfrak I_{42}(u,u_1, u_2)} T_{t_1t_2}(k) T_{t_1't_2'}(k') \Bigr]  \\ 
&\leq M^{-2(d+1)N} \sum_{t_2, t_2' \subseteq u_2} \mathbb E_{v_2, v_2'} \mathbb E_{\sigma} \Bigl[ \sum_{\begin{subarray}{c} t_1, t_1' \\ \mathbb I \in \mathfrak I_{42}(u, u_1, u_2)\end{subarray}} T_{t_1t_2}(k) T_{t_1't_2'}(k') \bigl| \sigma(t_2) = v_2, \sigma(t_2') = v_2' \Bigr]  \\ &\lesssim  M^{-2(d+1)N} \sum_{t_2, t_2' \subseteq u_2} 1 \\ &\lesssim M^{-2(d+1)N}\bigl( M^{-dh(u_2) + Nd}\bigr)^2,
\end{align*}  
where the estimate from \eqref{conditional exp for S_{42}} has been inserted in the third step above.  The final expression is the bound claimed in \eqref{S_{42}(u, u_1, u_2) estimate}.  
\end{proof} 
%%%%%%%%%%%%%%%%%%%%%%%%%%%%%%%%%%%%%%%%%%%%%%%%%%%%%%%%%%%%%%%%%%%%%%%

\section{Proposition \ref{Kakeya on average}: Proof of the upper bound \eqref{generic upper bound}}
\label{upperboundsection}
Using the theory developed in Section~\ref{percolation section}, we can establish inequality \eqref{generic upper bound} with $b_N = C_M/N$ as in Proposition~\ref{Kakeya on average} with relative ease.  For $x\in\mathbb{R}^{d+1}$, we write $x = (x_1,\overline{x})$, where $\overline{x} = (x_2,\ldots,x_{d+1})$.  Since the Kakeya-type set defined by \eqref{Kakeya sets} is contained in the parallelepiped $[C_0, C_0 + 1] \times [-2C_0,2C_0]^d$ , we may write 
\begin{align}\label{E: Probability}
	\mathbb{E}_{\sigma}\left|K_N(\sigma)\cap [C_0,C_0+1]\times\mathbb{R}^d\right| &= \mathbb E_{\sigma}\left(\int_{C_0}^{C_0+1}\int_{[-2C_0,2C_0]^d} {\bf 1}_{K_N(\sigma)}(x_1,\overline{x})d\overline{x}dx_1\right)\nonumber\\ 
&= \int_{C_0}^{C_0+1}\int_{[-2C_0,2C_0]^d} \mathbb E_{\sigma}\left({\bf 1}_{K_N(\sigma)}(x_1,\overline{x}) \right) \, d\overline{x}dx_1 \nonumber \\
	& = \int_{C_0}^{C_0+1}\int_{[-2C_0,2C_0]^d}\text{Pr}(x)\ d\overline{x}dx_1,
\end{align}
where $\text{Pr}(x)$ denotes the probability that the point $(x_1,\overline{x})$ is contained in the set $K_N(\sigma)$.  To establish inequality \eqref{generic upper bound} then, it suffices to show that this probability is bounded by a constant multiple of $1/N$, the constant being uniform in $x \in[C_0,C_0+1]\times\mathbb{R}^{d}$.

Let us recall the definition of Poss$(x)$ from \eqref{defn Poss(p)}. We would like to define a certain percolation process on the tree $\mathcal{T}_N(\text{Poss}(x))$ whose probability of survival can majorize $\text{Pr}(x)$.  By Lemma \ref{away from root hyperplane lemma}(\ref{exactly one slope per t}),  there corresponds to every $t \in$ Poss$(x)$ exactly one $v(t) \in \Omega_N$ such that $\mathcal P_{t, v(t)}$ contains $x$. Let us also recall that $v(t) = \gamma(\alpha(t))$ for some $\alpha(t) \in \mathcal D_M^{[N]}$. By Corollary \ref{corollary - finite binary structure}, $\alpha(t)$ is uniquely identified by $\beta(t) := \psi(\alpha(t))$, which is a deterministic sequence of length $N$ with entries 0 or 1. Here $\psi$ is the tree isomorphism described in Lemma~\ref{prop - binary structure}.   

Given a slope assignment $\sigma = \sigma_{\tau}$ generated by a sticky map $\tau: \mathcal T_N([0,1)^d) \rightarrow \mathcal T_N([0,1);2)$ as defined in Proposition~\ref{slope assignment} and a vertex $t = \langle  i_1, \cdots, i_N \rangle \in \mathcal T_N(\text{Poss}(x))$ with $h(t) = N$, we assign a value of $0$ or $1$ to each edge of the ray identifying $t$ as follows. Let $e$ be the edge identified by the vertex $\langle i_1, i_2, \cdots, i_k \rangle$. Set 
\begin{equation}  \label{percolation r.v.s} Y_e := \Biggl\{ \begin{aligned} 1 &\text{ if } \pi_k(\tau(t)) = \pi_k(\beta(t)), \\ 0 &\text{ if } \pi_k(\tau(t)) \ne \pi_k(\beta(t)). \end{aligned} \end{equation} 
To clarify the notation above, recall that both $\tau(t)$ and $\beta(t)$ are $N$-long binary sequences, and $\pi_k$ denotes the $k$th component of the input. Though the definition of $Y_e$ suggests a potential conflict for different choices of $t$, our next lemma confirms that this is not the case.   
\begin{lemma} \label{percolation description lemma} 
The description in \eqref{percolation r.v.s} is consistent in $t$; i.e., it assigns a uniquely defined binary random variable $Y_e$ to each edge of $\mathcal T_N(\text{Poss}(x))$. The collection $\{ Y_e \}$ is independent and identically distributed as Bernoulli$(\frac{1}{2})$ random variables.   
\end{lemma} 
\begin{proof}
Let $t, t' \in \mathcal T_N(\text{Poss}(x))$, $h(t) = h(t') = N$. Set $u = D(t,t')$, the youngest common ancestor of $t$ and $t'$. In order to verify consistency, we need to ascertain that for every edge $e$ in $\mathcal T_N(\text{Poss}(x))$ leading up to $u$ and for every sticky map $\tau$, the prescription (\ref{percolation r.v.s}) yields the same value of $Y_e$ whether we use $t$ or $t'$. Rephrasing this, it suffices to establish that 
\begin{equation} \label{consistency check}
\pi_k(\tau(t)) = \pi_k(\tau(t')) \quad \text{ and } \quad \pi_k(\beta(t)) = \pi_k(\beta(t')) \quad \text{ for all } 0 \leq k \leq h(u). 
\end{equation} 
Both equalities are consequences of the height and lineage-preserving property of sticky maps, by virtue of which \[ h(D(t,t')) \leq \min \bigl[ h(D(\tau(t), \tau(t'))), h(D(\beta(t), \beta(t')))\bigr]. \] Of these, stickiness of $\tau$ has been proved in Proposition \ref{sticky assignment}. The unambiguous definition and stickiness of $\beta$ has been verified in Lemma \ref{away from root hyperplane lemma - tree version}. 

For the remainder, we recall from Section~\ref{stickiness section} (see the discussion preceding Proposition \ref{sticky assignment}) that for $t = \langle i_1, i_2, \cdots, i_N \rangle$, the projection $\pi_k(\tau(t)) = X_{\langle i_1, \cdots, i_k \rangle}$ is a Bernoulli$(\frac{1}{2})$ random variable, so Pr$(Y_e = 1) = \frac{1}{2}$. Further the random variables $Y_e$ associated with distinct edges $e$ in $\mathcal T_N(\text{Poss}(x))$ are determined by distinct Bernoulli random variables of the form $X_{\langle i_1, \cdots, i_k \rangle}$. The stated independence of the latter collection implies the same for the former.        
\end{proof}  

Thus the collection $\mathbb{Y}_N = \{Y_e\}_{e\in\mathcal{E}}$ defines a Bernoulli percolation on $\mathcal{T}_N(\text{Poss}(x))$, where $\mathcal{E}$ is the edge set of $\mathcal{T}_N(\text{Poss}(x))$.  As described in Section~\ref{percolation}, the event $\{Y_e=0\}$ corresponds to the removal of the edge $e$ from $\mathcal{E}$, and the event $\{Y_e=1\}$ corresponds to retaining this edge.

\begin{lemma}\label{PossLemma}
Let $\text{Pr}(x) = \text{Pr}\{\tau : x \in K_N(\sigma_{\tau}) \} $ be as in \eqref{E: Probability}, and $\{Y_e \}$ as in \eqref{percolation r.v.s}. 
\begin{enumerate}[(a)]
\item For any $x \in [C_0, C_0+1] \times \mathbb R^d$, the event $\{ \tau : x \in K_N(\sigma_{\tau})\}$ is contained in 
\begin{equation}  \label{final event in percolation} \{\tau : \exists \text{ a full-length ray in } \mathcal T_N(\text{Poss}(x)) \text{ that survives percolation via } \{ Y_e \}  \}. \end{equation} 
\item As a result, $$\text{Pr}(x) \leq \text{Pr}\bigl(\text{survival after percolation on } \mathcal{T}_N(\text{Poss}(x))\bigr).$$
\end{enumerate} 
\end{lemma}
\begin{proof}
It is clear that $x \in K_N(\sigma_{\tau})$ if and only if there exists $t \in \text{Poss}(x)$ such that $\sigma_{\tau}(t) = v(t)$, where $v(t)$ is the unique slope in $\Omega_N$ prescribed by Lemma \ref{away from root hyperplane lemma}(\ref{exactly one slope per t}) for which $x \in \mathcal P_{t, v(t)}$. 
In other words, we have 
\begin{align}  
\{ \tau :  x \in K_N(\sigma_{\tau}) \} &= 
\bigcup\{\sigma(t)= v(t) : t\in\text{Poss}(x)\} \nonumber \\ &= \bigcup \{ \tau(t) = \beta(t) : t\in\text{Poss}(x)\}, \label{FormalPoss}
\end{align}
where the last step follows from the preceding one by unraveling the string of bijective mappings $\gamma^{-1}$, $\Phi^{-1}$ and $\psi$ (described in Proposition \ref{slope assignment}) that leads from $\sigma(t)$ to $\tau(t)$, and which incidentally also generates $\beta(t) = \langle j_1, \cdots, j_N \rangle \in \mathcal T([0,1);2)$ from $v(t)$. 
Since $t$ is identified by some sequence $\langle i_1,i_2,\ldots, i_N\rangle$, we have its associated random binary sequence $$\tau(t) = \langle X_{\langle i_1\rangle},X_{\langle i_1,i_2\rangle}, \ldots,X_{\langle i_1,i_2,\ldots,i_N\rangle}\rangle \in \mathcal T_N([0,1);2). $$ 
Using this, we can rewrite \eqref{FormalPoss} as follows:
\begin{align}\label{percolation event}
	&\bigcup_{t\in\text{Poss}(x)}\{\sigma(t)=v(t)\}\nonumber\\
	&= \bigcup_{t\in\text{Poss}(x)} \bigl\{\langle X_{\langle i_1\rangle},X_{\langle i_1,i_2\rangle}, \ldots,X_{\langle i_1,i_2,\ldots,i_N\rangle}\rangle = \langle j_1,j_2,\ldots,j_N\rangle\bigr\}\nonumber\\
	&= \bigcup_{t\in\text{Poss}(x)}\bigcap_{k=1}^N\{X_{\langle i_1,\ldots,i_k\rangle} = j_k\} \nonumber \\ &= \bigcup_{\mathcal{R} \leftrightarrow \langle i_1, \cdots, i_N \rangle \in\partial\mathcal{T}}\bigcap_{e \leftrightarrow \langle i_1,\ldots,i_k\rangle\in\mathcal{E}\cap \mathcal{R}}\{X_{\langle i_1,\ldots,i_k\rangle} - j_k = 0\} \nonumber \\ &=\bigcup_{\mathcal{R}\in\partial\mathcal{T}}\bigcap_{e \in \mathcal E \cap \mathcal R} \{Y_e = 1\}.
\end{align}
In the above steps we have set $\mathcal{T} := \mathcal{T}_N(\text{Poss}(x))$ for brevity and let $\mathcal{E}$ be the edge set of $\mathcal{T}$. The last step uses (\ref{percolation r.v.s}), and the final event is the same as the one in \eqref{final event in percolation}.  Using \eqref{percolation event}, we have
\begin{align}\label{tie to formal percolation}
	\text{Pr}(x) &\leq \text{Pr}\left(\bigcup_{\mathcal{R}\in\partial\mathcal{T}}\bigcap_{e \in \mathcal E \cap \mathcal R} \{Y_e = 1\} \right).
\end{align}
This last expression is obviously equivalent to the righthand side of \eqref{survival probability}, verifying the second part of the lemma.
\end{proof}

Our next task is therefore to estimate the survival probability of $\mathcal T_N(\text{Poss}(x))$ under Bernoulli$(\frac 12)$ percolation. For this purpose and in view of the discussion in Section \ref{survival probability section}, we should visualize $\mathcal T_N(\text{Poss}(x))$ as an electrical circuit, the resistance of an edge terminating at a vertex of height $k$ being $2^{k-1}$, per equation \eqref{resistance}.  Let us denote by $R(\text{Poss}(x))$ the resistance of the entire circuit.  In light of the theorem of Lyons, restated in the form of Proposition~\ref{survival prob}, it suffices to establish the following lemma.

\begin{lemma}\label{Resistance}
	With the resistance of $\text{Poss}(x)$ defined as above, we have 
\begin{equation}\label{ResistBound}
	R(\text{Poss}(x)) \gtrsim N.
\end{equation}
\end{lemma}

\begin{proof}
We begin by constructing a different electrical network from the one naturally associated to our tree $\text{Poss}(x)$.  For every $k\geq 1$, we connect all vertices at height $k$ by an ideal conductor to make one node $V_k$, as in Figure~\ref{FigResist}.  Call this new circuit $E$.

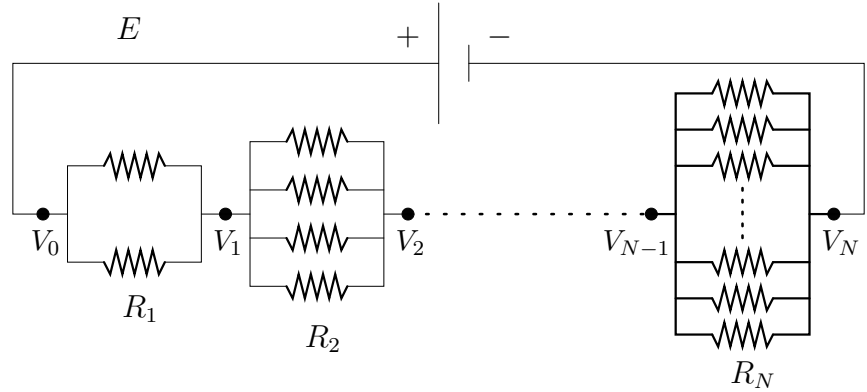
\begin{figure}[h]
\setlength{\unitlength}{.8mm}
\begin{picture}(0,0)(-100,35)

	\put(-58,29){\large\shortstack{$E$}}
	\special{sh 0.99}\put(-70,0){\ellipse{2}{2}}
	\put(-72,-6){\shortstack{$V_0$}}
	\special{sh 0.99}\put(-40,0){\ellipse{2}{2}}
	\put(-42,-6){\shortstack{$V_1$}}
	\special{sh 0.99}\put(-10,0){\ellipse{2}{2}}
	\put(-12,-6){\shortstack{$V_2$}}
	\special{sh 0.99}\put(30,0){\ellipse{2}{2}}
	\put(22,-6){\shortstack{$V_{N-1}$}}
	\special{sh 0.99}\put(60,0){\ellipse{2}{2}}
	\put(58,-6){\shortstack{$V_N$}} 

	\path(60,0)(65,0)(65,25)(0,25)
	\path(-70,0)(-75,0)(-75,25)(-5,25)
	\path(0,22)(0,28)
	\path(-5,15)(-5,35)
	\put(-12,28){\large\shortstack{$+$}}
	\put(3,28){\large\shortstack{$-$}} 

	\path(-70,0)(-66,0)
	\path(-60,-8)(-66,-8)(-66,8)(-60,8)
	\path(-40,0)(-44,0)
	\path(-50,-8)(-44,-8)(-44,8)(-50,8) 

	\path(-40,0)(-36,0)
	\path(-30,-12)(-36,-12)(-36,12)(-30,12)
	\path(-36,-4)(-30,-4)
	\path(-36,4)(-30,4)
	\path(-10,0)(-14,0)
	\path(-20,-12)(-14,-12)(-14,12)(-20,12) 
	\path(-20,-4)(-14,-4)
	\path(-20,4)(-14,4) 

	\thicklines\dottedline{3}(-10,0)(30,0)

	\path(30,0)(34,0)
	\path(40,-20)(34,-20)(34,20)(40,20)
	\path(40,-14)(34,-14)
	\path(40,-8)(34,-8)
	\path(40,8)(34,8)
	\path(40,14)(34,14)
	\path(60,0)(56,0)	
	\path(50,-20)(56,-20)(56,20)(50,20)
	\path(50,-14)(56,-14)
	\path(50,-8)(56,-8)
	\path(50,8)(56,8)
	\path(50,14)(56,14) 
	\thicklines\dottedline{2}(45,-4)(45,5) 

	\path(-60,-8)(-59,-10)(-58,-6)(-57,-10)(-56,-6)(-55,-10)(-54,-6)(-53,-10)(-52,-6)(-51,-10)(-50,-8)
	\path(-60,8)(-59,10)(-58,6)(-57,10)(-56,6)(-55,10)(-54,6)(-53,10)(-52,6)(-51,10)(-50,8)
	\put(-57,-17){\large\shortstack{$R_1$}}

	\path(-30,-12)(-29,-14)(-28,-10)(-27,-14)(-26,-10)(-25,-14)(-24,-10)(-23,-14)(-22,-10)(-21,-14)(-20,-12)
	\path(-30,12)(-29,14)(-28,10)(-27,14)(-26,10)(-25,14)(-24,10)(-23,14)(-22,10)(-21,14)(-20,12)
	\path(-30,-4)(-29,-6)(-28,-2)(-27,-6)(-26,-2)(-25,-6)(-24,-2)(-23,-6)(-22,-2)(-21,-6)(-20,-4)
	\path(-30,4)(-29,6)(-28,2)(-27,6)(-26,2)(-25,6)(-24,2)(-23,6)(-22,2)(-21,6)(-20,4)
	\put(-27,-22){\large\shortstack{$R_2$}} 

	\path(40,-20)(41,-22)(42,-18)(43,-22)(44,-18)(45,-22)(46,-18)(47,-22)(48,-18)(49,-22)(50,-20)
	\path(40,-14)(41,-16)(42,-12)(43,-16)(44,-12)(45,-16)(46,-12)(47,-16)(48,-12)(49,-16)(50,-14)
	\path(40,-8)(41,-10)(42,-6)(43,-10)(44,-6)(45,-10)(46,-6)(47,-10)(48,-6)(49,-10)(50,-8)
	\path(40,20)(41,22)(42,18)(43,22)(44,18)(45,22)(46,18)(47,22)(48,18)(49,22)(50,20)
	\path(40,14)(41,16)(42,12)(43,16)(44,12)(45,16)(46,12)(47,16)(48,12)(49,16)(50,14)
	\path(40,8)(41,10)(42,6)(43,10)(44,6)(45,10)(46,6)(47,10)(48,6)(49,10)(50,8) 
	\put(43,-28){\large\shortstack{$R_N$}} 

\end{picture}
\vspace{5.25cm}\caption{\label{FigResist}A diagram of the circuit $E$ for a typical $\text{Poss}(x)$.  Each resistor at height $k$ from the root $V_0$ has resistance $\sim 2^k$.  The total resistance between $V_{k-1}$ and $V_k$ is denoted by $R_k$.}
\end{figure}
The resistance of $E$ cannot be greater than the resistance of the original circuit, by Proposition~\ref{resistance prop}.  Now fix $k$, $1\leq k\leq N$, and let $R_k$ denote the resistance between $V_{k-1}$ and $V_k$.  The number of edges between $V_{k-1}$ and $V_k$ is equal to the number $N_k$ of $k$th generation vertices in $\mathcal T_N(\text{Poss}(x))$. Recalling the containment \eqref{containment of Poss(p)} from Lemma \ref{description of Poss(p) lemma}, we find that $N_k$ is bounded above by $\overline{N}_k$, the number of $k$th level vertices in  $\mathcal T_N(\{0\} \times [0,1)^d \cap (x-x_1 \Omega_N))$. By Lemma \ref{lemma Omega_N vertex count}(\ref{affine Omega_N}), $\overline{N}_k \lesssim 2^{k}$, where the implicit constant is uniform in $x \in [C_0, C_0+1] \times [-2C_0, 2C_0]^d$.  Thus,
\begin{equation}
	\frac 1{R_k} = \sum_{1}^{N_k} \frac{1}{2^{k-1}} = \frac{N_k}{2^{k-1}}\lesssim \frac {\overline{N}_k}{2^k} \lesssim 1,\nonumber
\end{equation}
and this holds for any $1\leq k\leq N$.  Since the resistors $\{R_k\}_{k=1}^N$ are in series, $R(\text{Poss}(x)) \geq R(E) = \sum_{k=1}^{N} R_k \gtrsim N$, establishing inequality \eqref{ResistBound}.
\end{proof}

Combining Lemmas~\ref{PossLemma} and~\ref{Resistance} with Proposition~\ref{survival prob} gives us the desired bound of $\lesssim 1/N$ on \eqref{E: Probability}.  This completes the proof of inequality \eqref{generic upper bound}, and so too Proposition~\ref{Kakeya on average}.

}

\noindent \author{\textsc{Edward Kroc}}\\
University of British Columbia, Vancouver, Canada. \\
Electronic address: \texttt{ekroc@math.ubc.ca}
\vskip0.2in 
\noindent \author{\textsc{Malabika Pramanik}}\\
University of British Columbia, Vancouver, Canada. \\
Electronic address: \texttt{malabika@math.ubc.ca}

%%%%%%%%%%%%%%%%%%%%%%%%%%%%%%%%%%%%%%%%%%%%%%%%%%%%%%%%%%%%%%%%%%%%%%%

\end{document}